\documentclass[12pt]{amsart}

\usepackage{tikz}
\usepackage{graphicx}

\usepackage{indentfirst}
\usepackage{url}
\usepackage{amsmath}
\usepackage{amsthm}
\usepackage{amssymb}
\usepackage{amsfonts}

\usepackage[colorlinks=true, allcolors=black]{hyperref}

\usepackage[margin=1in]{geometry}

\theoremstyle{plain}
\newtheorem{lem}{Lemma}[section]
\newtheorem{cor}[lem]{Corollary}

\newtheorem{thm}[lem]{Theorem}
\newtheorem{conj}[lem]{Conjecture}

\theoremstyle{definition}
\newtheorem{defn}[lem]{Definition}

\theoremstyle{remark}
\newtheorem{rmk}[lem]{Remark}

\newcommand{\bvec}[1]{\mathbf{#1}}
\newcommand{\RR}[0]{\mathbb{R}}

\newcommand{\ZZ}[0]{\mathbb{Z}}
\newcommand{\CC}[0]{\mathbb{C}}

\newcommand{\PP}[0]{\mathcal{P}}

\newcommand{\OO}[0]{\mathcal{O}}
\newcommand{\lbar}[1]{\overline{#1}}
\newcommand{\kronsym}[2]{\left( \dfrac{{#1}}{{#2}}\right)}
\newcommand{\gausssum}[2]{G \! \left( \dfrac{{#1}}{{#2}}\right)}

\newcommand{\gausssumr}[3]{G_{#3} ({#1}, {#2})}

\newcommand{\PSL}[0]{\operatorname{PSL}}

\newcommand{\diag}[0]{\operatorname{diag}}
\newcommand{\supp}[0]{\operatorname{supp}}
\newcommand{\sgnt}[0]{\operatorname{sgnt}}
\newcommand{\sinc}[0]{\operatorname{sinc}}
\newcommand{\dist}[0]{\operatorname{dist}}
\newcommand{\Ind}[1]{\mathbf{1}_{\left\{{#1}\right\}}}
\newcommand{\e}[1]{\mathrm{e}\!\left({#1}\right)}

\newcommand{\Mat}[0]{M}
\newcommand{\repnum}[0]{R}
\newcommand{\numvars}[0]{s}
\newcommand{\numrepresented}[0]{n}
\newcommand{\Upspsi}[0]{\psi}
\newcommand{\Upsuppvar}[0]{\rho}
\newcommand{\lvl}[0]{L}
\newcommand{\coordsize}[0]{X}
\newcommand{\Vol}[0]{\operatorname{Vol}}
\newcommand{\modmultinv}[1]{{#1}^*}
\newcommand{\dimnum}[0]{d}
\newcommand{\numposevals}[0]{\nu}

\renewcommand{\pmod}[1]{\:\left(\operatorname{mod} {#1}\right)}
\renewcommand{\Re}[0]{\operatorname{Re}}

\newcommand{\numberthis}{\stepcounter{equation}\tag{\theequation}}

\numberwithin{equation}{section}

\title[Kloosterman circle method and weighted representation numbers]{The Kloosterman circle method and weighted representation numbers of quadratic forms}
\author{Edna Jones}
\address{Edna Jones: Department of Mathematics,
Tulane University,
6823 St.\ Charles Avenue,
New Orleans, LA 70118, USA} \email{ejones27@tulane.edu}


\begin{document}
\maketitle

\begin{abstract}
\indent
We develop a version of the Kloosterman circle method with a bump function that is used to provide asymptotics for weighted representation numbers of nonsingular integral quadratic forms. 
Unlike many applications of the Kloosterman circle method, we explicitly state some constants in the error terms that depend on the quadratic form. 
This version of the Kloosterman circle method uses Gauss sums, Kloosterman sums, Sali\'{e} sums, and a principle of nonstationary phase. 
We briefly discuss a potential application of this version of the Kloosterman circle method to a proof of a strong asymptotic local-global principle for certain Kleinian sphere packings.
\end{abstract}

\tableofcontents


\section{Introduction}

Kloosterman \cite{KloostermanRepQuaternaryQuadForms} developed what is now called the Kloosterman circle method to prove an asymptotic formula for the number of representations of an integer by a positive definite diagonal integral quaternary quadratic form. 
He used ideas from the Hardy--Littlewood circle method and adapted them to handle quaternary quadratic forms. 
For an overview of the Hardy--Littlewood circle method, see \cite{VaughanH-LMethod}.
The Kloosterman circle method is described Section~11.4 in \cite{IwaniecAutomorphicForms} and in Sections~20.3 and 20.4 in \cite{IwaniecKowalskiAnalyticNumThy}. 

Before developing the Kloosterman circle method, Kloosterman~\cite{KloostermanThesis} already had used the Hardy--Littlewood circle method to determine an asymptotic for the number of representations of an integer by a positive definite diagonal integral quadratic form in $\numvars \ge 5$ variables. 
However, as explained by Kloosterman in Section~1 of \cite{KloostermanRepQuaternaryQuadForms}, the error term obtained from the Hardy--Littlewood circle method is too large to provide asymptotic formulas for positive definite quaternary quadratic forms. 

We will use the Kloosterman circle method to obtain asymptotics for a weighted number of representations of an integer by a nonsingular integral quadratic form in $\numvars \ge 4$ variables. 
In our version of the Kloosterman circle method, we will use a bump function to ensure the convergence of our generating function for our weighted representation numbers. (In \cite{IwaniecAutomorphicForms} and \cite{IwaniecKowalskiAnalyticNumThy}, the generating function has an argument in the upper half plane. This ensures that their generating function converges.)
Our bump function provides greater flexibility for our version of the Kloosterman circle method than the version found in \cite{IwaniecAutomorphicForms} and \cite{IwaniecKowalskiAnalyticNumThy}. 

Because we are discussing asymptotics and error terms, it will be useful to define big $O$ notation and related terminology. 
If $f$ and $g$ are both functions of $x$, then the notation $f(x) = O(g(x))$ means that there exists a constant $C > 0$ such that $|f(x)| \le C g(x)$ for all $x \in D$, where $D$ is an appropriate domain that can be deduced from the context. 
The constant $C$ is called the \emph{implied constant}.
We take $f \ll g$ to have the same meaning as $f = O(g)$.
If the implied constant depends on a parameter $\alpha$, then we write $f = O_\alpha (g)$ or $f \ll_\alpha g$.

Heath-Brown~\cite{H-BCircleMethod} uses the delta method to obtain an asymptotic for a weighted number of representations of an integer by a quadratic form. However, in \cite{H-BCircleMethod}, it can be difficult to determine how the implied constants depend on the quadratic form.

In the error terms of our main result, we explicitly state some constants dependent on a quadratic form in $\numvars$ variables. This is unlike what is done in \cite{IwaniecAutomorphicForms}, \cite{IwaniecKowalskiAnalyticNumThy}, and \cite{H-BCircleMethod}. The implied constants in our main result only depend on the number~$\numvars$, some positive number $\varepsilon$, and the bump function involved. To obtain our result, we do not use the delta method. (This is unlike what was done in \cite{DietmannSmallSolnsQuadEqns} or \cite{BrowningDietmannRepByQuadForms}.)

In many applications of the circle method to the problem of representing integers by forms, one does not need to know any of the implied constants that depend on the form. 
However, there are instances in which it is useful to know some of these implied constants. 
For instance, Dietmann~\cite{DietmannSmallSolnsQuadEqns} needed to know some constants dependent on a quadratic form in order to provide a search bound for the smallest solution to a quadratic polynomial in $\numvars$ variables with integer coefficients. In order to do this, Dietmann provided an asymptotic for a certain weighted representation number in which the implied constants only depended on $\numvars$ and some $\varepsilon > 0$. (See Theorem~2 of \cite{DietmannSmallSolnsQuadEqns}.) However, the asymptotic Dietmann gives in Theorem~2 of \cite{DietmannSmallSolnsQuadEqns} does not say anything of use for positive definite quaternary quadratic forms. 
Browning and Dietmann~\cite{BrowningDietmannRepByQuadForms} improved this asymptotic formula for a particular chosen bump function. (See Proposition~1 in \cite{BrowningDietmannRepByQuadForms}.) To contrast, our result applies to any bump function.

A potential application of our version of the Kloosterman circle method is proving a strong asymptotic local-global principle for certain Kleinian sphere packings. (See Section~\ref{chapter:Local-GlobalSpherePackings} for more information on a strong asymptotic local-global principle for certain Kleinian sphere packings.)
Because multiple positive definite integral quaternary quadratic forms may be involved at the same time in this potential application, we would like to know some of the constants dependent on the quadratic forms involved. 

Before we state our main result, some notation in the result should be mentioned. 
Let $\Mat_\numvars (R)$ denote the set of $\numvars \times \numvars$ matrices over a ring $R$.
The determinant of a square matrix $A$ is denoted by $\det(A)$. 

The ring of integers modulo $m$ is denoted by $\ZZ/m\ZZ$. The multiplicative group of integers modulo $m$ is denoted by ${(\ZZ/m\ZZ)}^\times$. 
If $d \in {(\ZZ/m\ZZ)}^\times$, then we denote the multiplicative inverse of $d$ modulo $m$ by $\modmultinv{d}$. 

The following notation is standard in analytic number theory. 
Let $\e{x} = e^{2 \pi i x}$. 
In a product over $p$, the index of multiplication $p$ is taken to be prime. In general, $p$ is taken to be prime. 
The gamma function is
\begin{align*}
\Gamma(z) &= \int_{0}^{\infty} t^{z - 1} e^{-t} \ d t 
\end{align*}
for $z \in \CC$ with $\Re(z) > 0$. The gamma function can be analytically continued to a meromorphic function on $\CC$ so that $1/\Gamma(z)$ is an entire function. 
Note that $1/\Gamma(0) = 0$, because $\Gamma(z)$ has a simple pole at $z = 0$.

For a statement $Y$, let $\Ind{Y}$ be the indicator function
\begin{align*}
\Ind{Y} = 
\begin{cases}
1	&\text{if $Y$ holds,} \\
0	&\text{otherwise.}
\end{cases}
\end{align*}

For a positive integer $\numrepresented$, the divisor function $\tau(\numrepresented)$ is the number of positive divisors of $\numrepresented$. It is well-known that for all $\varepsilon > 0$, we have $\tau(\numrepresented) \ll_\varepsilon \numrepresented^\varepsilon$. (See, for example, (2.20) in \cite{MontgomeryVaughanMultNT}.)

The space of real-valued and infinitely differentiable functions on $\RR^\numvars$ is denoted by $C^\infty (\RR^\numvars)$. We call a function $f$ in $C^\infty (\RR^\numvars)$ a \emph{smooth} function. 
For a continuous function $f$, define the \emph{support} of $f$ (denoted by $\supp(f)$) to be the closure of the set $\{ \bvec{x} \in \RR^\numvars : f(\bvec{x}) \ne 0 \}$. A function $f$ is said to be \emph{compactly supported} if $\supp(f)$ is a compact set. 
The space of real-valued, infinitely differentiable, and compactly supported functions on $\RR^\numvars$ is denoted by $C_c^\infty (\RR^\numvars)$. We call a function $f$ in $C_c^\infty (\RR^\numvars)$ a \emph{bump} function (or a \emph{test} function). 

Let $\Upspsi \in C_c^\infty (\RR^\numvars)$ be a bump function. Since $\Upspsi$ is compactly supported, there exists a nonnegative real number $\Upsuppvar$ such that 
\begin{align}
\supp(\Upspsi) \subseteq \{ \bvec{x} \in \RR^\numvars : \| \bvec{x} \| \le \Upsuppvar \} . \label{subset:Upsupp def}
\end{align}
Define $\Upsuppvar_\Upspsi$ to be the smallest nonnegative $\Upsuppvar$ that satisfies \eqref{subset:Upsupp def}.

For $\Upspsi \in C_c^\infty (\RR^\numvars)$, $\coordsize > 0$, and $\bvec{m} \in \RR^\numvars$, let $\Upspsi_\coordsize$ be defined by
\begin{align}
\Upspsi_\coordsize (\bvec{m}) = \Upspsi\left( \frac{1}{\coordsize} \bvec{m} \right) . \label{eq:Upsilon\coordsize}
\end{align}
Notice that $\supp(\Upspsi_\coordsize) \subseteq \{ \bvec{m} \in \RR^\numvars : \| \bvec{m} \| \le \Upsuppvar_\Upspsi \coordsize \}$ and $\Upsuppvar_{\Upspsi_\coordsize} = \Upsuppvar_\Upspsi \coordsize$. 
Also, $\Upspsi_1 = \Upspsi$.

For a positive real number $\coordsize$, an integer $\numrepresented$, and an integral quadratic form $F$, let $\repnum_{F, \Upspsi, \coordsize} (\numrepresented)$ be the weighted representation number defined by 
\begin{align}
\repnum_{F, \Upspsi, \coordsize} (\numrepresented) = \sum_{\bvec{m} \in \ZZ^\numvars} \Ind{F(\bvec{m}) = \numrepresented} \Upspsi_\coordsize (\bvec{m}) . \label{eq:R\coordsize\numrepresented}
\end{align}
The main results of this paper concern asymptotics for this weighted representation number.

We now define some quantities that appear in our main results. 
For a positive integer $\numrepresented$ and an integral quadratic form $F$ in $\numvars$ variables, define $\mathfrak{S}_F (\numrepresented)$ to be the \emph{singular series} 
\begin{align}
\mathfrak{S}_F (\numrepresented) &= \sum_{q=1}^{\infty} \frac{1}{q^\numvars} \sum_{d \in (\ZZ/q\ZZ)^\times} \sum_{\bvec{h} \in (\ZZ/q\ZZ)^\numvars} \e{\frac{d}{q} \left( F(\bvec{h}) - \numrepresented \right)} . \label{eq:singseries}
\end{align}
The singular series will appear in the main term for our asymptotic for the weighted representation number $\repnum_{F, \Upspsi, \coordsize} (\numrepresented)$. The singular series is known to contain information modulo $q$ for every positive integer $q$. See Section~11.5 in \cite{IwaniecAutomorphicForms} for more information about the singular series. 

For a nonsingular quadratic form $F$, a real number $\numrepresented$, and a positive real number~$\coordsize$, 
we define the \emph{real factor} $\sigma_{F, \Upspsi, \infty} (\numrepresented, \coordsize)$ to be 
\begin{align}
\sigma_{F, \Upspsi, \infty} (\numrepresented, \coordsize) &= \lim_{\varepsilon \to 0^+} \frac{1}{2 \varepsilon} \int_{\left| F(\bvec{m}) - \frac{\numrepresented}{\coordsize^2} \right| < \varepsilon} \Upspsi(\bvec{m}) \ d \bvec{m} . \label{eq:archmediandensity vol}
\end{align}
This real factor can be viewed as a weighted density of real solutions to $F(\bvec{m}) = \numrepresented / \coordsize^2$.

The main result of this paper is the following theorem about the weighted representation number $\repnum_{F, \Upspsi, \coordsize} (\numrepresented)$. 
\begin{thm} \label{thm:R\coordsize\numrepresented asymp1}
Suppose that $\numrepresented$ is a positive integer.
Suppose that $F$ is a nonsingular integral quadratic form in $\numvars \ge 4$ variables. Let $A \in \Mat_\numvars (\ZZ)$ be the Hessian matrix of $F$. Let $\sigma_1$ be largest singular value of $A$, and let $\numposevals$ be the number of positive eigenvalues of $A$. Let $\lvl$ be the smallest positive integer such that $\lvl A^{-1} \in \Mat_\numvars (\ZZ)$. 
Suppose that $\Upspsi \in C_c^\infty (\RR^\numvars)$ is a bump function. 
Then for $\coordsize \ge 1 / \sigma_1$ 
and $\varepsilon > 0$, the weighted representation number $\repnum_{F, \Upspsi, \coordsize} (\numrepresented)$ is 
\begin{align*}
&\repnum_{F, \Upspsi, \coordsize} (\numrepresented) \\
		&= \mathfrak{S}_F (\numrepresented) \sigma_{F, \Upspsi, \infty} (\numrepresented, \coordsize) \coordsize^{\numvars - 2} \\
		&\qquad + O_{\Upspsi , \numvars , \varepsilon} \left( \frac{\lvl^{\numvars/2} \coordsize^{(\numvars - 1)/2 + \varepsilon} \sigma_1^{(3 - \numvars )/2 + \varepsilon}}{\Gamma(\numposevals/2) \left( \prod_{j=1}^\numposevals \lambda_j \right)^{1/2}} \left(\frac{\numrepresented}{\coordsize^2} - \frac{{\Upsuppvar_\Upspsi}^2}{2} \Ind{\numposevals > 1} \sum_{j=\numposevals+1}^\numvars \lambda_j \right)^{\numposevals/2 - 1} \right. \\
				&\qquad\qquad\qquad\qquad \left. \times \tau(\numrepresented) \prod_{p \mid 2 \det(A)} ( 1 - p^{-1/2} )^{-1} \right) \\
		&\qquad + O_{\Upspsi , \numvars , \varepsilon} \left( \coordsize^{(\numvars - 1) /2 + \varepsilon} \sigma_1^{(\numvars + 1)/2 + \varepsilon} \lvl^{\numvars/2} \tau(\numrepresented) \prod_{p \mid 2 \det(A)} ( 1 - p^{-1/2} )^{-1} \right) , \numberthis \label{eq:R\coordsize\numrepresented asymp1}
\end{align*}
where $\lambda_1, \lambda_2, \ldots, \lambda_\numposevals$ are the positive eigenvalues of $A$ and $\lambda_{\numposevals+1}, \lambda_{\numposevals+2}, \ldots, \lambda_\numvars$ are the negative eigenvalues of $A$.
\end{thm}
\begin{rmk}
If $\numposevals$ mentioned in Theorem~\ref{thm:R\coordsize\numrepresented asymp1} is zero, then the first error term in \eqref{eq:R\coordsize\numrepresented asymp1} is zero since $1/\Gamma(0)$.
\end{rmk}
\begin{rmk}
The integer $\lvl$ mentioned in Theorem~\ref{thm:R\coordsize\numrepresented asymp1} exists since $\det(A) A^{-1} \in \Mat_\numvars (\ZZ)$.
\end{rmk}
\begin{rmk}
If $n$ is a negative integer (instead of a positive integer), Theorem~\ref{thm:R\coordsize\numrepresented asymp1} can be applied to $\repnum_{-F, \Upspsi, \coordsize} (-\numrepresented)$. Because $\repnum_{F, \Upspsi, \coordsize} (\numrepresented) = \repnum_{-F, \Upspsi, \coordsize} (-\numrepresented)$, this would give an asymptotic formula for $\repnum_{F, \Upspsi, \coordsize} (\numrepresented)$ as $n \to -\infty$.
\end{rmk}

When we set $\coordsize = \numrepresented^{1/2}$ in Theorem~\ref{thm:R\coordsize\numrepresented asymp1}, 
we obtain the following asymptotic for the weighted representation number $\repnum_{F, \Upspsi, \coordsize} (\numrepresented)$.
\begin{cor} \label{cor:R\coordsize\numrepresented asymp2}
Suppose that $F$ is a nonsingular integral quadratic form in $\numvars \ge 4$ variables. Let $A \in \Mat_\numvars (\ZZ)$ be the Hessian matrix of $F$. Let $\sigma_1$ be largest singular value of $A$, and let $\numposevals$ be the number of positive eigenvalues of $A$. Let $\lvl$ be the smallest positive integer such that $\lvl A^{-1} \in \Mat_\numvars (\ZZ)$.
If $\numrepresented$ is a positive integer and $\varepsilon > 0$,
then
the weighted representation number $\repnum_{F, \Upspsi, \coordsize} (\numrepresented)$ is 
\begin{align*}
\repnum_{F, \Upspsi, \coordsize} (\numrepresented) &= \mathfrak{S}_F \!\left( \numrepresented \right) \sigma_{F, \Upspsi, \infty} \!\left( \numrepresented, \numrepresented^{1/2} \right) \numrepresented^{\numvars/2 - 1} \\
		&\quad + O_{\Upspsi , \numvars , \varepsilon} \left( \left( \sigma_1^{(\numvars + 1)/2 + \varepsilon} + \frac{\sigma_1^{(3 - \numvars )/2 + \varepsilon}}{\Gamma(\numposevals/2) \left( \prod_{j=1}^\numposevals \lambda_j \right)^{1/2}} \left( 1 - \frac{{\Upsuppvar_\Upspsi}^2}{2} \Ind{\numposevals > 1} \sum_{j=\numposevals+1}^\numvars \lambda_j \right)^{\numposevals/2 - 1} \right) \right. \\
		&\qquad\qquad\qquad\qquad \left. \times \numrepresented^{(\numvars - 1) /4 + \varepsilon/2} \tau(\numrepresented) \lvl^{\numvars/2} \prod_{p \mid 2 \det(A)} ( 1 - p^{-1/2} )^{-1}  \right) ,
\end{align*}
where $\lambda_1, \lambda_2, \ldots, \lambda_\numposevals$ are the positive eigenvalues of $A$ and $\lambda_{\numposevals+1}, \lambda_{\numposevals+2}, \ldots, \lambda_\numvars$ are the negative eigenvalues of $A$.
\end{cor}

By fixing a particular bump function $\Upspsi$ in Theorem~\ref{thm:R\coordsize\numrepresented asymp1} or in Corollary~\ref{cor:R\coordsize\numrepresented asymp2}, we obtain the following corollary about the (unweighted) representation number of the positive definite integral quadratic form $F$. (The representation number of a positive definite integral form $F$ is the number of integral solutions to $F(\bvec{m}) = \numrepresented$.) This corollary is Theorem~11.2 in \cite{IwaniecAutomorphicForms} and Theorem~20.9 in \cite{IwaniecKowalskiAnalyticNumThy}.
\begin{cor} \label{cor:repnum F asymp}
Suppose that $\numrepresented$ is a positive integer. 
Suppose that $F$ is a positive definite integral quadratic form in $\numvars \ge 4$ variables. Let $A \in \Mat_\numvars (\ZZ)$ be the Hessian matrix of $F$. 
Then the number of integral solutions to $F(\bvec{m}) = \numrepresented$ is 
\begin{align*}
| \{ \bvec{m} \in \ZZ^\numvars : F(\bvec{m}) = \numrepresented \} | &= \mathfrak{S}_F (\numrepresented) \frac{(2 \pi)^{\numvars/2}}{\Gamma(\numvars/2) \sqrt{\det(A)}} \numrepresented^{\numvars/2 - 1} + O_{F , \varepsilon} \left( \numrepresented^{(\numvars - 1)/4 + \varepsilon} \right) 
\end{align*}
for any $\varepsilon > 0$.
\end{cor}
\begin{rmk}
The implied constant in Corollary~\ref{cor:repnum F asymp} depends on the quadratic form $F$, because the choice of $\Upspsi$ depends on $F$. 
\end{rmk}
If a bump function $\Upspsi$ is chosen first and then a positive definite quadratic form $F$ is chosen based on $\Upspsi$ (and possibly on $\numrepresented , \coordsize> 0$), we can have a result in which all of the implied constants only depend on $\Upspsi$, $\numvars$, and $\varepsilon$. We hope that such a result can be used towards a proof of a strong asymptotic local-global principle for certain Kleinian sphere packings. 

The remainder of this paper is organized as follows.
In Section~\ref{chapter:notation}, we define some notation that will be used throughout the paper. 
In Section~\ref{chapter:Kloostermancirclemethodsetup}, we set up the Kloosterman circle method and apply it to our particular problem. 
Once we apply the Kloosterman circle method to our problem, we obtain an arithmetic part and an archimedean part.
We analyze the arithmetic part in Section~\ref{chapter:arithmeticpart} and the archimedean part in Section~\ref{chapter:archimedeanpart}. 
In Section~\ref{chapter:puttogether}, we put together estimates from previous sections and complete our proofs of Theorem~\ref{thm:R\coordsize\numrepresented asymp1} and Corollaries~\ref{cor:R\coordsize\numrepresented asymp2} and \ref{cor:repnum F asymp}.
In Section~\ref{chapter:Local-GlobalSpherePackings}, we briefly discuss a potential application of our version of the Kloosterman circle method: a proof of a strong asymptotic local-global principle for certain Kleinian sphere packings.

\section{Some additional notation} \label{chapter:notation}

In this section, we state some notation used throughout this paper. This is not a comprehensive list of notation used in this paper, but most of the notation used in this paper is listed here, in the previous section, or in the next section.

The greatest common divisor of integers $a_1 , \ldots , a_m$ is denoted by $\gcd(a_1 , \ldots , a_m)$.

A vector $\bvec{m} \in \ZZ^\numvars$ is viewed as an $\numvars \times 1$ column vector. The $j$th entry of a vector $\bvec{m}$ is denoted by $m_j$. The entry in the $j$th row and the $k$th column of a matrix $A$ is denoted by $a_{j k}$. For a vector $\bvec{m}$ (or matrix $A$), let $\bvec{m}^\top$ (or $A^\top$) denote the transpose of $\bvec{m}$ (or $A$). A diagonal $\numvars \times \numvars$ matrix with diagonal entries $d_1, d_2, \ldots, d_\numvars$ is denoted by $\diag(d_1, d_2, \ldots, d_\numvars)$.

The \emph{dot product} of $\bvec{x}, \bvec{y} \in \RR^\numvars$ is $\bvec{x}^\top \bvec{y} = \sum_{j=1}^\numvars x_j y_j$ and is denoted by $\bvec{x} \cdot \bvec{y}$. The \emph{Euclidean norm} of a vector $\bvec{x}$ is $\| \bvec{x} \| = \sqrt{\bvec{x} \cdot \bvec{x}}$. 
For $\bvec{x} \in \RR^\numvars$ and a nonempty subset $U$ of $\RR^\numvars$, define the \emph{distance} between $\bvec{x}$ and $U$ to be
\begin{align*}
\dist(\bvec{x}, U) &= \inf_{\bvec{m} \in U} \| \bvec{x} - \bvec{m} \| .
\end{align*}

We use $\bvec{0}$ to denote the zero vector in $\RR^\numvars$.

For a nonsingular symmetric matrix $A \in \Mat_\numvars (\RR)$, we define the \emph{signature} of $A$ to be the number of positive eigenvalues of $A$ minus the number of negative eigenvalues of $A$. The signature of $A$ is denoted by $\sgnt(A)$.

As is standard in analytic number theory, the natural logarithm of $z$ is denoted by $\log(z)$, not by $\ln(z)$. 

For $z \in \CC$, the notation $\lbar{z}$ means the complex conjugate of $z$. The real part of a complex number $z$ is denoted by $\Re(z)$.

For a nonnegative integer $k$ and a function $f \colon \RR \to \RR$ with a $k$th derivative, we let $f^{(k)}$ denote the $k$th derivative of $f$. Note that $f^{(0)} = f$, $f^{(1)} = f'$, and $f^{(2)} = f''$.

We use the notation $\lfloor x \rfloor$ for the greatest integer less than or equal to $x$. Similarly, $\lceil x \rceil$ denotes the least integer greater than or equal to $x$.

For a sufficiently nice function $f \colon \RR^d \to \RR$ 
(where $d$ is a positive integer), define the \emph{Fourier transform} $\widehat{f}$ of $f$ to be 
\begin{align}
\widehat{f}(\bvec{y}) &= \int_{\RR^d} f(\bvec{x}) \e{- \bvec{x} \cdot \bvec{y}} \ d \bvec{x} 
\end{align}
for $\bvec{y} \in \RR^d$. 
Assuming that $f$ is sufficiently nice, the inverse Fourier transform of $\widehat{f}$ is $f$. That is,
\begin{align}
f(\bvec{y}) &= \int_{\RR^d} \widehat{f}(\bvec{x}) \e{\bvec{x} \cdot \bvec{y}} \ d \bvec{x} .
\end{align}
Whenever we use the Fourier transform or the inverse Fourier transform of a function, we assume that the function is sufficiently nice. 
There are various ways that ``sufficiently nice'' can be made precise. For example, a function is sufficiently nice if it is a Schwartz function. (See Corollary~8.23 of \cite{FollandRealAnalysis}.)

\section{Setting up the Kloosterman circle method} \label{chapter:Kloostermancirclemethodsetup}

In this section, we set up the Kloosterman circle method for our weighted representation number $\repnum_{F, \Upspsi, \coordsize} (\numrepresented)$. Unless otherwise specified, any notation mentioned here will be used for the remainder of this paper.

As in the statement of Theorem~\ref{thm:R\coordsize\numrepresented asymp1}, we let $F$ be a nonsingular integral quadratic form in $\numvars \ge 4$ variables. Let $A \in \Mat_\numvars (\ZZ)$ be the Hessian matrix of $F$, and let $\numposevals$ be the the number of positive eigenvalues of $A$. Observe that $A$ is symmetric with each of its diagonal entries in $2\ZZ$. Also, $F(\bvec{x}) = \frac{1}{2} \bvec{x}^\top A \bvec{x}$ for any $\bvec{x} \in \RR^\numvars$. 
Let $\{ \sigma_j \}_{j=1}^\numvars$ be the set of singular values of $A$, where $\sigma_1 \ge \cdots \ge \sigma_\numvars > 0$. 

Let $F^*$ be the adjoint quadratic form 
\begin{align*}
F^*(\bvec{x}) = \frac{1}{2} \bvec{x}^\top A^{-1} \bvec{x} 
\end{align*}
for any $\bvec{x} \in \RR^\numvars$. 
We note that $A^{-1}$ might not be an integral matrix. Let $\lvl$ be the smallest positive integer such that $\lvl A^{-1} \in \Mat_\numvars (\ZZ)$. 

Let $\Theta_{F, \Upspsi, \coordsize} (x)$ be the real analytic function with $\repnum_{F, \Upspsi, \coordsize} (\numrepresented)$ as the Fourier coefficients; i.e., let 
\begin{align}
\Theta_{F, \Upspsi, \coordsize} (x) = \sum_{n = -\infty}^\infty \repnum_{F, \Upspsi, \coordsize} (\numrepresented) \e{\numrepresented x} \label{eq:theta\coordsize\numrepresented}
\end{align}
for $x \in \RR$.
Notice that $\Theta_{F, \Upspsi, \coordsize} (x+1) = \Theta_{F, \Upspsi, \coordsize} (x)$. 
We call $\Theta_{F, \Upspsi, \coordsize}$ a \emph{weighted theta series} of $F$.

Using the Fourier transform, we see that 
\begin{align}
\repnum_{F, \Upspsi, \coordsize} (\numrepresented) = \int_0^1 \Theta_{F, \Upspsi, \coordsize} (x) \e{-nx} \ d x . \label{eq:R\coordsize (\numrepresented) inv Fourier}
\end{align}

The fact that the $\numrepresented$th Fourier coefficient of $\Theta_{F, \Upspsi, \coordsize} (x)$ is $\repnum_{F, \Upspsi, \coordsize} (\numrepresented)$ suggests that the function $\Theta_{F, \Upspsi, \coordsize}$ is related to $F$. The next lemma states how $\Theta_{F, \Upspsi, \coordsize}$ is directly related to $F$.

\begin{lem} \label{lem:theta\coordsize sum e(xF(m)) Upsilonprod}
For $x \in \RR$ and $\coordsize > 0$, the function $\Theta_{F, \Upspsi, \coordsize} (x)$ is 
\begin{align*}
\Theta_{F, \Upspsi, \coordsize} (x) &= \sum_{\bvec{m} \in \ZZ^\numvars} \e{x F(\bvec{m})} \Upspsi_\coordsize (\bvec{m}) .
\end{align*}
\end{lem}
Lemma~\ref{lem:theta\coordsize sum e(xF(m)) Upsilonprod} can be proved by showing that $\Theta_{F, \Upspsi, \coordsize} (x)$ and $\sum_{\bvec{m} \in \ZZ^\numvars} \e{x F(\bvec{m})} \Upspsi_\coordsize (\bvec{m})$ have the same Fourier coefficients.

\subsection{Using a Farey dissection} \label{sec:Fareydissect}

To use the Kloosterman circle method, we want to break up a unit interval (say $[ z , z+1 )$ for $z \in \CC$) into smaller intervals (or ``arcs'') using Farey sequences. This unit interval is where the ``circle'' in the ``circle method'' comes from. A unit interval is sometimes considered a circle since $\e{[ z , z+1 )}$ is a circle in the complex plane.

For $Q \ge 1$, the Farey sequence $\mathfrak{F}_Q$ of order $Q$ is the increasing sequence of all reduced fractions $\frac{a}{q}$ with $1 \le q \le Q$ and $\gcd(a, q) = 1$. (Traditionally, $Q$ is required to be an integer, but we will find that allowing $Q$ to be not an integer will remove some technicalities later. See Subsection~\ref{sec:parameters}.)
An element of $\mathfrak{F}_Q$ is called a \emph{Farey point}. 
We will state some well-known properties of $\mathfrak{F}_Q$. 
For more background on Farey sequences of order $Q$, please see Chapter~III of \cite{HardyWrightNT} or pp.~451--452 of \cite{IwaniecKowalskiAnalyticNumThy}.

Let $\frac{a'}{q'} < \frac{a}{q} < \frac{a''}{q''}$ be adjacent Farey points of order $Q$ (so that $q', q, q'' \le Q$ and $\frac{a}{q}$ is the only fraction in $\mathfrak{F}_Q$ that is greater than $\frac{a'}{q'}$ and less than $\frac{a''}{q''}$). 
The denominators $q', q''$ are determined by the conditions:
\begin{align}
&Q - q < q' \le Q, & a q' &\equiv 1 \pmod{q}, \label{q' cond}\\
&Q - q < q'' \le Q, & a q'' &\equiv -1 \pmod{q}. \label{q'' cond}
\end{align}

For a positive integer $q$ and an integer $a$ coprime to $q$, let $j_{a,q}$ be the interval
\begin{align}
j_{a,q} = \left[ \frac{a}{q} - \frac{1}{q(q+q')} , \frac{a}{q} + \frac{1}{q (q+q'')} \right) ,
\end{align}
where $\frac{a'}{q'} < \frac{a}{q} < \frac{a''}{q''}$ are adjacent Farey points of order $Q$.
The interval $j_{a,q}$ is called a \emph{Farey arc}. 
It is known that 
\begin{align}
\left[ -\frac{1}{1 + \lfloor Q \rfloor}, 1 - \frac{1}{1 + \lfloor Q \rfloor} \right) = \bigsqcup_{q=1}^{\lfloor Q \rfloor} \bigsqcup_{\substack{a = 0 \\ \gcd(a,q)=1}}^{q-1} j_{a,q} . \label{eq:Fareyfracdissectunion}
\end{align}
We say that the set 
\begin{align}
\{ j_{a,q} \}_{1 \le q \le Q , \; 0 \le a \le q-1 , \; \gcd(a, q) = 1}
\end{align}
of Farey arcs is a \emph{Farey dissection of order $Q$} of the circle. 

We now have all the tools to express an idea that is crucial in the the Kloosterman circle method.
\begin{lem} \label{lem:KloostDecomp}
Let $f \colon \RR \to \CC$ be a periodic function of period $1$ and with real Fourier coefficients (so that $\lbar{f(x)} = f(-x)$ for all $x \in \RR$). Then
\begin{align}
\int_0^1 f(x) \ d x 
	&= 2 \Re\left( \sum_{1 \le q \le Q} \int_{0}^{\frac{1}{q Q}} \sum_{\substack{Q < d \le q+Q \\ qdx < 1 \\ \gcd(d, q) = 1}} f \left( x - \frac{\modmultinv{d}}{q} \right) \ d x \right) \label{eq:KloostDecompSumIntSum} \\
	&= 2 \Re \left( \sum_{1 \le q \le Q} \sum_{\substack{Q < d \le q + Q \\ \gcd(d, q) = 1}} \int_{0}^{1} \frac{1}{qd} \ f\left( \frac{x}{qd} - \frac{\modmultinv{d}}{q} \right) \ d x \right) , \label{eq:KloostDecomp0-1} 
\end{align}
where 
$\modmultinv{d}$ is the multiplicative inverse of $d$ modulo $q$.
\end{lem}

\begin{rmk}
When analytic number theorists talk about using the Kloosterman circle method, they often mean that they are using a result similar to Lemma~\ref{lem:KloostDecomp}. One of the most notable applications of Lemma~\ref{lem:KloostDecomp} is writing the function $\Ind{\numrepresented = 0}$ as a sum of exponential functions by having $f(x) = \e{\numrepresented x}$ in \eqref{eq:KloostDecomp0-1}. 
(See Proposition~11.1 in \cite{IwaniecAutomorphicForms} and Proposition~20.7 in \cite{IwaniecKowalskiAnalyticNumThy}.)
\end{rmk}
\begin{rmk}
We note that Lemma~\ref{lem:KloostDecomp} is general enough to work with a variety of representation problems by having $f(x) = \e{-\numrepresented x} \sum_{m \in \ZZ} R(m) \e{m x}$ for a given representation number $R(\numrepresented)$. Notice that Lemma~\ref{lem:KloostDecomp} is applicable to the function $f(x) = \e{-nx} \Theta_{F, \Upspsi, \coordsize} (x)$ and so can give us an expression of $\repnum_{F, \Upspsi, \coordsize} (m)$ containing an integral with a Farey dissection.
\end{rmk}
\begin{proof}[Proof of Lemma~\ref{lem:KloostDecomp}]
Because $f$ has a period of $1$, we have the following equality of integrals: 
\begin{align*}
\int_0^1 f(x) \ d x &= \int_{-\frac{1}{1 + \lfloor Q \rfloor}}^{1 - \frac{1}{1 + \lfloor Q \rfloor}} f(x) \ d x .
\end{align*}
Using our Farey dissection of $\left[ -\frac{1}{1 + \lfloor Q \rfloor}, 1 - \frac{1}{1 + \lfloor Q \rfloor} \right)$, notice that 
\begin{align*}
\int_0^1 f(x) \ d x &= \sum_{1 \le q \le Q} \sum_{\substack{a = 0 \\ \gcd(a, q) = 1}}^{q-1} \int_{j_{a,q}} f(x) \ d x \\
	&= \sum_{1 \le q \le Q} \sum_{\substack{a = 0 \\ \gcd(a, q) = 1}}^{q-1} \int_{\frac{a}{q} - \frac{1}{q(q+q')}}^{\frac{a}{q} + \frac{1}{q (q+q'')}} f(x) \ d x \\
	&= \sum_{1 \le q \le Q} \sum_{\substack{a = 0 \\ \gcd(a, q) = 1}}^{q-1} \int_{- \frac{1}{q(q+q')}}^{\frac{1}{q (q+q'')}} f \left( x + \frac{a}{q} \right) \ d x 
\end{align*}
by letting $x \mapsto x + a/q$ in the last integral.

By breaking up and rearranging the integrals, we obtain
\begin{align}
\int_0^1 f(x) \ d x &= J_1 + J_2 , \label{eq:Fareysplit}
\end{align}
where
\begin{align*}
J_1 &= \sum_{1 \le q \le Q} \sum_{\substack{a = 0 \\ \gcd(a, q) = 1}}^{q-1} \int_{- \frac{1}{q(q+q')}}^{0} f \left( x + \frac{a}{q} \right) \ d x
\end{align*}
and
\begin{align*}
J_2 &= \sum_{1 \le q \le Q} \sum_{\substack{a = 0 \\ \gcd(a, q) = 1}}^{q-1} \int_{0}^{\frac{1}{q (q+q'')}} f \left( x + \frac{a}{q} \right) \ d x .
\end{align*}

In $J_1$, we let $d = q+q'$. By \eqref{q' cond}, the expression $d = q+q'$ is equivalent to the combination of conditions $Q < d \le q + Q$ and $a \equiv \modmultinv{d} \pmod{q}$. Thus, when we use the periodicity of $f$, we have 
\begin{align*}
J_1 &= \sum_{1 \le q \le Q} \sum_{\substack{Q < d \le q + Q \\ \gcd(d, q) = 1}} \int_{- \frac{1}{q d}}^{0} f \left( x + \frac{\modmultinv{d}}{q} \right) \ d x .
\end{align*}
We now map $x \mapsto -x$, so 
\begin{align*}
J_1 &= \sum_{1 \le q \le Q} \sum_{\substack{Q < d \le q + Q \\ \gcd(d, q) = 1}} \int_{\frac{1}{q d}}^{0} - f \left( - x + \frac{\modmultinv{d}}{q} \right) \ d x \\
	&= \sum_{1 \le q \le Q} \sum_{\substack{Q < d \le q + Q \\ \gcd(d, q) = 1}} \int_{0}^{\frac{1}{q d}} f \left( - \left(x - \frac{\modmultinv{d}}{q} \right) \right) \ d x .
\end{align*}
Since $\lbar{f(x)} = f(-x)$ for all $x \in \RR$, we conclude that  
\begin{align*}
J_1 
	&= \lbar{\sum_{1 \le q \le Q} \sum_{\substack{Q < d \le q + Q \\ \gcd(d, q) = 1}} \int_{0}^{\frac{1}{q d}} f \left( x - \frac{\modmultinv{d}}{q} \right) \ d x} . \numberthis \label{eq:Fareysplit q'}
\end{align*}

In $J_2$, we let $d = q+q''$. By \eqref{q'' cond}, the expression $d = q+q''$ is equivalent to the combination of conditions $Q < d \le q + Q$ and $a \equiv -\modmultinv{d} \pmod{q}$. Thus, when we use the periodicity of $f$, we have 
\begin{align}
J_2 &= \sum_{1 \le q \le Q} \sum_{\substack{Q < d \le q + Q \\ \gcd(d, q) = 1}} \int_{0}^{\frac{1}{q d}} f \left( x - \frac{\modmultinv{d}}{q} \right) \ d x . \label{eq:Fareysplit q''}
\end{align}

Substituting \eqref{eq:Fareysplit q'} and \eqref{eq:Fareysplit q''} into \eqref{eq:Fareysplit}, we obtain 
\begin{align*}
\int_0^1 f(x) \ d x &= \lbar{\sum_{1 \le q \le Q} \sum_{\substack{Q < d \le q + Q \\ \gcd(d, q) = 1}} \int_{0}^{\frac{1}{q d}} f \left( x - \frac{\modmultinv{d}}{q} \right) \ d x} \\
		&\qquad + \sum_{1 \le q \le Q} \sum_{\substack{Q < d \le q + Q \\ \gcd(d, q) = 1}} \int_{0}^{\frac{1}{q d}} f \left( x - \frac{\modmultinv{d}}{q} \right) \ d x \\
	&= 2 \Re\left( \sum_{1 \le q \le Q} \sum_{\substack{Q < d \le q + Q \\ \gcd(d, q) = 1}} \int_{0}^{\frac{1}{q d}} f \left( x - \frac{\modmultinv{d}}{q} \right) \ d x \right) . \numberthis \label{eq:KloostDecompSumSumInt} 
\end{align*}
By switching the order of summation and integration in \eqref{eq:KloostDecompSumSumInt}, we deduce \eqref{eq:KloostDecompSumIntSum}.
By mapping $x \mapsto \frac{x}{q d}$ in \eqref{eq:KloostDecompSumSumInt}, we obtain \eqref{eq:KloostDecomp0-1}. 
\end{proof}

We now apply Lemma~\ref{lem:KloostDecomp} to \eqref{eq:R\coordsize (\numrepresented) inv Fourier} to decompose the ``circle.'' Taking $f(x) = \Theta_{F, \Upspsi, \coordsize} (x) \e{-nx}$ in \eqref{eq:KloostDecompSumIntSum} of Lemma~\ref{lem:KloostDecomp}, we see that 
\begin{align*}
\repnum_{F, \Upspsi, \coordsize} (\numrepresented) 
	&= 2 \Re\left( \sum_{1 \le q \le Q} \int_{0}^{\frac{1}{q Q}} \e{-\numrepresented x} \sum_{\substack{Q < d \le q+Q \\ qdx < 1 \\ \gcd(d, q) = 1}} \Theta_{F, \Upspsi, \coordsize} \left( x - \frac{\modmultinv{d}}{q} \right) \e{\numrepresented \frac{\modmultinv{d}}{q}} \ d x \right) . \numberthis \label{eq:R(\numrepresented) KloostDecomp}
\end{align*}
The appearance of $\Theta_{F, \Upspsi, \coordsize} \left( x - \modmultinv{d}/q \right)$ leads us to try to evaluate $\Theta_{F, \Upspsi, \coordsize} \left( x - \modmultinv{d}/q \right)$ in the next subsection.

\subsection{Examining a weighted theta series}
Due to the appearance of $\Theta_{F, \Upspsi, \coordsize} \left( x - \modmultinv{d}/q \right)$ in our expression of $\repnum_{F, \Upspsi, \coordsize} (\numrepresented)$ in \eqref{eq:R(\numrepresented) KloostDecomp}, we would like to evaluate the weighted theta series $\Theta_{F, \Upspsi, \coordsize} \left( x - \modmultinv{d}/q \right)$.
By Lemma~\ref{lem:theta\coordsize sum e(xF(m)) Upsilonprod}, we see that 
\begin{align*}
\Theta_{F, \Upspsi, \coordsize} \left( x - \frac{\modmultinv{d}}{q} \right) 
	&= \sum_{\bvec{m} \in \ZZ^\numvars} \e{- \frac{\modmultinv{d}}{q} F(\bvec{m})} \e{x F(\bvec{m})} \Upspsi_\coordsize (\bvec{m}) .
\end{align*}

Because the value $\e{- \frac{\modmultinv{d}}{q} F(\bvec{m})}$ only depends on $\bvec{m}$ modulo $q$, we would like to split the sum over $\bvec{m} \in \ZZ^\numvars$ into sums over congruence classes modulo $q$. Doing this results in the following:
\begin{align*}
\Theta_{F, \Upspsi, \coordsize} \left( x - \frac{\modmultinv{d}}{q} \right) 
	&= \sum_{\bvec{h} \in (\ZZ/q\ZZ)^\numvars} \e{- \frac{\modmultinv{d}}{q} F(\bvec{h})} \sum_{\bvec{m} \equiv \bvec{h} \pmod{q}} \e{x F(\bvec{m})} \Upspsi_\coordsize (\bvec{m}) .
\end{align*}
By replacing $\bvec{m}$ by $\bvec{h} + q \bvec{m}$ (where $\bvec{m} \in \ZZ^\numvars$), we obtain
\begin{align}
\Theta_{F, \Upspsi, \coordsize} \left( x - \frac{\modmultinv{d}}{q} \right) &= \sum_{\bvec{h} \in (\ZZ/q\ZZ)^\numvars} \e{- \frac{\modmultinv{d}}{q} F(\bvec{h})} \sum_{\bvec{m} \in \ZZ^\numvars} \e{x F(\bvec{h}+q\bvec{m})} \Upspsi_\coordsize (\bvec{h}+q\bvec{m}) . \label{eq:theta\coordsize cong classes}
\end{align}

\subsection{Using Poisson summation}

To obtain asymptotics on $\Theta_{F, \Upspsi, \coordsize} \left( x - \modmultinv{d}/q \right)$, we must estimate 
\begin{align}
\sum_{\bvec{m} \in \ZZ^\numvars} \e{x F(\bvec{h}+q\bvec{m})} \Upspsi_\coordsize (\bvec{h}+q\bvec{m}). \label{eq:eF(h+qm)bumpsum}
\end{align}
However, if we could estimate \eqref{eq:eF(h+qm)bumpsum} directly, we probably would have estimated $\Theta_{F, \Upspsi, \coordsize} (x - \modmultinv{d}/q)$ directly in the first place. Since we did not, we need to use another tool. To understand \eqref{eq:eF(h+qm)bumpsum}, 
we use Poisson summation in the next lemma. Poisson summation allows us to concentrate the bulk of the sum into one term.

\begin{lem} \label{lem:eFbump Poisson}
For $\bvec{h} \in (\ZZ/q\ZZ)^\numvars$, $x \in \RR$, and positive $q \in \ZZ$, we have 
\begin{align*}
&\sum_{\bvec{m} \in \ZZ^\numvars} \e{x F(\bvec{h}+q\bvec{m})} \Upspsi_\coordsize (\bvec{h}+q\bvec{m}) \\
	&= \sum_{\bvec{r} \in \ZZ^\numvars} \frac{1}{q^\numvars} \e{\frac{1}{q} \bvec{h} \cdot \bvec{r}} \int_{\RR^\numvars} \e{x F(\bvec{m}) - \frac{1}{q} \bvec{m} \cdot \bvec{r}} \Upspsi_\coordsize (\bvec{m}) \ d \bvec{m} .
\end{align*}
\end{lem}
\begin{proof}
We use the Poisson summation formula
\begin{align}
\sum_{\bvec{m} \in \ZZ^\numvars} f(\bvec{m}) = \sum_{\bvec{r} \in \ZZ^\numvars} \widehat{f}(\bvec{r}) , \label{eq:Poisson summation formula}
\end{align}
where 
$\widehat{f}$ is the Fourier transform of $f$. 

The Fourier transform of $f(\bvec{m}) = \e{x F(\bvec{h}+q\bvec{m})} \Upspsi_\coordsize (\bvec{h}+q\bvec{m})$ is 
\begin{align*}
\widehat{f}(\bvec{r}) &= \int_{\RR^\numvars} \e{x F(\bvec{h}+q\bvec{m})} \Upspsi_\coordsize (\bvec{h}+q\bvec{m}) \e{- \bvec{m} \cdot \bvec{r}} \ d \bvec{m} .
\end{align*}
%
By mapping $\bvec{m}$ to $\frac{1}{q} (\bvec{m} - \bvec{h})$ and substituting into the Poisson summation formula~\eqref{eq:Poisson summation formula}, we obtain the result of this lemma.
\end{proof}

By applying Lemma~\ref{lem:eFbump Poisson} to \eqref{eq:theta\coordsize cong classes}, we obtain
\begin{align*}
\Theta_{F, \Upspsi, \coordsize} \left( x - \frac{\modmultinv{d}}{q} \right) &= \sum_{\bvec{h} \in (\ZZ/q\ZZ)^\numvars} \e{- \frac{\modmultinv{d}}{q} F(\bvec{h})} \sum_{\bvec{r} \in \ZZ^\numvars} \frac{1}{q^\numvars} \e{\frac{1}{q} \bvec{h} \cdot \bvec{r}} \\
		&\qquad\qquad\qquad \times \int_{\RR^\numvars} \e{x F(\bvec{m}) - \frac{1}{q} \bvec{m} \cdot \bvec{r}} \Upspsi_\coordsize (\bvec{m}) \ d \bvec{m} \\
	&= \sum_{\bvec{r} \in \ZZ^\numvars} \frac{1}{q^\numvars} \gausssumr{-\modmultinv{d}}{q}{\bvec{r}} \mathcal{I}_{F, \Upspsi} (x, \coordsize, \bvec{r}, q) , \numberthis \label{eq:theta\coordsize arithmetic-analytic sep}
\end{align*}
where $\gausssumr{d}{q}{\bvec{r}}$ is the Gauss sum
\begin{align}
\gausssumr{d}{q}{\bvec{r}} = \sum_{\bvec{h} \in (\ZZ/q\ZZ)^\numvars} \e{\frac{1}{q} (d F(\bvec{h}) + \bvec{h} \cdot \bvec{r})} \label{eq:gausssumr}
\end{align}
and
\begin{align}
\mathcal{I}_{F, \Upspsi} (x, \coordsize, \bvec{r}, q) &= \int_{\RR^\numvars} \e{x F(\bvec{m}) - \frac{1}{q} \bvec{m} \cdot \bvec{r}} \Upspsi_\coordsize (\bvec{m}) \ d \bvec{m} . \label{eq:IFx\coordsize rq}
\end{align}

Substituting \eqref{eq:theta\coordsize arithmetic-analytic sep} into \eqref{eq:R(\numrepresented) KloostDecomp}, we see that
\begin{align*}
&\repnum_{F, \Upspsi, \coordsize} (\numrepresented) \\
&= 2 \Re\left( \sum_{1 \le q \le Q} \int_{0}^{\frac{1}{q Q}} \e{-\numrepresented x} \sum_{\substack{Q < d \le q+Q \\ qdx < 1 \\ \gcd(d, q) = 1}} \sum_{\bvec{r} \in \ZZ^\numvars} \frac{1}{q^\numvars} \gausssumr{-\modmultinv{d}}{q}{\bvec{r}} \mathcal{I}_{F, \Upspsi} (x, \coordsize, \bvec{r}, q) \e{\numrepresented \frac{\modmultinv{d}}{q}} \ d x \right) . 
\end{align*}
Switching the order of the $\bvec{r}$ and $d$ sums, we obtain
\begin{align}
\repnum_{F, \Upspsi, \coordsize} (\numrepresented) &= 2 \Re\left( \sum_{1 \le q \le Q} \frac{1}{q^\numvars} \int_{0}^{\frac{1}{q Q}} \e{-\numrepresented x} \sum_{\bvec{r} \in \ZZ^\numvars} \mathcal{I}_{F, \Upspsi} (x, \coordsize, \bvec{r}, q) T_{\bvec{r}} (q, \numrepresented; x) \ d x \right) , \label{eq:R\coordsize \numrepresented with Trq\numrepresented x}
\end{align}
where
\begin{align}
T_{\bvec{r}} (q, \numrepresented; x) = \sum_{\substack{Q < d \le q+Q \\ qdx < 1 \\ \gcd(d, q) = 1}} \e{\numrepresented \frac{\modmultinv{d}}{q}} \gausssumr{-\modmultinv{d}}{q}{\bvec{r}} . \label{eq:Trq\numrepresented x with gausssumr}
\end{align}

We call the expression $T_{\bvec{r}} (q, \numrepresented; x)$ the \emph{arithmetic part}, because $T_{\bvec{r}} (q, \numrepresented; x)$ contains information modulo $q$. We call the integral $\mathcal{I}_{F, \Upspsi} (x, \coordsize, \bvec{r}, q)$ the \emph{archimedean part}. 
We will use Gauss sums, Kloosterman sums, and Sali\'{e} sums to obtain estimates on the arithmetic part $T_{\bvec{r}} (q, \numrepresented; x)$. (See Section~\ref{chapter:arithmeticpart}.)
We will use a principle of nonstationary phase to obtain estimates on the archimedean part $\mathcal{I}_{F, \Upspsi} (x, \coordsize, \bvec{r}, q)$. (See Section~\ref{chapter:archimedeanpart}.)

\section{Analyzing the arithmetic part} \label{chapter:arithmeticpart}

In this section, we analyze the arithmetic part $T_{\bvec{r}} (q, \numrepresented; x)$. We do this by completing the sum $T_{\bvec{r}} (q, \numrepresented; x)$, using Gauss sums, and obtaining estimates of our complete sums. An upper bound for the absolute value of the arithmetic part is given in the last subsection of this section.

\subsection{Completing the sum}

In general, finding the value of a sum of a periodic function over an arbitrary range of integer is more difficult than finding the value of a sum of a periodic function over its period. In our calculation of $\repnum_{F, \Upspsi, \coordsize} (\numrepresented)$, we currently have the sum $T_{\bvec{r}} (q, \numrepresented; x)$.
We notice that $T_{\bvec{r}} (q, \numrepresented; x)$ is an incomplete sum since it is a sum over $d$ over a range that might not span the entire period of the summand function. 
We use the following lemma to complete this sum so that a sum over the period of the summand function appears.

\begin{lem} \label{lem:completeTl}
For $\bvec{r} \in \ZZ^\numvars$, positive $q \in \ZZ$, $\numrepresented \in \ZZ$, and $x \in \RR$, the sum $T_{\bvec{r}} (q, \numrepresented; x)$ is 
\begin{align*}
T_{\bvec{r}} (q, \numrepresented; x) &= \sum_{\ell \pmod{q}} \gamma(\ell) K(\ell, \numrepresented, \bvec{r}; q) ,
\end{align*}
where
\begin{align*}
\gamma(\ell) &= \frac{1}{q} \sum_{Q < b \le \min\{q+Q, \lceil 1/(qx) - 1 \rceil \}} \e{- \frac{b \ell}{q}} 
\end{align*}
and
\begin{align*}
K(\ell, \numrepresented, \bvec{r}; q) &= \sum_{d \in (\ZZ/q\ZZ)^\times} \e{\frac{\ell d + n \modmultinv{d}}{q}} \gausssumr{-\modmultinv{d}}{q}{\bvec{r}} . 
\end{align*}
\end{lem}

Before we prove Lemma~\ref{lem:completeTl}, we state a formula related to an indicator function that comes up frequently in number theory and will appear in the proof of Lemma~\ref{lem:completeTl}. 
\begin{lem} \label{lem:sumelin}
Let $a, q \in \ZZ$ and $q>0$. Then 
\begin{align*}
\sum_{j \in \ZZ/q\ZZ} \e{\frac{a j}{q}} &= q \Ind{a \equiv 0 \pmod{q}} \\
	&= \begin{cases}
		q	&\text{if $a \equiv 0 \pmod{q}$,} \\
		0	&\text{otherwise.}
	\end{cases}
\end{align*}
\end{lem}
\begin{proof}
This lemma follows from the fact that the exponential sum appearing in the formula can be viewed as a geometric sum with $j$ ranging from $0$ to $q-1$. 
\end{proof}

Using Lemma~\ref{lem:sumelin}, we now give a proof of Lemma~\ref{lem:completeTl}.
\begin{proof}[Proof of Lemma~\ref{lem:completeTl}]
We want to split up the sum $T_{\bvec{r}} (q, \numrepresented; x)$ into sums over residue classes.
To do this, we note the following fact: For any integer $d$,
\begin{align}
\Ind{\gcd(d,q) = 1} &= \sum_{b \in (\ZZ/q\ZZ)^\times} \Ind{b \equiv d \pmod{q}} = \sum_{b \in (\ZZ/q\ZZ)^\times} \frac{1}{q} \sum_{\ell \pmod{q}} \e{\frac{\ell b - \ell d}{q}} . \label{eq:1summodq}
\end{align}
The last equality follows from Lemma~\ref{lem:sumelin} with $a = b - d$ and $j = \ell$.

We rewrite \eqref{eq:Trq\numrepresented x with gausssumr} as 
\begin{align*}
T_{\bvec{r}} (q, \numrepresented; x) &= \sum_{\substack{Q < d \le q+Q \\ d < 1/(qx) \\ \gcd(d, q) = 1}} \e{\numrepresented \frac{\modmultinv{d}}{q}} \gausssumr{-\modmultinv{d}}{q}{\bvec{r}} \\
	&= \sum_{\substack{Q < d \le \min\{q+Q, \lceil 1/(qx) - 1 \rceil \} \\ \gcd(d, q) = 1}} \e{\numrepresented \frac{\modmultinv{d}}{q}} \gausssumr{-\modmultinv{d}}{q}{\bvec{r}} \\
	&= \sum_{Q < d \le \min\{q+Q, \lceil 1/(qx) - 1 \rceil \}} \Ind{\gcd(d,q) = 1} \e{\numrepresented \frac{\modmultinv{d}}{q}} \gausssumr{-\modmultinv{d}}{q}{\bvec{r}} .
\end{align*}
By applying \eqref{eq:1summodq} to the last expression, we obtain 
\begin{align*}
T_{\bvec{r}} (q, \numrepresented; x) &= \sum_{Q < d \le \min\{q+Q, \lceil 1/(qx) - 1 \rceil \}} \sum_{b \in (\ZZ/q\ZZ)^\times} \Ind{b \equiv d \pmod{q}} \e{\frac{n \modmultinv{d}}{q}} \gausssumr{-\modmultinv{d}}{q}{\bvec{r}} .
\end{align*}
Because $\Ind{b \equiv d \pmod{q}} \e{\frac{\numrepresented \modmultinv{d}}{q}} \gausssumr{-\modmultinv{d}}{q}{\bvec{r}}$ is periodic modulo $q$ and the inner sum only contributes if $b \equiv d \pmod{q}$, we can change some of the $d$'s to $b$'s to obtain 
\begin{align*}
T_{\bvec{r}} (q, \numrepresented; x) &= \sum_{Q < d \le \min\{q+Q, \lceil 1/(qx) - 1 \rceil \}} \sum_{b \in (\ZZ/q\ZZ)^\times} \Ind{b \equiv d \pmod{q}} \e{\frac{n \modmultinv{b}}{q}} \gausssumr{-\modmultinv{b}}{q}{\bvec{r}} .
\end{align*}
Applying \eqref{eq:1summodq} again, we see that
\begin{align*}
T_{\bvec{r}} (q, \numrepresented; x) &= \sum_{Q < d \le \min\{q+Q, \lceil 1/(qx) - 1 \rceil \}} \sum_{b \in (\ZZ/q\ZZ)^\times} \frac{1}{q} \sum_{\ell \pmod{q}} \e{\frac{\ell b - \ell d}{q}} \e{\frac{n \modmultinv{b}}{q}} \gausssumr{-\modmultinv{b}}{q}{\bvec{r}} \\
	&= \sum_{\ell \pmod{q}} \frac{1}{q} \sum_{Q < d \le \min\{q+Q, \lceil 1/(qx) - 1 \rceil \}} \e{\frac{- \ell d}{q}} \\
		&\qquad\qquad\qquad\qquad\qquad\qquad \times \sum_{b \in (\ZZ/q\ZZ)^\times} \e{\frac{\ell b + n \modmultinv{b}}{q}} \gausssumr{-\modmultinv{b}}{q}{\bvec{r}} \\
	&= \sum_{\ell \pmod{q}} \gamma(\ell) K(\ell, \numrepresented, \bvec{r}; q) . \qedhere
\end{align*}
\end{proof}

We have now separated $T_{\bvec{r}} (q, \numrepresented; x)$ into a sum $\gamma(\ell)$ and a complete sum $K(\ell, \numrepresented, \bvec{r}; q)$. We will examine the sum $K(\ell, \numrepresented, \bvec{r}; q)$ more in later subsections. For now, we will focus on $\gamma(\ell)$.
\begin{lem} \label{lem:gamma bound}
If $|\ell| \le \frac{q}{2}$ and $\ell$ is an integer, then
\begin{align*}
| \gamma(\ell) | &\le (1+|\ell|)^{-1} .
\end{align*}
\end{lem}
\begin{proof}
Let $Y = \min\{q+Q, \lceil 1/(qx) - 1 \rceil \}$.

If $\ell = 0$, then
\begin{align*}
\gamma(\ell) = \frac{1}{q} \sum_{Q < b \le Y} \e{- \frac{b \ell}{q}} = \frac{1}{q} \sum_{Q < b \le Y} 1 = \frac{\lfloor Y - Q \rfloor}{q} \le \frac{\lfloor (q+Q) - Q \rfloor}{q} = \frac{q}{q} = 1
\end{align*}
since $Y \le q+Q$ by definition. Thus, $\gamma(\ell) \le (1+|\ell|)^{-1}$ when $\ell=0$.

Now suppose that $0 < |\ell| \le \frac{q}{2}$. Then
\begin{align*}
\gamma(\ell) &= \frac{1}{q} \sum_{Q < b \le Y} \e{- \frac{b \ell}{q}} = \frac{1}{q} \sum_{Q < b \le Y} \left(\e{- \frac{\ell}{q}}\right)^b \\
	&= \frac{1}{q} \e{- \frac{\ell (Q+1)}{q}} \sum_{0 \le b \le \lfloor Y - Q - 1 \rfloor} \left(\e{- \frac{\ell}{q}}\right)^b \\
	&= \frac{1}{q} \e{- \frac{\ell (Q+1)}{q}} \frac{1 - \e{- \frac{\ell}{q}}^{ \lfloor Y - Q \rfloor}}{1 - \e{- \frac{\ell}{q}}} .
\end{align*}
By multiplying the numerator and the denominator by $\e{\frac{\ell}{2q}}$, we obtain
\begin{align}
\gamma(\ell) &= \frac{1}{q} \e{- \frac{\ell (Q+1)}{q}} \e{\frac{\ell}{2q}} \frac{1 - \e{- \frac{\ell}{q}}^{ \lfloor Y - Q \rfloor}}{\e{\frac{\ell}{2q}} - \e{- \frac{\ell}{2q}}} \nonumber \\
	 &= \frac{1}{q} \e{- \frac{\ell (Q+1)}{q}} \e{\frac{\ell}{2q}} \frac{1 - \e{- \frac{\ell}{q}}^{ \lfloor Y - Q \rfloor}}{2 i \sin\left(\frac{\pi \ell}{q}\right)} \label{eq:gamma sine}
\end{align}
since $\sin(2 \pi x) = \dfrac{\e{x}-\e{-x}}{2i}$.
By taking absolute values, we obtain
\begin{align*}
| \gamma(\ell) | &=\frac{1}{q} \frac{\left| 1 - \e{- \frac{\ell}{q}}^{ \lfloor Y - Q \rfloor} \right|}{\left| 2 \sin\left(\frac{\pi \ell}{q}\right) \right|} .
\end{align*}
Since $\left| 1 - \e{- \frac{\ell}{q}}^{ \lfloor Y - Q \rfloor} \right| \le 2$, we have 
\begin{align}
| \gamma(\ell) | &\le \frac{1}{q} \frac{1}{\left| \sin\left(\frac{\pi \ell}{q}\right) \right|} . \label{eq:gamma bound1}
\end{align}

For $|x| \le \frac{\pi}{2}$, we have $|\sin(x)| \ge \frac{2 |x|}{\pi}$. Because $|\ell| \le \frac{q}{2}$, the statement $\left|\frac{\pi \ell}{q}\right| \le \frac{\pi}{2}$ is true. Therefore,
\begin{align*}
\left| \sin\left(\frac{\pi \ell}{q}\right) \right| \ge \frac{2 |\ell|}{q} .
\end{align*}
Substituting this into \eqref{eq:gamma bound1}, we see that
\begin{align*}
| \gamma(\ell) | &\le \frac{1}{2 | \ell |} .
\end{align*}

Observe that $(2 |\ell|)^{-1} \le (1+|\ell|)^{-1}$ since $|\ell| \ge 1$. Therefore,
$\gamma(\ell) \le (1+|\ell|)^{-1}$ when $0 < |\ell| \le \frac{q}{2}$.

Combining this result with the result when $\ell = 0$, we have $\gamma(\ell) \le (1+|\ell|)^{-1}$.
\end{proof}

\subsection{Using Gauss sums}

To examine $K(\ell, \numrepresented, \bvec{r}; q)$, we first analyze the Gauss sum $\gausssumr{d}{q}{\bvec{r}}$.
We first develop a bound for $\gausssumr{d}{q}{\bvec{r}}$ when $d$ and $q$ are coprime.

\begin{lem} \label{lem:gausssumr bound}
If $\gcd(d, q) = 1$ and $\bvec{r} \in \ZZ^\numvars$, then 
\begin{align*}
\left| \gausssumr{d}{q}{\bvec{r}} \right| &\le (\gcd(L, q))^{\numvars/2} q^{\numvars/2}.
\end{align*}
\end{lem}
\begin{proof}
We have
\begin{align*}
\left| \gausssumr{d}{q}{\bvec{r}} \right|^2 &= \left( \sum_{\bvec{h} \in (\ZZ/q\ZZ)^\numvars} \e{\frac{1}{q} (d F(\bvec{h}) + \bvec{h} \cdot \bvec{r})} \right) \lbar{\left( \sum_{\bvec{j} \in (\ZZ/q\ZZ)^\numvars} \e{\frac{1}{q} (d F(\bvec{j}) + \bvec{j} \cdot \bvec{r})} \right)} \\
	&= \left( \sum_{\bvec{h} \in (\ZZ/q\ZZ)^\numvars} \e{\frac{1}{q} (d F(\bvec{h}) + \bvec{h} \cdot \bvec{r})} \right) \left( \sum_{\bvec{j} \in (\ZZ/q\ZZ)^\numvars} \e{-\frac{1}{q} (d F(\bvec{j}) + \bvec{j} \cdot \bvec{r})} \right) \\
	&= \sum_{\bvec{h} \in (\ZZ/q\ZZ)^\numvars} \sum_{\bvec{j} \in (\ZZ/q\ZZ)^\numvars} \e{\frac{1}{q} (d (F(\bvec{h}) - F(\bvec{j})) + (\bvec{h} - \bvec{j}) \cdot \bvec{r})} . \numberthis \label{eq:|gausssumr|^2}
\end{align*}

Let $\bvec{k} = \bvec{h} - \bvec{j}$. Then
\begin{align*}
F(\bvec{h}) &= F(\bvec{j} + \bvec{k}) \\
	&= \frac{1}{2} (\bvec{j} + \bvec{k})^\top A (\bvec{j} + \bvec{k}) \\
	&= \frac{1}{2} \bvec{j}^\top A \bvec{j} + \frac{1}{2} \bvec{j}^\top A \bvec{k} + \frac{1}{2} \bvec{k}^\top A \bvec{j} + \frac{1}{2} \bvec{k}^\top A \bvec{k} \numberthis \label{eq:Fvectorsum1} .
\end{align*}
Because $\bvec{k}^\top A \bvec{j} \in \RR$ and $A$ is symmetric, we have $\bvec{k}^\top A \bvec{j} = \bvec{j}^\top A \bvec{k}$.
Substituting this into \eqref{eq:Fvectorsum1}, we obtain
\begin{align*}
F(\bvec{h}) &= F(\bvec{j}) + \bvec{j}^\top A \bvec{k} + F(\bvec{k}) = F(\bvec{j}) + \bvec{j} \cdot (A \bvec{k}) + F(\bvec{k}) .
\end{align*}
Using this in \eqref{eq:|gausssumr|^2}, we see that
\begin{align*}
\left| \gausssumr{d}{q}{\bvec{r}} \right|^2 &= \sum_{\bvec{j} \in (\ZZ/q\ZZ)^\numvars} \sum_{\bvec{k} \in (\ZZ/q\ZZ)^\numvars} \e{\frac{1}{q} (d (\bvec{j} \cdot (A \bvec{k}) + F(\bvec{k})) + \bvec{k} \cdot \bvec{r})} \\
	&= \sum_{\bvec{k} \in (\ZZ/q\ZZ)^\numvars} \e{\frac{1}{q} (d F(\bvec{k}) + \bvec{k} \cdot \bvec{r})} \sum_{\bvec{j} \in (\ZZ/q\ZZ)^\numvars} \e{\frac{d}{q} \bvec{j} \cdot (A \bvec{k})} . \numberthis \label{eq:|gausssumr|^2 with indfunc}
\end{align*}

Now let $\bvec{w} = A \bvec{k}$. Then
\begin{align*}
\sum_{\bvec{j} \in (\ZZ/q\ZZ)^\numvars} \e{\frac{d}{q} \bvec{j} \cdot (A \bvec{k})} &= \sum_{\bvec{j} \in (\ZZ/q\ZZ)^\numvars} \e{\frac{d}{q} \bvec{j} \cdot \bvec{w}} \\
	&= \sum_{\bvec{j} \in (\ZZ/q\ZZ)^\numvars} \e{\frac{d}{q} \sum_{\ell = 1}^{\numvars} j_\ell w_\ell} \\
	&= \sum_{\bvec{j} \in (\ZZ/q\ZZ)^\numvars} \prod_{\ell = 1}^{\numvars} \e{\frac{d w_\ell j_\ell}{q}} . \numberthis \label{eq:multidimindfuncprod1}
\end{align*}
Because $\bvec{j}$ runs over all vectors in $(\ZZ/q\ZZ)^\numvars$, we have 
\begin{align}
\sum_{\bvec{j} \in (\ZZ/q\ZZ)^\numvars} \prod_{\ell = 1}^{\numvars} \e{\frac{d w_\ell j_\ell}{q}} &= \prod_{\ell = 1}^{\numvars} \left( \sum_{j_\ell \in \ZZ/q\ZZ} \e{\frac{d w_\ell j_\ell}{q}} \right) . \label{eq:multidimindfuncprod2}
\end{align}
By Lemma~\ref{lem:sumelin}, this expression is equal to 
\begin{align*}
\prod_{\ell = 1}^{\numvars} \left( \sum_{j_\ell \in \ZZ/q\ZZ} \e{\frac{d w_\ell j_\ell}{q}} \right) &= \prod_{\ell = 1}^{\numvars} \left( q \Ind{d w_\ell \equiv 0 \pmod{q}} \right) \\
	&= q^\numvars \prod_{\ell = 1}^{\numvars} \Ind{d w_\ell \equiv 0 \pmod{q}} . \numberthis \label{eq:multidimindfuncprod3}
\end{align*}
Because $d$ is coprime to $q$, the last expression is equal to 
\begin{align*}
q^\numvars \prod_{\ell = 1}^{\numvars} \Ind{d w_\ell \equiv 0 \pmod{q}} &= q^\numvars \prod_{\ell = 1}^{\numvars} \Ind{w_\ell \equiv 0 \pmod{q}} \\
	&= q^\numvars \Ind{\bvec{w} \equiv \bvec{0} \pmod{q}} \\
	&= q^\numvars \Ind{A \bvec{k} \equiv \bvec{0} \pmod{q}} . \numberthis \label{eq:multidimindfuncprod4}
\end{align*}

Combining \eqref{eq:|gausssumr|^2 with indfunc}, \eqref{eq:multidimindfuncprod1}, \eqref{eq:multidimindfuncprod2}, \eqref{eq:multidimindfuncprod3}, and \eqref{eq:multidimindfuncprod4}, we obtain 
\begin{align*}
\left| \gausssumr{d}{q}{\bvec{r}} \right|^2 &= q^\numvars \sum_{\bvec{k} \in (\ZZ/q\ZZ)^\numvars} \e{\frac{1}{q} (d F(\bvec{k}) + \bvec{k} \cdot \bvec{r})} \Ind{A \bvec{k} \equiv \bvec{0} \pmod{q}} .
\end{align*}
Therefore,
\begin{align*}
\left| \gausssumr{d}{q}{\bvec{r}} \right|^2 &= \left| q^\numvars \sum_{\bvec{k} \in (\ZZ/q\ZZ)^\numvars} \e{\frac{1}{q} (d F(\bvec{k}) + \bvec{k} \cdot \bvec{r})} \Ind{A \bvec{k} \equiv \bvec{0} \pmod{q}} \right| \\
	&\le q^\numvars \sum_{\bvec{k} \in (\ZZ/q\ZZ)^\numvars} \left| \e{\frac{1}{q} (d F(\bvec{k}) + \bvec{k} \cdot \bvec{r})} \Ind{A \bvec{k} \equiv \bvec{0} \pmod{q}} \right| \\
	&= q^\numvars \sum_{\bvec{k} \in (\ZZ/q\ZZ)^\numvars} \Ind{A \bvec{k} \equiv \bvec{0} \pmod{q}} \\
	&= q^\numvars \left| \left\{ \bvec{k} \in (\ZZ/q\ZZ)^\numvars : A \bvec{k} \equiv \bvec{0} \pmod{q} \right\} \right| . \numberthis \label{ineq:|gausssumr|^2 with setcardinality}
\end{align*}

Recall that $\lvl$ is the smallest positive integer such that $\lvl A^{-1} \in \Mat_\numvars (\ZZ)$. Thus,
\begin{align*}
\left| \left\{ \bvec{k} \in (\ZZ/q\ZZ)^\numvars : A \bvec{k} \equiv \bvec{0} \pmod{q} \right\} \right| &\le \left| \left\{ \bvec{k} \in (\ZZ/q\ZZ)^\numvars : \lvl A^{-1} A \bvec{k} \equiv \lvl A^{-1} \bvec{0} \pmod{q} \right\} \right| \\
	&= \left| \left\{ \bvec{k} \in (\ZZ/q\ZZ)^\numvars : \lvl \bvec{k} \equiv \bvec{0} \pmod{q} \right\} \right| \\
	&= (\gcd(L, q))^\numvars .
\end{align*}
Substituting this into \eqref{ineq:|gausssumr|^2 with setcardinality}, we conclude that 
\begin{align*}
\left| \gausssumr{d}{q}{\bvec{r}} \right|^2 &\le q^\numvars (\gcd(L, q))^\numvars .
\end{align*}
By taking square roots of the previous inequality, we obtain the conclusion of this lemma.
\end{proof}

Lemma~\ref{lem:gausssumr bound} gives an upper bound for the absolute value of the Gauss sum $\gausssumr{d}{q}{\bvec{r}}$. Under certain conditions, we would like to have some exact calculations for the Gauss sum $\gausssumr{d}{q}{\bvec{r}}$. 
In order to do this, we need to look some other Gauss sums called quadratic Gauss sums.

Let $\gausssum{d}{q}$ be the quadratic Gauss sum
\begin{align*}
\gausssum{d}{q} = \sum_{h \pmod{q}} \e{\frac{d h^2}{q}}.
\end{align*}

For an odd integer $q$, let 
\begin{align*}
\varepsilon_q = 
\begin{cases}
1	&\text{if $q \equiv 1 \pmod{4}$,} \\
i	&\text{if $q \equiv 3 \pmod{4}$.}
\end{cases}
\end{align*}

The following lemma is a rephrasing of Theorem~1.5.2 in \cite{BEWGaussSums} and gives a value for $\gausssum{d}{q}$ when $q$ is odd and $d$ is coprime to $q$. 
\begin{lem}[Theorem~1.5.2 in \cite{BEWGaussSums}] \label{lem:gausssum rel prime}
Let $d$ be an integer and $q$ be a positive odd integer. If $\gcd(d, q) = 1$, then
\begin{align*}
\gausssum{d}{q} &= \kronsym{d}{q} \varepsilon_q \sqrt{q} ,
\end{align*}
where $\kronsym{\cdot}{q}$ is the Jacobi symbol.
\end{lem}

To use Lemma~\ref{lem:gausssum rel prime}, we need to be able to diagonalize the Hessian matrix $A$ modulo $q$ when $q$ is coprime $2 \det(A)$. 
We will actually be able to diagonalize $A$ modulo $q$ whenever $q$ is odd. 
To begin with, our next lemma says that you can diagonalize $A$ when $q$ is an odd prime power. It is a rephrasing of Theorem~31 in \cite{WatsonIntQuadForms}.
\begin{lem}[Theorem~31 in \cite{WatsonIntQuadForms}] \label{lem:diagonalizeAppower}
Let $F$ be an integral quadratic form with a Hessian matrix $A$. Suppose that $p$ is an odd prime and $k$ is a positive integer. Then there exist $D, P \in \Mat_\numvars (\ZZ)$ such that $D$ is diagonal, $p \nmid \det(P)$, and
\begin{align}
D &\equiv P^\top A P \pmod{p^k} . \label{cong:DPtAPp^k}
\end{align}
\end{lem}

We use Lemma~\ref{lem:diagonalizeAppower} to prove the following lemma (Lemma~\ref{lem:diagonalizeAqodd}). With Lemma~\ref{lem:diagonalizeAqodd}, we have what we need to diagonalize the Hessian matrix $A$ modulo $q$ when $q$ is odd.
\begin{lem} \label{lem:diagonalizeAqodd}
Let $F$ be an integral quadratic form with a Hessian matrix $A$. Suppose that $q$ is odd. 
Then there exist $D, P \in \Mat_\numvars (\ZZ)$ such that $D$ is diagonal, $\det(P)$ and $q$ are coprime, 
and
\begin{align}
D &\equiv P^\top A P \pmod{q} . \label{cong:DPtAPq}
\end{align}
\end{lem}
\begin{proof}
If $q = 1$, then there is nothing to prove. 

If $q$ is an odd prime power (say $q = p^k$), then we can apply Lemma~\ref{lem:diagonalizeAppower} directly and obtain the result of this lemma. 

Now suppose that $q = \prod_{j=1}^\ell p_j^{e_j}$, where the $p_j$ are distinct primes. As already mentioned, for each $j$, we can find $D_j, P_j \in \Mat_\numvars (\ZZ)$ such that $D_j$ is diagonal, $\det(P_j)$ and $q$ are coprime, 
and \eqref{cong:DPtAPq} is satisfied with $D=D_j$, $P=P_j$, and $q=p_j^{e_j}$. By applying the Chinese remainder theorem to create $D$ from the $D_j$ and $P$ from the $P_j$, we obtain the result of this lemma.
\end{proof}

Since we have now developed the necessary tools, we explicitly compute $\gausssumr{d}{q}{\bvec{r}}$ when $q$ is coprime to $2 \det(A) d$.
\begin{lem} \label{lem:gausssumr coprime eval}
Suppose that $\gcd(q, 2 \det(A) d) = 1$ and $\bvec{r} \in \ZZ^\numvars$. Choose $\alpha \in \ZZ$ so that $\alpha A^{-1} \in \Mat_\numvars (\ZZ)$ and $\gcd(\alpha, q) = 1$. 
Then
\begin{align*}
\gausssumr{d}{q}{\bvec{r}} &= \kronsym{\det(A)}{q} \left( \varepsilon_q \kronsym{2 d}{q} \sqrt{q} \right)^\numvars \e{\frac{- \modmultinv{d}}{q} \modmultinv{2} \modmultinv{\alpha} 2 \alpha F^*(\bvec{r})} ,
\end{align*}
where $\modmultinv{d}, \modmultinv{\alpha}, \modmultinv{2} \in \ZZ/q\ZZ$ are such that $d \modmultinv{d} \equiv \alpha \modmultinv{\alpha} \equiv 2 (\modmultinv{2}) \equiv 1 \pmod{q}$.
\end{lem}
\begin{rmk}
Lemma~\ref{lem:gausssumr coprime eval} is similar to Lemma~20.13 in \cite{IwaniecKowalskiAnalyticNumThy}. However, the lemmas concern slightly different quantities. Also, the statement of Lemma~20.13 in \cite{IwaniecKowalskiAnalyticNumThy} has a few errors that we correct in our statement of Lemma~\ref{lem:gausssumr coprime eval}.
\end{rmk}
\begin{proof}[Proof of Lemma~\ref{lem:gausssumr coprime eval}]
We apply Lemma~\ref{lem:diagonalizeAqodd} to obtain $D, P \in \Mat_\numvars (\ZZ)$ such that $D$ is diagonal, $\gcd(\det(P), q) =1$, 
and
\begin{align}
D \equiv P^\top A P \pmod{q} . \label{cong:DPtAPq in proof}
\end{align}
We have
\begin{align*}
\gausssumr{d}{q}{\bvec{r}} &= \sum_{\bvec{h} \in (\ZZ/q\ZZ)^\numvars} \e{\frac{1}{q} (d F(\bvec{h}) + \bvec{h} \cdot \bvec{r})} 
	= \sum_{\bvec{h} \in (\ZZ/q\ZZ)^\numvars} \e{\frac{1}{q} \left( \modmultinv{2} d \bvec{h}^\top A \bvec{h} + \bvec{h}^\top \bvec{r} \right)} .
\end{align*}
We do a change of variables $\bvec{h} = P \bvec{k}$ to obtain
\begin{align*}
\gausssumr{d}{q}{\bvec{r}} 
	&= \sum_{\bvec{k} \in (\ZZ/q\ZZ)^\numvars} \e{\frac{1}{q} \left( \modmultinv{2} d \bvec{k}^\top P^\top A P \bvec{k} + \bvec{k}^\top P^\top \bvec{r} \right)} \\
	&= \sum_{\bvec{k} \in (\ZZ/q\ZZ)^\numvars} \e{\frac{\modmultinv{2}}{q} \left( d \bvec{k}^\top D \bvec{k} + 2 \bvec{k} \cdot (P^\top \bvec{r}) \right)} . \numberthis \label{eq:gausssumr diag1}
\end{align*}

Let $D = \diag(d_1, d_2, \ldots, d_\numvars)$ and $\bvec{b} = P^\top \bvec{r}$. Then \eqref{eq:gausssumr diag1} becomes 
\begin{align*}
\gausssumr{d}{q}{\bvec{r}} &= \sum_{\bvec{k} \in (\ZZ/q\ZZ)^\numvars} \e{\frac{\modmultinv{2}}{q} \left( d \bvec{k}^\top D \bvec{k} + 2 \bvec{k} \cdot \bvec{b} \right)} \\
	&= \sum_{\bvec{k} \in (\ZZ/q\ZZ)^\numvars} \e{\frac{\modmultinv{2}}{q} \sum_{j=1}^\numvars \left( d d_j k_j^2 + 2 b_j k_j \right)} \\
	&= \sum_{\bvec{k} \in (\ZZ/q\ZZ)^\numvars} \prod_{j=1}^\numvars \e{\frac{\modmultinv{2}}{q} \left( d d_j k_j^2 + 2 b_j k_j \right)} \\
	&= \prod_{j=1}^\numvars \left( \sum_{k_j \in \ZZ/q\ZZ} \e{\frac{\modmultinv{2}}{q} \left( d d_j k_j^2 + 2 b_j k_j \right)} \right) .
\end{align*}

Note that each $d_j$ is coprime to $q$ since $q$ is coprime to $\det(D) = \prod_{j=1}^\numvars d_j$. 
Let $\modmultinv{d}_j$ be the multiplicative inverse of $d_j$ modulo $q$. 
Then by completing the square, we obtain
\begin{align*}
\gausssumr{d}{q}{\bvec{r}} 
	&= \prod_{j=1}^\numvars \e{\frac{- \modmultinv{2} \modmultinv{d} \modmultinv{d}_j b_j^2}{q}} \left( \sum_{k_j \in \ZZ/q\ZZ} \e{\frac{\modmultinv{2} d d_j}{q} (k_j + \modmultinv{d} \modmultinv{d}_j b_j)^2} \right) . \numberthis \label{eq:gausssumr completedsquare}
\end{align*}

Now we make a change of variables $\ell_j = k_j + \modmultinv{d} \modmultinv{d}_j b_j$ to see that 
\begin{align*}
\sum_{k_j \in \ZZ/q\ZZ} \e{\frac{\modmultinv{2} d d_j}{q} (k_j + \modmultinv{d} \modmultinv{d}_j b_j)^2} &= \sum_{\ell_j \in \ZZ/q\ZZ} \e{\frac{\modmultinv{2} d d_j}{q} \ell_j^2} .
\end{align*}
We apply Lemma~\ref{lem:gausssum rel prime} to obtain 
\begin{align*}
\sum_{k_j \in \ZZ/q\ZZ} \e{\frac{\modmultinv{2} d d_j}{q} (k_j + \modmultinv{d} \modmultinv{d}_j b_j)^2} &= \kronsym{\modmultinv{2} d d_j}{q} \varepsilon_q \sqrt{q} .
\end{align*}
Because $\kronsym{\modmultinv{2}}{q} = \kronsym{2}{q}$, we have 
\begin{align*}
\sum_{k_j \in \ZZ/q\ZZ} \e{\frac{\modmultinv{2} d d_j}{q} (k_j + \modmultinv{d} \modmultinv{d}_j b_j)^2} &= \kronsym{2 d d_j}{q} \varepsilon_q \sqrt{q} .
\end{align*}
Substituting this into \eqref{eq:gausssumr completedsquare}, we see that 
\begin{align*}
\gausssumr{d}{q}{\bvec{r}} &= \prod_{j=1}^\numvars \e{\frac{- \modmultinv{2} \modmultinv{d} \modmultinv{d}_j b_j^2}{q}} \kronsym{2 d d_j}{q} \varepsilon_q \sqrt{q} \\
	&= \kronsym{\prod_{j=1}^\numvars d_j}{q} \left( \kronsym{2 d}{q} \varepsilon_q \sqrt{q} \right)^\numvars \prod_{j=1}^\numvars \e{\frac{- \modmultinv{2} \modmultinv{d} \modmultinv{d}_j b_j^2}{q}} \\
	&= \kronsym{\det(D)}{q} \left( \kronsym{2 d}{q} \varepsilon_q \sqrt{q} \right)^\numvars \e{\frac{- \modmultinv{d}}{q} \modmultinv{2} \sum_{j=1}^\numvars \modmultinv{d}_j b_j^2} . \numberthis \label{eq:gausssumr intermed1}
\end{align*}

Because $D \equiv P^\top A P \pmod{q}$, we know that $\det(D) \equiv \det(A) \det(P)^2 \pmod{q}$. Thus, 
\begin{align}
\kronsym{\det(D)}{q} = \kronsym{\det(A) \det(P)^2}{q} = \kronsym{\det(A)}{q} , \label{eq:kronsym dets}
\end{align}
and 
\begin{align}
\gausssumr{d}{q}{\bvec{r}} &= \kronsym{\det(A)}{q} \left( \kronsym{2 d}{q} \varepsilon_q \sqrt{q} \right)^\numvars \e{\frac{- \modmultinv{d}}{q} \modmultinv{2} \sum_{j=1}^\numvars \modmultinv{d}_j b_j^2} . \label{eq:gausssumr intermed2}
\end{align}

For a matrix $B$ that is invertible over $\ZZ/q\ZZ$, let $\modmultinv{B}$ be the multiplicative inverse of $B$ modulo $q$. Note that $\modmultinv{D} = \diag( \modmultinv{d}_1 , \ldots , \modmultinv{d}_\numvars )$ and 
\begin{align*}
\modmultinv{D} = \modmultinv{(P^\top A P)} = \modmultinv{P} \modmultinv{A} \modmultinv{(P^\top)} .
\end{align*}
Therefore, 
\begin{align*}
\sum_{j=1}^\numvars \modmultinv{d}_j b_j^2 &\equiv \bvec{b}^\top \modmultinv{D} \bvec{b} 
	\equiv \bvec{b}^\top \modmultinv{P} \modmultinv{A} \modmultinv{(P^\top)} \bvec{b} \pmod{q} .
\end{align*}
Since $\bvec{b} = P^\top \bvec{r}$, we have
\begin{align*}
\sum_{j=1}^\numvars \modmultinv{d_j} b_j^2 
	&\equiv \bvec{r}^\top \modmultinv{A} \bvec{r} \pmod{q} . \numberthis \label{cong:quadAmodq}
\end{align*}

A short calculation shows that 
\begin{align}
\modmultinv{A} &\equiv \modmultinv{\alpha} \alpha A^{-1} \pmod{q} . \label{cong:Abar}
\end{align}
Substituting this into \eqref{cong:quadAmodq}, we obtain
\begin{align*}
\sum_{j=1}^\numvars \modmultinv{d}_j b_j^2 &\equiv \modmultinv{\alpha} \alpha \bvec{r}^\top A^{-1} \bvec{r} 
	\equiv 2 \modmultinv{\alpha} \alpha F^*(\bvec{r}) \pmod{q} .
\end{align*}
We substitute this into \eqref{eq:gausssumr intermed2} and conclude that 
\begin{align*}
\gausssumr{d}{q}{\bvec{r}} &= \kronsym{\det(A)}{q} \left( \kronsym{2 d}{q} \varepsilon_q \sqrt{q} \right)^\numvars \e{\frac{- \modmultinv{2} \modmultinv{d}}{q} 2 \modmultinv{\alpha} \alpha F^*(\bvec{r})} . \qedhere 
\end{align*}
\end{proof}

\subsection{Decomposing some complete sums}

In the previous subsection, we derived some estimates for the Gauss sum $\gausssumr{d}{q}{\bvec{r}}$. We would like to apply these estimates to the complete sum $K(\ell, \numrepresented, \bvec{r}; q)$. However, we have additional estimates for the Gauss sum when $q$ is coprime to $2 \det(A)$. We would like to exploit as much as we can from these estimates, so we split $q$ into $q = q_0 q_1$ such that $q_0$ is the largest factor of $q$ having all of its prime divisors dividing $2 \det(A)$. This implies that $\gcd(q_1 , 2 \det(A)) = 1$. To do this, we use the following lemma. 
\begin{lem} \label{lem:Ksumdecomp}
Let $q = q_0 q_1$ such that $\gcd(q_0, q_1) =1$. Let
\begin{align*}
K^{(q_1)}(\ell, \numrepresented, \bvec{r}; q_0) &= \sum_{d_0 \in (\ZZ/q_0\ZZ)^\times} \e{\frac{\modmultinv{q}_1 (\ell d_0 + n \modmultinv{d}_0)}{q_0}} \gausssumr{- \modmultinv{q}_1 \modmultinv{d}_0}{q_0}{\modmultinv{q}_1 \bvec{r}} 
\end{align*}
and
\begin{align*}
K^{(q_0)}(\ell, \numrepresented, \bvec{r}; q_1) &= \sum_{d_1 \in (\ZZ/q_1\ZZ)^\times} \e{\frac{\modmultinv{q}_0 (\ell d_1 + n \modmultinv{d}_1)}{q_1}} \gausssumr{- \modmultinv{q}_0 \modmultinv{d}_1}{q_1}{\modmultinv{q}_0 \bvec{r}}, 
\end{align*}
where $\modmultinv{d}_0 , \modmultinv{q}_1 \in \ZZ/q_0\ZZ$ are such that $d_0 \modmultinv{d}_0 \equiv q_1 \modmultinv{q}_1 \equiv 1 \pmod{q_0}$ and $\modmultinv{d}_1 , \modmultinv{q}_0 \in \ZZ/q_1\ZZ$ are such that $d_1 \modmultinv{d}_1 \equiv q_0 \modmultinv{q}_0 \equiv 1 \pmod{q_1}$. Then
\begin{align}
K(\ell, \numrepresented, \bvec{r}; q) &= K^{(q_1)}(\ell, \numrepresented, \bvec{r}; q_0) K^{(q_0)}(\ell, \numrepresented, \bvec{r}; q_1) . \label{eq:Ksplit}
\end{align}
\end{lem}
\begin{proof}
For a given $d$, we have $d_0 \in \ZZ/q_0\ZZ$ and $d_1 \in \ZZ/q_1\ZZ$ such that $d_0 \equiv d \pmod{q_0}$ and $d_1 \equiv d \pmod{q_1}$. By the Chinese remainder theorem, we see that if $d$ coprime to $q$, then $d \equiv d_0 q_1 \modmultinv{q}_1 + d_1 q_0 \modmultinv{q}_0 \pmod{q}$ and $\modmultinv{d} \equiv \modmultinv{d}_0 q_1 \modmultinv{q}_1 + \modmultinv{d}_1 q_0 \modmultinv{q}_0 \pmod{q}$. Applying this to $K(\ell, \numrepresented, \bvec{r}; q)$, we see that 
\begin{align*}
K(\ell, \numrepresented, \bvec{r}; q) &= \sum_{d_0 \in (\ZZ/q_0\ZZ)^\times} \sum_{d_1 \in (\ZZ/q_1\ZZ)^\times} \e{\frac{\ell (d_0 q_1 \modmultinv{q}_1 + d_1 q_0 \modmultinv{q}_0) + n (\modmultinv{d}_0 q_1 \modmultinv{q}_1 + \modmultinv{d}_1 q_0 \modmultinv{q}_0)}{q_0 q_1}} \\
		&\qquad\qquad\qquad\qquad\qquad \times \gausssumr{-(\modmultinv{d}_0 q_1 \modmultinv{q}_1 + \modmultinv{d}_1 q_0 \modmultinv{q}_0)}{q_0 q_1}{\bvec{r}} \\
	&= \sum_{d_0 \in (\ZZ/q_0\ZZ)^\times} \sum_{d_1 \in (\ZZ/q_1\ZZ)^\times} \e{\frac{\ell d_0 \modmultinv{q}_1 + n \modmultinv{d}_0 \modmultinv{q}_1}{q_0}} \e{\frac{\ell d_1 \modmultinv{q}_0 + n \modmultinv{d}_1 \modmultinv{q}_0}{q_1}} \\
		&\qquad\qquad\qquad\qquad\qquad \times \gausssumr{-(\modmultinv{d}_0 q_1 \modmultinv{q}_1 + \modmultinv{d}_1 q_0 \modmultinv{q}_0)}{q_0 q_1}{\bvec{r}} . \numberthis \label{eq:Ksplit1}
\end{align*}

We now take a look at the Gauss sum 
\begin{align}
\gausssumr{-(\modmultinv{d}_0 q_1 \modmultinv{q}_1 + \modmultinv{d}_1 q_0 \modmultinv{q}_0)}{q_0 q_1}{\bvec{r}} &= \sum_{\bvec{h} \in (\ZZ/q\ZZ)^\numvars} \e{\frac{1}{q_0 q_1} (-(\modmultinv{d}_0 q_1 \modmultinv{q}_1 + \modmultinv{d}_1 q_0 \modmultinv{q}_0) F(\bvec{h}) + \bvec{h} \cdot \bvec{r})} . \label{eq:gausssumrsplit1}
\end{align}
For a given $\bvec{h} \in (\ZZ/q\ZZ)^\numvars$, we have $\bvec{j} \in (\ZZ/q_0\ZZ)^\numvars$ and $\bvec{k} \in (\ZZ/q_1\ZZ)^\numvars$ such that $\bvec{j} \equiv \bvec{h} \pmod{q_0}$ and $\bvec{k} \equiv \bvec{h} \pmod{q_1}$. By the Chinese remainder theorem, we see that $\bvec{h} \equiv q_1 \modmultinv{q}_1 \bvec{j} + q_0 \modmultinv{q}_0 \bvec{k} \pmod{q}$. Therefore, \eqref{eq:gausssumrsplit1} becomes 
\begin{align*}
&\gausssumr{-(\modmultinv{d}_0 q_1 \modmultinv{q}_1 + \modmultinv{d}_1 q_0 \modmultinv{q}_0)}{q_0 q_1}{\bvec{r}} \\
	&= \sum_{\bvec{j} \in (\ZZ/q_0\ZZ)^\numvars} \sum_{\bvec{k} \in (\ZZ/q_1\ZZ)^\numvars} \e{\frac{1}{q_0 q_1} (-(\modmultinv{d}_0 q_1 \modmultinv{q}_1 + \modmultinv{d}_1 q_0 \modmultinv{q}_0) F(q_1 \modmultinv{q}_1 \bvec{j} + q_0 \modmultinv{q}_0 \bvec{k})} \\ 
		&\qquad\qquad\qquad\qquad \times \e{\frac{1}{q_0 q_1} (q_1 \modmultinv{q}_1 \bvec{j} + q_0 \modmultinv{q}_0 \bvec{k}) \cdot \bvec{r})} . \\
	&= \sum_{\bvec{j} \in (\ZZ/q_0\ZZ)^\numvars} \sum_{\bvec{k} \in (\ZZ/q_1\ZZ)^\numvars} \e{\frac{1}{q_0 q_1} (-(\modmultinv{d}_0 q_1 \modmultinv{q}_1 + \modmultinv{d}_1 q_0 \modmultinv{q}_0) F(q_1 \modmultinv{q}_1 \bvec{j} + q_0 \modmultinv{q}_0 \bvec{k}) )} \\
		&\qquad\qquad\qquad\qquad \times \e{\frac{\modmultinv{q}_1}{q_0} \bvec{j} \cdot \bvec{r}} \e{\frac{\modmultinv{q}_0}{q_1} \bvec{k} \cdot \bvec{r}} . \numberthis \label{eq:gausssumrsplit2}
\end{align*}

We now take a closer look at the quadratic form appearing in the Gauss sum:
\begin{align*}
F(q_1 \modmultinv{q}_1 \bvec{j} + q_0 \modmultinv{q}_0 \bvec{k}) &= \frac{1}{2} (q_1 \modmultinv{q}_1 \bvec{j} + q_0 \modmultinv{q}_0 \bvec{k})^\top A (q_1 \modmultinv{q}_1 \bvec{j} + q_0 \modmultinv{q}_0 \bvec{k}) \\
	&= \frac{1}{2} (q_1 \modmultinv{q}_1)^2 \bvec{j}^\top A \bvec{j} + q_0 \modmultinv{q}_0 q_1 \modmultinv{q}_1 \bvec{j}^\top A \bvec{k} + \frac{1}{2} (q_0 \modmultinv{q}_0)^2 \bvec{k}^\top A \bvec{k} \\
	&= (q_1 \modmultinv{q}_1)^2 F(\bvec{j}) + q_0 \modmultinv{q}_0 q_1 \modmultinv{q}_1 \bvec{j}^\top A \bvec{k} + (q_0 \modmultinv{q}_0)^2 F(\bvec{k}) \\
	&\equiv (q_1 \modmultinv{q}_1)^2 F(\bvec{j}) + (q_0 \modmultinv{q}_0)^2 F(\bvec{k}) \pmod{q_0 q_1} . \numberthis \label{eq:quadformsplit1}
\end{align*}
Substituting this into \eqref{eq:gausssumrsplit2}, we obtain 
\begin{align*}
&\gausssumr{-(\modmultinv{d}_0 q_1 \modmultinv{q}_1 + \modmultinv{d}_1 q_0 \modmultinv{q}_0)}{q_0 q_1}{\bvec{r}} \\
	&= \sum_{\bvec{j} \in (\ZZ/q_0\ZZ)^\numvars} \sum_{\bvec{k} \in (\ZZ/q_1\ZZ)^\numvars} \e{\frac{1}{q_0} (-\modmultinv{d}_0 \modmultinv{q}_1((q_1 \modmultinv{q}_1)^2 F(\bvec{j}) + (q_0 \modmultinv{q}_0)^2 F(\bvec{k})) )} \\
		&\qquad\qquad \times \e{\frac{1}{q_1} (-\modmultinv{d}_1 \modmultinv{q}_0 ((q_1 \modmultinv{q}_1)^2 F(\bvec{j}) + (q_0 \modmultinv{q}_0)^2 F(\bvec{k})) )} \e{\frac{\modmultinv{q}_1}{q_0} \bvec{j} \cdot \bvec{r}} \e{\frac{\modmultinv{q}_0}{q_1} \bvec{k} \cdot \bvec{r}} \\
	&= \sum_{\bvec{j} \in (\ZZ/q_0\ZZ)^\numvars} \e{\frac{1}{q_0} (-\modmultinv{d}_0 \modmultinv{q}_1 F(\bvec{j}) + \modmultinv{q}_1 \bvec{j} \cdot \bvec{r})} \sum_{\bvec{k} \in (\ZZ/q_1\ZZ)^\numvars} \e{\frac{1}{q_1} (-\modmultinv{d}_1 \modmultinv{q}_0 F(\bvec{k}) + \modmultinv{q}_0 \bvec{k} \cdot \bvec{r})} \\
	&= \gausssumr{-\modmultinv{d}_0 \modmultinv{q}_1}{q_0}{\modmultinv{q}_1 \bvec{r}} \gausssumr{-\modmultinv{d}_1 \modmultinv{q}_0}{q_1}{\modmultinv{q}_0 \bvec{r}} . \numberthis \label{eq:gausssumrsplit3}
\end{align*}
Substituting this into \eqref{eq:Ksplit1}, we obtain \eqref{eq:Ksplit}.
\end{proof}

From now on, unless otherwise specified, let $q_0$ be the largest factor of $q$ having all of its prime divisors dividing $2 \det(A)$ and $q_1 = q/q_0$ so that $\gcd(q_1 , 2 \det(A)) = 1$. Note that $q_0$ and $q_1$ are coprime.

\subsection{Estimating some complete sums}

Now that we are able to decompose $K(\ell, \numrepresented, \bvec{r}; q)$ into other sums, we can provide better estimates on $K(\ell, \numrepresented, \bvec{r}; q)$. We begin by bounding $K^{(q_1)}(\ell, \numrepresented, \bvec{r}; q_0)$.

\begin{lem} \label{lem:Kq1lmnq0 bound}
The sum $K^{(q_1)}(\ell, \numrepresented, \bvec{r}; q_0)$ satisfies the following:
\begin{align*}
| K^{(q_1)}(\ell, \numrepresented, \bvec{r}; q_0) | &\le (\gcd(L, q_0))^{\numvars/2} q_0^{\numvars/2 + 1} .
\end{align*}
\end{lem}
\begin{proof}
Observe that 
\begin{align*}
| K^{(q_1)}(\ell, \numrepresented, \bvec{r}; q_0) | &= \left| \sum_{d_0 \in (\ZZ/q_0\ZZ)^\times} \e{\frac{\modmultinv{q}_1 (\ell d_0 + n \modmultinv{d}_0)}{q_0}} \gausssumr{- \modmultinv{q}_1 \modmultinv{d}_0}{q_0}{\modmultinv{q}_1 \bvec{r}} \right| \\
	&\le \sum_{d_0 \in (\ZZ/q_0\ZZ)^\times} \left| \gausssumr{- \modmultinv{q}_1 \modmultinv{d}_0}{q_0}{\modmultinv{q}_1 \bvec{r}} \right| . \numberthis \label{ineq:Kq1l\numrepresented rq0 ineq1}
\end{align*}
By applying Lemma~\ref{lem:gausssumr bound} to $\left| \gausssumr{- \modmultinv{q}_1 \modmultinv{d}_0}{q_0}{\modmultinv{q}_1 \bvec{r}} \right|$ in \eqref{ineq:Kq1l\numrepresented rq0 ineq1}, we obtain
\begin{align*}
| K^{(q_1)}(\ell, \numrepresented, \bvec{r}; q_0) | &\le \sum_{d_0 \in (\ZZ/q_0\ZZ)^\times} (\gcd(L, q_0))^{\numvars/2} q_0^{\numvars/2} \\
	&\le (\gcd(L, q_0))^{\numvars/2} q_0^{\numvars/2 + 1} . \qedhere
\end{align*}
\end{proof}

Since $q_1$ is coprime to $2\det(A)$, we can better evaluate and bound $K^{(q_0)}(\ell, \numrepresented, \bvec{r}; q_1)$. We first write $K^{(q_0)}(\ell, \numrepresented, \bvec{r}; q_1)$ in terms of a Kloosterman sum or a Sali\'{e} sum. 
We obtain the following by applying Lemma~\ref{lem:gausssumr coprime eval} to the definition of $K^{(q_0)}(\ell, \numrepresented, \bvec{r}; q_1)$.
\begin{lem} \label{lem:K(q_0)(l, \numrepresented, r; q_1) eval}
Let $\alpha \in \ZZ$ be such that $\alpha A^{-1} \in \Mat_\numvars (\ZZ)$ and $\gcd(\alpha, q_1) = 1$.
The sum $K^{(q_0)}(\ell, \numrepresented, \bvec{r}; q_1)$ has the following evaluation:
\begin{align}
K^{(q_0)}(\ell, \numrepresented, \bvec{r}; q_1) &= \kronsym{\det(A)}{q_1} \left( \varepsilon_{q_1} \kronsym{-2 \modmultinv{q}_0 }{q_1} \sqrt{q_1} \right)^\numvars \kappa_{\numvars, q_1} (\modmultinv{q}_0 (\ell + \modmultinv{2} \modmultinv{\alpha} 2 \alpha F^*(\bvec{r})), \modmultinv{q}_0 \numrepresented) , \label{eq:K(q_0)(l, \numrepresented, r; q_1) eval}
\end{align}
where 
\begin{align}
\kappa_{\numvars, q_1} (a,b) &= \sum_{d \pmod{q_1}} \kronsym{d}{q_1}^\numvars \e{\frac{a d + b \modmultinv{d}}{q_1}} 
\end{align}
is either a Kloosterman sum (if $\numvars$ is even) or a Sali\'{e} sum (if $\numvars$ is odd).
\end{lem}
\begin{proof}
We first evaluate $\gausssumr{- \modmultinv{q}_0 \modmultinv{d}}{q_1}{\modmultinv{q}_0 \bvec{r}}$ when $d$ is coprime to $q_1$. Lemma~\ref{lem:gausssumr coprime eval} shows us that 
\begin{align*}
\gausssumr{- \modmultinv{q}_0 \modmultinv{d}}{q_1}{\modmultinv{q}_0 \bvec{r}} &= \kronsym{\det(A)}{q_1} \left( \varepsilon_{q_1} \kronsym{- 2 \modmultinv{q}_0 \modmultinv{d}}{q_1} \sqrt{q_1} \right)^\numvars \e{- \frac{\modmultinv{(- \modmultinv{q}_0 \modmultinv{d})}}{q_1} \modmultinv{2} \modmultinv{\alpha} 2 \alpha F^*(\modmultinv{q}_0 \bvec{r})} \\
	&= \kronsym{\det(A)}{q_1} \left( \varepsilon_{q_1} \kronsym{- 2 \modmultinv{q}_0}{q_1} \sqrt{q_1} \right)^\numvars \kronsym{d}{q_1}^\numvars \e{\frac{d \modmultinv{q}_0 \modmultinv{2} \modmultinv{\alpha} 2 \alpha F^*(\bvec{r})}{q_1}} \numberthis \label{eq:gausssumr in Kq0l\numrepresented rq1}
\end{align*}
since $\kronsym{\modmultinv{d}}{q_1} = \kronsym{d}{q_1}$.

Substituting \eqref{eq:gausssumr in Kq0l\numrepresented rq1} into the definition of $K^{(q_0)}(\ell, \numrepresented, \bvec{r}; q_1)$, we obtain 
\begin{align*}
&K^{(q_0)}(\ell, \numrepresented, \bvec{r}; q_1) \\
	&= \kronsym{\det(A)}{q_1} \left( \varepsilon_{q_1} \kronsym{- 2 \modmultinv{q}_0}{q_1} \sqrt{q_1} \right)^\numvars \\
		&\qquad \times \sum_{\substack{d \pmod{q_1} \\ \gcd(d, q_1) = 1}} \e{\frac{\modmultinv{q}_0 ((\ell + \modmultinv{2} \modmultinv{\alpha} 2 \alpha F^*(\bvec{r})) d + n \modmultinv{d})}{q_1}}\kronsym{d}{q_1}^\numvars ,
\end{align*}
which implies \eqref{eq:K(q_0)(l, \numrepresented, r; q_1) eval} since $\kronsym{d}{q_1} = 0$ if $\gcd(d, q_1) \ne 1$.
\end{proof}

To obtain a bound on $K^{(q_0)}(\ell, \numrepresented, \bvec{r}; q_1)$, we need bounds on Kloosterman sums and Sali\'{e} sums. When $\numvars$ is even, we use a corollary (Corollary~11.12 in \cite{IwaniecKowalskiAnalyticNumThy}) of Weil's bound for Kloosterman sums.
\begin{lem}[Corollary~11.12 in \cite{IwaniecKowalskiAnalyticNumThy}] \label{lem:WeilKloostbound}
If $\numvars$ is even, $a$ and $b$ are integers, and $q$ is a positive integer, then
\begin{align*}
|\kappa_{\numvars, q} (a,b)| &\le \tau(q) (\gcd(a, b, q))^{1/2} q^{1/2} .
\end{align*}
\end{lem}

When $\numvars$ is odd, notice that $\kappa_{\numvars,q}(a,b) = \kappa_{1,q}(a,b)$ is a Sali\'{e} sum.
To obtain a bound for Sali\'{e} sums, we use Lemma~12.4 in \cite{IwaniecKowalskiAnalyticNumThy} to evaluate Sali\'{e} sums under certain circumstances.
\begin{lem}[Lemma~12.4 in \cite{IwaniecKowalskiAnalyticNumThy}] \label{lem:Salieeval}
Suppose that $a$ and $b$ are integers, $q$ is a positive integer, and $\gcd(q, 2a) = 1$. Then
\begin{align*}
\kappa_{1,q}(a,b) 
	&= \varepsilon_q q^{1/2} \kronsym{a}{q} \sum_{\substack{v \pmod{q} \\ v^2 \equiv a b \pmod{q}}} \e{\frac{2v}{q}} .
\end{align*}
\end{lem}

We would like to partially remove the coprimallity condition for the evaluation of Sali\'{e} sums. To do this, we notice that $\kappa_{1,q}(a,b) = \kappa_{1,q}(b,a)$ 
since $\kronsym{x}{q} = \kronsym{\modmultinv{x}}{q}$. Therefore, we only need to consider the case when $a$, $b$, and $q$ have a common divisor. 
(For our purposes, $q$ will already be odd.)

To do this, we first decompose Sali\'{e} sums so that we only have to concern ourselves with the case that $q$ is a prime power, i.e., $q = p^k$ for some odd prime $p$. This decomposition is made precise with the following lemma about the twisted multiplicatiive property of Sali\'{e} sums.
\begin{lem} \label{lem:Saliesumtwistedmult}
Suppose that $a$ and $b$ are integers and $q_1$ and $q_2$ are positive odd integers such that $\gcd(q_1, q_2) = 1$. Then
\begin{align*}
\kappa_{1,q_1 q_2} (a,b) &= \kappa_{1,q_1} (a \modmultinv{q}_2, b \modmultinv{q}_2) \kappa_{1,q_2} (a \modmultinv{q}_1, b \modmultinv{q}_1) ,
\end{align*}
where $q_2 \modmultinv{q}_2 \equiv 1 \pmod{q_1}$ and $q_1 \modmultinv{q}_1 \equiv 1 \pmod{q_2}$.
\end{lem}
\begin{proof}
By definition,
\begin{align*}
\kappa_{1, q_1 q_2} (a,b) &= \sum_{d \pmod{q_1 q_2}} \kronsym{d}{q_1 q_2} \e{\frac{a d + b \modmultinv{d}}{q_1 q_2}} \\
	&= \sum_{d \pmod{q_1 q_2}} \kronsym{d}{q_1} \kronsym{d}{q_2} \e{\frac{a d + b \modmultinv{d}}{q_1 q_2}} . \numberthis \label{eq:Saliesummult1}
\end{align*}
For a given $d$ coprime to $q_1 q_2$, we have $d_1 \in \ZZ/q_1\ZZ$ and $d_2 \in \ZZ/q_2\ZZ$ such that $d_1 \equiv d \pmod{q_1}$ and $d_2 \equiv d \pmod{q_2}$. 
Let $\modmultinv{d}_1 \in \ZZ/q_1\ZZ$ be such that $d_1 \modmultinv{d}_1 \equiv 1 \pmod{q_1}$, and let $\modmultinv{d}_2 \in \ZZ/q_2\ZZ$ be such that $d_2 \modmultinv{d}_2 \equiv 1 \pmod{q_2}$.
By the Chinese remainder theorem, we have $d \equiv d_1 q_2 \modmultinv{q}_2 + d_2 q_1 \modmultinv{q}_1 \pmod{q_1 q_2}$ and $\modmultinv{d} \equiv \modmultinv{d}_1 q_2 \modmultinv{q}_2 + \modmultinv{d}_2 q_1 \modmultinv{q}_1 \pmod{q_1 q_2}$. 
Applying these facts to \eqref{eq:Saliesummult1}, we see that
\begin{align*}
&\kappa_{1, q_1 q_2} (a,b) \\
	&= \sum_{d_1 \pmod{q_1}} \sum_{d_2 \pmod{q_2}} \kronsym{d_1}{q_1} \kronsym{d_2}{q_2} \e{\frac{a (d_1 q_2 \modmultinv{q}_2 + d_2 q_1 \modmultinv{q}_1) + b (\modmultinv{d}_1 q_2 \modmultinv{q}_2 + \modmultinv{d}_2 q_1 \modmultinv{q}_1)}{q_1 q_2}} \\
	&= \sum_{d_1 \pmod{q_1}} \kronsym{d_1}{q_1} \e{\frac{a \modmultinv{q}_2 d_1 + b \modmultinv{q}_2 \modmultinv{d}_1}{q_1}} \sum_{d_2 \pmod{q_2}} \kronsym{d_2}{q_2} \e{\frac{a \modmultinv{q}_1 d_2 + b \modmultinv{q}_1 \modmultinv{d}_2}{q_2}} \\
	&= \kappa_{1,q_1} (a \modmultinv{q}_2, b \modmultinv{q}_2) \kappa_{1,q_2} (a \modmultinv{q}_1, b \modmultinv{q}_1) .  \qedhere
\end{align*}
\end{proof}

The previous lemma can be applied repeatedly to $\kappa_{1,c}(a,b)$ to obtain a product of Sali\'{e} sums with each of their denominators being prime powers. This allows us to only consider the case of $q$ being a prime power. 
We would like to bound on Sali\'{e} sums when $q$ is a prime power. To do this, we need the following bound on an exponential sum.
\begin{lem} \label{lem:v2sum in Salie}
Suppose that $a$ is an integer, $p$ is an odd prime, and $k$ is a positive integer. Then 
\begin{align}
\left| \sum_{\substack{v \pmod{p^k} \\ v^2 \equiv a \pmod{p^k}}} \e{\frac{2v}{p^k}} \right| \le 2 . \label{eq:sumv2cong}
\end{align}
\end{lem}
\begin{proof}
First suppose that $a \equiv 0 \pmod{p^k}$. Then the sum in \eqref{eq:sumv2cong} is a sum over $v$ of the form $v = p^{\left\lceil k/2 \right\rceil} u$, where $u \in \ZZ/p^{\left\lfloor k/2 \right\rfloor} \ZZ$. That is,
\begin{align*}
 \sum_{\substack{v \pmod{p^k} \\ v^2 \equiv a \pmod{p^k}}} \e{\frac{2v}{p^k}} &= \sum_{u \in \ZZ/p^{\left\lfloor k/2 \right\rfloor} \ZZ} \e{\frac{2 p^{\left\lceil k/2 \right\rceil} u}{p^{k}}} \\
	&= \sum_{u \in \ZZ/p^{\left\lfloor k/2 \right\rfloor} \ZZ} \e{\frac{2 u}{p^{\left\lfloor k/2 \right\rfloor}}} \\
	&= \Ind{\left\lfloor k/2 \right\rfloor = 0} 
\end{align*}
by Lemma~\ref{lem:sumelin} and the fact that $p$ is odd. Therefore, if $a \equiv 0 \pmod{p^{k}}$, then \eqref{eq:sumv2cong} holds.

Now suppose that $a \not\equiv 0 \pmod{p^{k}}$. Let $a_0 p^{a_1} \equiv a \pmod{p^{k}}$, where $0 \le a_1 \le k-1$ and $a_0 \in {(\ZZ/p^{k-a_1}\ZZ)}^\times$. We first explore the possible solutions to 
\begin{align}
v^2 \equiv a \pmod{p^{k}} . \label{cong:v2apk}
\end{align}
If there exists a solution $v \in \ZZ/p^{k}\ZZ$ to \eqref{cong:v2apk}, then $v \not\equiv 0 \pmod{p^{k}}$ and there exist integers $v_0$ and $v_1$ such that $0 \le v_1 \le k-1$, $0 \le v_0 \le p^{k-v_1} - 1$, $\gcd(v_0 , p^{k-v_1}) = 1$, and $v \equiv v_0 p^{v_1} \pmod{p^{k}}$. In order for $v^2 \equiv a \pmod{p^{k}}$, we must have 
\begin{align}
v_1 = \frac{a_1}{2} \label{eq:v1 cond}
\end{align}
and 
\begin{align}
v_0^2 \equiv a_0 \pmod{p^{k-a_1}} . \label{cong:v0 cond}
\end{align}
Because ${(\ZZ/p^{k-a_1}\ZZ)}^\times$ is cyclic, there are either zero or two solutions $x_0 \in \ZZ$ such that $0 \le x_0 \le p^{k-a_1} - 1$, $\gcd(x_0 , p^{k-a_1}) = 1$, and $x_0$ satisfies 
\begin{align}
x_0^2 \equiv a_0 \pmod{p^{k-a_1}} . \label{cong:x0 cond}
\end{align}
Each solution $x_0$ lifts to $p^{a_1 - v_1} = p^{\frac{a_1}{2}}$ solutions for $v_0 \in {(\ZZ/p^{k-v_1}\ZZ)}^\times = {(\ZZ/p^{k-\frac{a_1}{2}}\ZZ)}^\times$. Each lift is of the form 
\begin{align}
v_0 &\equiv x_0 + x_1 p^{k-a_1} \pmod{p^{k-\frac{a_1}{2}}} , \label{cong:v0x0lift}
\end{align}
where $x_1 \in \ZZ/p^{\frac{a_1}{2}}\ZZ$. One can verify that each $v_0$ of the form in \eqref{cong:v0x0lift} contributes to a unique solution $v = v_0 p^{\frac{a_1}{2}}$ to \eqref{cong:v2apk}.
Therefore, we have shown that there are either $0$ or $2 p^{\frac{a_1}{2}}$ solutions to \eqref{cong:v2apk} modulo $p^{k}$.

It is clear that if there are no solutions to \eqref{cong:v2apk}, then 
\begin{align*}
\sum_{\substack{v \pmod{p^k} \\ v^2 \equiv a \pmod{p^k}}} \e{\frac{2v}{p^k}} = 0 .
\end{align*}
Thus, we only need to concern ourselves with the case that there are $2 p^{\frac{a_1}{2}}$ solutions to \eqref{cong:v2apk}. Let $x_0 \in \ZZ \cap [0, p^{k-a_1} - 1]$ be coprime to $p^{k-a_1}$ and satisfy \eqref{cong:x0 cond}. 
The only other integer in $[0, p^{k-a_1} -1]$ that is coprime to $p^{k-a_1}$ and satisfies \eqref{cong:x0 cond} is $p^{k-a_1} - x_0$. 
Therefore, using \eqref{cong:v0x0lift}, we find that 
\begin{align*}
&\sum_{\substack{v \pmod{p^k} \\ v^2 \equiv a \pmod{p^k}}} \e{\frac{2v}{p^k}} \\
	&= \sum_{x_1 \in \ZZ/p^{\frac{a_1}{2}}\ZZ} \left( \e{\frac{2 (x_0 + x_1 p^{k-a_1}) p^{\frac{a_1}{2}}}{p^{k}}} + \e{\frac{2 (p^{k-a_1} - x_0 + x_1 p^{k-a_1}) p^{\frac{a_1}{2}}}{p^{k}}} \right) \\
	&= \left( \e{\frac{2 x_0}{p^{k-\frac{a_1}{2}}}} + \e{\frac{2 (p^{k-a_1} - x_0)}{p^{k-\frac{a_1}{2}}}} \right) \sum_{x_1 \in \ZZ/p^{\frac{a_1}{2}}\ZZ} \e{\frac{2 x_1}{p^{\frac{a_1}{2}}}} . 
\end{align*}
By Lemma~\ref{lem:sumelin} and the fact that $p$ is odd, we obtain
\begin{align}
\sum_{\substack{v \pmod{p^k} \\ v^2 \equiv a \pmod{p^k}}} \e{\frac{2v}{p^k}} 
	&= \left( \e{\frac{2 x_0}{p^{k-\frac{a_1}{2}}}} + \e{\frac{2 (p^{k-a_1} - x_0)}{p^{k-\frac{a_1}{2}}}} \right) \Ind{a_1 = 0} . \label{eq:sumv2cong split}
\end{align}
We take absolute values of both sides of \eqref{eq:sumv2cong split} to obtain \eqref{eq:sumv2cong}.
\end{proof}

Using Lemma~\ref{lem:v2sum in Salie}, we obtain the following bound on Sali\'{e} sums when $q$ is a prime power.
\begin{lem} \label{lem:Saliesumprimepows}
Suppose that $a$ and $b$ are integers, $k$ is a positive integer, $\ell$ is a nonnegative integer, $p$ is an odd prime, and $\gcd(a,p)=1$. Then
\begin{align}
|\kappa_{1, p^k} (a p^\ell, b p^\ell)| \le \tau(p^k) (\gcd(a p^\ell, b p^\ell, p^k))^{1/2} p^{k/2} . \label{ineq:Saliesumprimepows bound}
\end{align}
\end{lem}
\begin{proof}
By definition,
\begin{align*}
\kappa_{1, p^k} (a p^\ell, b p^\ell) &= \sum_{d \pmod{p^k}} \kronsym{d}{p^k} \e{\frac{a p^\ell d + b p^\ell \modmultinv{d}}{p^k}} .
\end{align*}

Suppose that $\ell \ge k$. Then $\e{\frac{a p^\ell d + b p^\ell \modmultinv{d}}{p^k}}$ is always $1$. Thus, since $\gcd(a p^\ell, b p^\ell, p^k) = p^k$, when we apply a trivial bound for the sum, we obtain 
\begin{align*}
| \kappa_{1, p^k} (a p^\ell, b p^\ell) | 
\le p^k \le \tau(p^k) (\gcd(a p^\ell, b p^\ell, p^k))^{1/2} p^{k/2} .
\end{align*}

Now suppose that $\ell < k$. Then
\begin{align*}
\kappa_{1, p^k} (a p^\ell, b p^\ell) &= \sum_{d \pmod{p^k}} \kronsym{d}{p}^k \e{\frac{a d + b \modmultinv{d}}{p^{k-\ell}}} .
\end{align*}
Now $ \kronsym{d}{p}^k \e{\dfrac{a d + b \modmultinv{d}}{p^{k-\ell}}}$ is periodic modulo $p^{k-\ell}$, so
\begin{align}
\kappa_{1, p^k} (a p^\ell, b p^\ell) &= p^\ell \sum_{d \pmod{p^{k-\ell}}} \kronsym{d}{p}^k \e{\frac{a d + b \modmultinv{d}}{p^{k-\ell}}} . \label{eq:Saliesumprimepows1}
\end{align}

If $k$ is even, then $\sum_{d \pmod{p^{k-\ell}}} \kronsym{d}{p}^k \e{\dfrac{a d + b \modmultinv{d}}{p^{k-\ell}}}$ is a Kloosterman sum and we can apply Lemma~\ref{lem:WeilKloostbound}
 and see that 
\begin{align*}
| \kappa_{1, p^k} (a p^\ell, b p^\ell) | &= p^\ell \tau(p^{k-\ell}) \gcd(a, b, p^{k-\ell})^{1/2} (p^{k-\ell})^{1/2} = \tau(p^{k-\ell}) (p^{\ell})^{1/2} (p^{k})^{1/2} \\
	&\le \tau(p^{k}) (p^{\ell})^{1/2} (p^{k})^{1/2} .
\end{align*}
Since $\gcd(a p^\ell, b p^\ell, p^k) = p^\ell$, we obtain \eqref{ineq:Saliesumprimepows bound} in the case that $\ell < k$ and $k$ is even.

Now suppose that $k$ is odd and $\ell$ is even. Then 
\begin{align*}
\sum_{d \pmod{p^{k-\ell}}} \kronsym{d}{p}^k \e{\frac{a d + b \modmultinv{d}}{p^{k-\ell}}} = \kappa_{1, p^{k-\ell}} (a,b) ,
\end{align*}
so we can apply Lemma~\ref{lem:Salieeval} and obtain 
\begin{align*}
\kappa_{1, p^k} (a p^\ell, b p^\ell) &= p^\ell \varepsilon_ {p^{k-\ell}} (p^{k-\ell})^{1/2} \kronsym{a}{p^{k-\ell}} \sum_{\substack{v \pmod{p^{k-\ell}} \\ v^2 \equiv a b \pmod{p^{k-\ell}}}} \e{\frac{2v}{p^{k-\ell}}} .
\end{align*}
By taking absolute values of both sides and applying Lemma~\ref{lem:v2sum in Salie}, we find that 
\begin{align}
| \kappa_{1, p^k} (a p^\ell, b p^\ell) | &\le 2 (p^\ell)^{1/2} (p^{k})^{1/2} . \label{ineq:Saliesumkoddleven}
\end{align}
Since $\gcd(a p^\ell, b p^\ell, p^k) = p^\ell$ and $\tau(p^k) \ge 2$, we obtain \eqref{ineq:Saliesumprimepows bound} in the case that $\ell < k$, $k$ is odd, and $\ell$ is even.

Now suppose that $k$ and $\ell$ are odd. Then
\begin{align}
\sum_{d \pmod{p^{k-\ell}}} \kronsym{d}{p}^k \e{\frac{a d + b \modmultinv{d}}{p^{k-\ell}}} &= \sum_{d \pmod{p^{k-\ell}}} \kronsym{d}{p} \e{\frac{a d + b \modmultinv{d}}{p^{k-\ell}}} \label{eq:Saliesumklodd}
\end{align}
and $k-\ell$ is even. 

Let $d = d_1 + d_2 p^{\frac{k-\ell}{2}}$, where $0 \le d_1 \le p^{\frac{k-\ell}{2}} -1$ and $d_2 \in \ZZ / p^{\frac{k-\ell}{2}} \ZZ$. If $\gcd(d_1 , p) = 1$, then let $\modmultinv{d}_1 \in \ZZ$ be such that $0 < \modmultinv{d}_1 < p^{k-\ell}$ and $d_1 \modmultinv{d}_1 \equiv 1 \pmod{p^{k-\ell}}$. Observe that $\modmultinv{d} \equiv \modmultinv{d}_1 - (\modmultinv{d}_1)^2 d_2 p^{\frac{k-\ell}{2}} \pmod{p^{k-\ell}}$ since
$( d_1 + d_2 p^{\frac{k-\ell}{2}}) (\modmultinv{d}_1 - (\modmultinv{d}_1)^2 d_2 p^{\frac{k-\ell}{2}}) \equiv 1 \pmod{p^{k-\ell}}$.
Using this in \eqref{eq:Saliesumklodd}, we see that 
\begin{align*}
&\sum_{d \pmod{p^{k-\ell}}} \kronsym{d}{p}^k \e{\frac{a d + b \modmultinv{d}}{p^{k-\ell}}} \\
	&= \sum_{d_1 =0}^{p^{\frac{k-\ell}{2}} - 1} \sum_{d_2 \pmod{p^{\frac{k-\ell}{2}}}} \kronsym{d_1 + d_2 p^{\frac{k-\ell}{2}}}{p} \e{\frac{a (d_1 + d_2 p^{\frac{k-\ell}{2}}) + b (\modmultinv{d}_1 - (\modmultinv{d}_1)^2 d_2 p^{\frac{k-\ell}{2}})}{p^{k-\ell}}} \\
	&= \sum_{d_1 =0}^{p^{\frac{k-\ell}{2}} - 1} \kronsym{d_1}{p} \e{\frac{a d_1 + b \modmultinv{d}_1}{p^{k-\ell}}} \sum_{d_2 \pmod{p^{\frac{k-\ell}{2}}}} \e{\frac{(a - b (\modmultinv{d}_1)^2) d_2}{p^{\frac{k-\ell}{2}}}} . \numberthis \label{eq:Saliesumklodd1}
\end{align*}

Applying Lemma~\ref{lem:sumelin} to the sum over $d_2$ in \eqref{eq:Saliesumklodd1}, we obtain
\begin{align*}
\sum_{d \pmod{p^{k-\ell}}} \kronsym{d}{p}^k \e{\frac{a d + b \modmultinv{d}}{p^{k-\ell}}} 
	&= p^{\frac{k-\ell}{2}} \sum_{d_1 =0}^{p^{\frac{k-\ell}{2}} - 1} \kronsym{d_1}{p} \e{\frac{a d_1 + b \modmultinv{d}_1}{p^{k-\ell}}} \Ind{a \equiv b (\modmultinv{d}_1)^2 \pmod{p^{\frac{k-\ell}{2}}}} .
\end{align*}
We take absolute values of both sides to obtain 
\begin{align*}
&\left| \sum_{d \pmod{p^{k-\ell}}} \kronsym{d}{p}^k \e{\frac{a d + b \modmultinv{d}}{p^{k-\ell}}} \right| \\
	&\le p^{\frac{k-\ell}{2}} \left| \left\{ d_1 \in \ZZ : 0 \le d_1 \le p^{\frac{k-\ell}{2}} -1, \, \gcd(d_1,p)=1, \, a \equiv b (\modmultinv{d}_1)^2 \pmod{p^{\frac{k-\ell}{2}}} \right\} \right| . \numberthis \label{ineq:Saliesumklodd2}
\end{align*}
since $\kronsym{d_1}{p} = 0$ if $\gcd(d_1 , p) \ne 1$.

Note that if $\gcd(d_1,p)=1$, then the condition $a \equiv b (\modmultinv{d}_1)^2 \pmod{p^{\frac{k-\ell}{2}}}$ is equivalent to the condition $d_1^2 \equiv \modmultinv{a} b \pmod{p^{\frac{k-\ell}{2}}}$. Thus, it suffices to count the number of $d_1$ such that $0 \le d_1 \le p^{\frac{k-\ell}{2}} -1$, $\gcd(d_1,p)=1$, 
and $d_1^2 \equiv \modmultinv{a} b \pmod{p^{\frac{k-\ell}{2}}}$. 
Now the number of $d_0 \pmod{p}$ such that $\gcd(d_0,p)=1$ and 
$d_0^2 \equiv \modmultinv{a} b \pmod{p}$ is $1 + \kronsym{\modmultinv{a} b}{p}$. 
By applying Hensel's lemma, we see that the number of $d_1 \pmod{p^{\frac{k-\ell}{2}}}$ such that $\gcd(d_1, p) = 1$ and 
$d_1^2 \equiv \modmultinv{a} b \pmod{p^{\frac{k-\ell}{2}}}$ is also $1 + \kronsym{\modmultinv{a} b}{p}$. 
Hence, we have
\begin{align*}
&\left| \left\{ d_1 \in \ZZ : 0 \le d_1 \le p^{\frac{k-\ell}{2}} -1, \, \gcd(d_1,p)=1, \, a \equiv b (\modmultinv{d}_1)^2 \pmod{p^{\frac{k-\ell}{2}}} \right\} \right| \\
	&= 1 + \kronsym{\modmultinv{a} b}{p} .
\end{align*}

Applying this to \eqref{ineq:Saliesumklodd2}, we find that
\begin{align*}
\left| \sum_{d \pmod{p^{k-\ell}}} \kronsym{d}{p}^k \e{\frac{a d + b \modmultinv{d}}{p^{k-\ell}}} \right| &\le p^{\frac{k-\ell}{2}} \left( 1 + \kronsym{\modmultinv{a} b}{p} \right) 
\le 2 p^{\frac{k-\ell}{2}} .
\end{align*}
We apply this to \eqref{eq:Saliesumprimepows1} to see that 
\begin{align*}
|\kappa_{1, p^k} (a p^\ell, b p^\ell)| \le 2 p^\ell p^{\frac{k-\ell}{2}} = 2 (p^\ell)^{1/2} (p^k)^{1/2} \le \tau(p^k) (p^\ell)^{1/2} (p^k)^{1/2} .
\end{align*}
Since $\gcd(a p^\ell, b p^\ell, p^k) = p^\ell$, we obtain \eqref{ineq:Saliesumprimepows bound} in the case that $\ell < k$ and $k$ and $\ell$ are both odd.
\end{proof}

Because the expression $\tau(q) (\gcd(a , b, q))^{1/2} q^{1/2}$ is multiplicative as a function of $q$, we can repeatedly apply Lemmas~\ref{lem:Saliesumtwistedmult} and \ref{lem:Saliesumprimepows} to obtain the following result. 
\begin{lem} \label{lem:WeilSaliebound}
If $\numvars$ is odd, $a$ and $b$ are integers, and $q$ is a positive odd integer, then
\begin{align}
|\kappa_{\numvars, q} (a,b)| &\le \tau(q) (\gcd(a, b, q))^{1/2} q^{1/2} . \label{ineq:Saliebound}
\end{align}
\end{lem}
Lemma~\ref{lem:WeilSaliebound} follows from multiple applications of Lemmas~\ref{lem:Saliesumtwistedmult} and \ref{lem:Saliesumprimepows}.

Combining the previous lemma with Lemma~\ref{lem:WeilKloostbound}, we have the following bound on $\kappa_{\numvars, q} (a,b)$ when $q$ is odd.
\begin{lem} \label{lem:KloostSaliebound}
If $a$, $b$, and $\numvars$ are integers and $q$ is a positive odd integer, then
\begin{align}
|\kappa_{\numvars, q} (a,b)| &\le \tau(q) (\gcd(a, b, q))^{1/2} q^{1/2} .
\end{align}
\end{lem}

By applying Lemmas~\ref{lem:K(q_0)(l, \numrepresented, r; q_1) eval} and \ref{lem:KloostSaliebound} to $K^{(q_0)}(\ell, \numrepresented, \bvec{r}; q_1)$, we obtain the following result.
\begin{lem} \label{lem:Kq0lmnq1 bound}
Let $\alpha \in \ZZ$ be such that $\alpha A^{-1} \in \Mat_\numvars (\ZZ)$ and $\gcd(\alpha, q_1) = 1$. 
The sum $K^{(q_0)}(\ell, \numrepresented, \bvec{r}; q_1)$ satisfies the following:
\begin{align*}
|K^{(q_0)}(\ell, \numrepresented, \bvec{r}; q_1)| &\le q_1^{(\numvars+1)/2} \tau(q_1) (\gcd(\ell + \modmultinv{2} \modmultinv{\alpha} 2 \alpha F^*(\bvec{r}), \numrepresented , q_1))^{1/2} .
\end{align*}
\end{lem}
\begin{proof}
By applying Lemmas~\ref{lem:K(q_0)(l, \numrepresented, r; q_1) eval} and \ref{lem:KloostSaliebound} to $K^{(q_0)}(\ell, \numrepresented, \bvec{r}; q_1)$, we see that 
\begin{align*}
|K^{(q_0)}(\ell, \numrepresented, \bvec{r}; q_1)| &\le q_1^{(\numvars+1)/2} \tau(q_1) (\gcd(\modmultinv{q}_0 (\ell + \modmultinv{2} \modmultinv{\alpha} 2 \alpha F^*(\bvec{r})), \modmultinv{q}_0 \numrepresented , q_1))^{1/2} .
\end{align*}
Because $\gcd(\modmultinv{q}_0, q_1) =1$, we obtain the result of the lemma.
\end{proof}

\subsection{Bounding the arithmetic part}

We have now provided bounds for $\gamma(\ell)$, and $K^{(q_1)}(\ell, \numrepresented, \bvec{r}; q_0)$, and $K^{(q_0)}(\ell, \numrepresented, \bvec{r}; q_1)$. 
Therefore, we have what we need to compute an upper bound for the absolute value of $T_{\bvec{r}} (q, \numrepresented; x)$. 
Such an upper bound appears in the following lemma. 

\begin{lem} \label{lem:Trq\numrepresented x bound}
For $\bvec{r} \in \ZZ^\numvars$, $\numrepresented \in \ZZ$, $x \in \RR$, and positive $q \in \ZZ$, the sum $T_{\bvec{r}} (q, \numrepresented; x)$ is 
\begin{align}
T_{\bvec{r}} (q, \numrepresented; x) &\ll (\gcd(\lvl,q_0))^{\numvars/2} (\gcd(\numrepresented, q_1))^{1/2} q_0^{1/2} q^{(\numvars+1)/2} \tau(q) \log(2q) .
\end{align}
The implied constant does not depend on $F$.
\end{lem}
\begin{proof}
We use Lemmas~\ref{lem:completeTl} and \ref{lem:Ksumdecomp} to obtain
\begin{align}
T_{\bvec{r}} (q, \numrepresented; x) &= \sum_{\ell \pmod{q}} \gamma(\ell) K^{(q_1)}(\ell, \numrepresented, \bvec{r}; q_0) K^{(q_0)}(\ell, \numrepresented, \bvec{r}; q_1) .
\end{align}
We now apply Lemmas~\ref{lem:gamma bound}, \ref{lem:Kq1lmnq0 bound}, and \ref{lem:Kq0lmnq1 bound} to see that 
\begin{align*}
| T_{\bvec{r}} (q, \numrepresented; x) | &\le 
	\sum_{-\frac{q}{2} < \ell \le \frac{q}{2}} (1+|\ell|)^{-1} (\gcd(L, q_0))^{\numvars/2} q_0^{\numvars/2 + 1} q_1^{(\numvars+1)/2} \\
		&\qquad\qquad \times \tau(q_1) (\gcd(\ell + \modmultinv{2} \modmultinv{\alpha} 2 \alpha F^*(\bvec{r}), \numrepresented , q_1))^{1/2} \\
	&= 
	q^{(\numvars+1)/2} q_0^{1/2} \tau(q_1) (\gcd(L, q_0))^{\numvars/2} \\
		&\qquad\qquad \times \sum_{-\frac{q}{2} < \ell \le \frac{q}{2}} (1+|\ell|)^{-1} (\gcd(\ell + \modmultinv{2} \modmultinv{\alpha} 2 \alpha F^*(\bvec{r}), \numrepresented , q_1))^{1/2} \\
	&\le q^{(\numvars+1)/2} q_0^{1/2} \tau(q) (\gcd(L, q_0))^{\numvars/2} \sum_{-\frac{q}{2} < \ell \le \frac{q}{2}} (1+|\ell|)^{-1} (\gcd(\numrepresented , q_1))^{1/2} \\
	&\le q^{(\numvars+1)/2} q_0^{1/2} \tau(q) (\gcd(L, q_0))^{\numvars/2} (\gcd(\numrepresented , q_1))^{1/2} \left( 1 + 2 \sum_{\ell=1}^{\left\lfloor \frac{q}{2} \right\rfloor} (1+\ell)^{-1} \right) \numberthis \label{ineq:Trqn} .
\end{align*}

We use Euler's summation formula \cite[Theorem~3.1]{ApostolAnalNumThy} 
to bound the sum in \eqref{ineq:Trqn}:
\begin{align*}
\sum_{\ell=1}^{\left\lfloor \frac{q}{2} \right\rfloor} (1+\ell)^{-1} &= \frac{1}{2} + \int_{1}^{\left\lfloor \frac{q}{2} \right\rfloor} (1+\ell)^{-1} \ d\ell - \int_{1}^{\left\lfloor \frac{q}{2} \right\rfloor} (\ell - \lfloor \ell \rfloor) (1+\ell)^{-2} \ d\ell \\
	&= \frac{1}{2} + \log\left( 1 + \left\lfloor \frac{q}{2} \right\rfloor \right) - \log(2) - \int_{1}^{\left\lfloor \frac{q}{2} \right\rfloor} (\ell - \lfloor \ell \rfloor) (1+\ell)^{-2} \ d\ell . \numberthis \label{eq:logsumTrqn}
\end{align*}
Because $0 \le \ell - \lfloor \ell \rfloor\le 1$, the integral in \eqref{eq:logsumTrqn} is bounded (in absolute value) by
\begin{align*}
\left| \int_{1}^{\left\lfloor \frac{q}{2} \right\rfloor} (1+\ell)^{-2} \ d\ell \right| &= \left| \frac{1}{2} - \left(1+ \left\lfloor \frac{q}{2} \right\rfloor \right)^{-1} \right| \le \frac{1}{2} .
\end{align*}
Therefore,
\begin{align}
\sum_{\ell=1}^{\left\lfloor \frac{q}{2} \right\rfloor} (1+\ell)^{-1} &= \log\left( 1 + \left\lfloor \frac{q}{2} \right\rfloor \right) + O(1) = \log\left( 1 + \frac{q}{2} \right) + O(1) \label{eq:logsumTrqn bound}
\end{align}
since $\left| \log\left( 1 + \frac{q}{2} \right) - \log\left( 1 + \left\lfloor \frac{q}{2} \right\rfloor \right) \right| \le \frac{1}{1 + \lfloor q/2 \rfloor} \le 1$.

Substituting \eqref{eq:logsumTrqn bound} into \eqref{ineq:Trqn}, we see that 
\begin{align*}
| T_{\bvec{r}} (q, \numrepresented; x) | &\le 
	q^{(\numvars+1)/2} q_0^{1/2} \tau(q) (\gcd(L, q_0))^{\numvars/2} (\gcd(\numrepresented , q_1))^{1/2} \left( 2 \log\left( 1 + \frac{q}{2} \right) + O(1) \right) .
\end{align*}
We obtain the result of this lemma by noticing that
\begin{align*}
2 \log\left( 1 + \frac{q}{2} \right) + O(1) \ll \log(2q)
\end{align*}
 for $q \ge 1$.
\end{proof}

\section{Analyzing the archimedean part} \label{chapter:archimedeanpart}

In this section, we analyze the archimedean part $\mathcal{I}_{F, \Upspsi} (x, \coordsize, \bvec{r}, q)$ using a principle of nonstationary phase and other bounds on certain oscillatory integrals. 
An \emph{oscillatory integral} is an integral of the form 
\begin{align}
\int_{\RR^\numvars} \e{f(\bvec{m})} \psi(\bvec{m}) \ d \bvec{m} , \label{expr:oscillatoryint}
\end{align}
where $f \in C^\infty (\RR^\numvars)$ and $\psi \in C_c^\infty (\RR^\numvars)$.

Typically, a principle of nonstationary phase says that the oscillatory integral in \eqref{expr:oscillatoryint} is relatively small in absolute value if the gradient of $f$ is nonzero for all $\bvec{m} \in \supp(\psi)$.
In Subsection~\ref{sec:PNSP}, we describe a one-dimensional version of the principle of nonstationary phase. 

In Subsection~\ref{sec:oscillatoryintbound}, we look at oscillatory integrals in which the function $f$ in \eqref{expr:oscillatoryint} is a quadratic polynomial. In that subsection, we give an upper bound for the absolute value of such an oscillatory integral.

In Subsection~\ref{sec:applyoscillatoryintbounds}, we use the results of Subsections~\ref{sec:PNSP} and \ref{sec:oscillatoryintbound} to obtain bounds for the archimedean part $\mathcal{I}_{F, \Upspsi} (x, \coordsize, \bvec{r}, q)$.

\subsection{A principle of nonstationary phase} \label{sec:PNSP}

To analyze the archimedean part $\mathcal{I}_{F, \Upspsi} (x, \coordsize, \bvec{r}, q)$, we will use a one-dimensional principle of nonstationary phase. We will use a principle of nonstationary phase that is different than what is found in Proposition~1 in Chapter~8 of \cite{SteinHarmonicAnalysis}, Proposition~6.1 of \cite{WolffHarmonicAnalysis}, Lemma~2.6 in \cite{TaoOscillatoryIntegralsNotes}, and Proposition~B.1 in \cite{GreenRestrictKakeyaPhenom}. (The results in \cite{SteinHarmonicAnalysis}, \cite{WolffHarmonicAnalysis}, \cite{TaoOscillatoryIntegralsNotes}, and \cite{GreenRestrictKakeyaPhenom} are not directly applicable for our purposes.)

The next theorem is our one-dimensional version of the principle of \mbox{nonstationary} phase and is a slight generalization of the statement 
\mbox{``Non-stationary} phase'' on p.~94 in \cite{ZhangL-G3ACP}. 
\begin{thm}[Principle of nonstationary phase in one variable]
 \label{thm:PNSP1d}
Let $\psi \in C_c^\infty (\RR)$ and let $M \ge 0$. 
Let $f \in C^\infty (\RR)$ be such that $|f'(x)| \ge B > 0$ and $| f^{(j)} (x) | \le |f'(x)|$ for all $x \in \supp(\psi)$ and for each integer $j$ satisfying $2 \le j \le \lceil M \rceil$. 
Then
\begin{align*}
\int_{\RR} \e{f(x)} \psi(x) \ d x \ll_{\psi, M} B^{-M}.
\end{align*}
\end{thm}
Theorem~\ref{thm:PNSP1d} follows from multiple applications of integration by parts.

\begin{rmk}
A similar result to Theorem~\ref{thm:PNSP1d} can be if shown if $\psi$ is a Schwartz function (not just a bump function). However, we just require that $\psi$ is a bump function for our purposes.
\end{rmk}

\subsection{Bounding an oscillatory integral} \label{sec:oscillatoryintbound}

In this subsection, we give an upper bound for the absolute value of an oscillatory integral associated with a quadratic polynomial.
The following theorem generalizes the statement ``Stationary phase'' on p.~95 of \cite{ZhangL-G3ACP} to quadratic polynomials in any positive number of variables. 
\begin{thm}
 \label{thm:quadoscillatorybound}
Suppose that $A$ is the Hessian matrix of a nonsingular quadratic form in $\numvars$ variables. Suppose that $\bvec{b} \in \RR^\numvars$, $c \in \RR$, and $\psi \in C_c^\infty (\RR^\numvars)$. 
Then
\begin{align}
\int_{\RR^\numvars} \e{\frac{1}{2} \bvec{x}^\top A \bvec{x} + \bvec{b} \cdot \bvec{x} + c} \psi(\bvec{x}) \ d \bvec{x} &\ll_\psi |\det(A)|^{-1/2} . \label{ineq:quadPSP}
\end{align}
\end{thm}
\begin{rmk}
With the hypotheses of Theorem~\ref{thm:quadoscillatorybound}, one can prove that 
\begin{align*}
&\int_{\RR^\numvars} \e{\frac{1}{2} \bvec{x}^\top A \bvec{x} + \bvec{b} \cdot \bvec{x} + c} \psi(\bvec{x}) \ d \bvec{x} \\
	&= \e{c - \frac{1}{2} \bvec{b}^\top A^{-1} \bvec{b}} \psi(-A^{-1} \bvec{b}) e^{i \pi \sgnt(A)/4} |\det(A)|^{-1/2} \\
		&\qquad + O_{\psi} (|\det(A)|^{-1/2} {\sigma_\numvars}^{-1}) , \numberthis \label{eq:quadPSP asymptotic}
\end{align*}
where $\sigma_\numvars$ is the smallest singular value of $A$. 
Because this is an asymptotic for an oscillatory integral that depends on when the gradient of $\frac{1}{2} \bvec{x}^\top A \bvec{x} + \bvec{b} \cdot \bvec{x} + c$ is zero,
this result could be called a principle of stationary phase for quadratic polynomials.
However, 
we will not need this result for our purposes.
\end{rmk}

Before we prove Theorem~\ref{thm:quadoscillatorybound}, we prove a number of lemmas that will be used in our proof of Theorem~\ref{thm:quadoscillatorybound}.
Our first lemma gives the Fourier transform of the one-dimensional Gaussian function $e^{- a x^2}$ with $a > 0$. 
\begin{lem} \label{lem:expFourier1d}
For $a > 0$ and $y \in \CC$, we have 
\begin{align}
\int_\RR e^{- a x^2} \e{- x y} \ d x = \sqrt{\frac{\pi}{a}} e^{- \pi^2 y^2 / a} . \label{eq:expFourier1d}
\end{align}
\end{lem}
\begin{proof}
It is well-known that 
\begin{align}
\int_\RR e^{- \pi x^2} \e{- x y} \ d x = e^{- \pi y^2} . \label{eq:FourierTransform exppix2}
\end{align}
(For example, see Example~1 and Exercise~4 of Chapter~2 of \cite{SSComplexAnalysis}.) Now
\begin{align*}
\int_\RR e^{- a x^2} \e{- x y} \ d x &= \int_\RR e^{- \pi \left(\sqrt{\frac{a}{\pi}} x \right)^2} \e{- \left(\sqrt{\frac{a}{\pi}} x\right) \left(\sqrt{\frac{\pi}{a}} y \right)} \ d x .
\end{align*}
Let $u = \sqrt{\frac{a}{\pi}} x$. Then
\begin{align*}
\int_\RR e^{- a x^2} \e{- x y} \ d x &= \sqrt{\frac{\pi}{a}} \int_\RR e^{- \pi u^2} \e{- u \sqrt{\frac{\pi}{a}} y} \ d u .
\end{align*}
We use \eqref{eq:FourierTransform exppix2} with the previous equation to conclude the result of the lemma.
\end{proof}

Fourier transform of the one-dimensional Gaussian function allows us to prove the following lemma about the Fourier transform of a multi-dimensional Gaussian function.
\begin{lem} \label{lem:expFouriermultid}
Suppose that $\bvec{a} = (a_1, \ldots , a_\numvars)^\top \in \RR^\numvars$ be such that $a_j > 0$ for each $1 \le j \le \numvars$. Then 
\begin{align}
\int_{\RR^\numvars} e^{-\sum_{j=1}^\numvars a_j x_j^2} \e{- \bvec{x} \cdot \bvec{y}} \ d x &= e^{-\sum_{j=1}^\numvars \pi^2 y_j^2 / a_j} \prod_{j=1}^\numvars \frac{\pi^{1/2}}{a_j^{1/2}} \label{eq:expFouriermultid}
\end{align}
for $\bvec{y} \in \CC^\numvars$.
\end{lem}
\begin{proof}
Using the Fubini-Tonelli theorem, we obtain 
\begin{align*}
\int_{\RR^\numvars} e^{-\sum_{j=1}^\numvars a_j x_j^2} \e{- \bvec{x} \cdot \bvec{y}} \ d x &= \prod_{j=1}^\numvars \int_{\RR} e^{- a_j x_j^2} \e{-x_j y_j} \ d x_j .
\end{align*}
By Lemma~\ref{lem:expFourier1d}, the integral becomes
\begin{align*}
\int_{\RR^\numvars} e^{-\sum_{j=1}^\numvars z_j x_j^2} \e{- \bvec{x} \cdot \bvec{y}} \ d x &= \prod_{j=1}^\numvars \sqrt{\frac{\pi}{a_j}} e^{- \pi^2 y_j^2 / a_j} ,
\end{align*}
which is equal to \eqref{eq:expFouriermultid}.
\end{proof}

We can now apply Plancherel's theorem to obtain the following result.
\begin{lem} \label{lem:Plancherel expbump}
Let $D$ be an $\numvars \times \numvars$ diagonal matrix with real entries and rank~$\numvars$. Suppose that $\psi \in C_c^\infty (\RR^\numvars)$. Then 
\begin{align*}
\int_{\RR^\numvars} \e{\frac{1}{2} \bvec{x}^\top D \bvec{x}} \psi(\bvec{x}) \ d \bvec{x} &= e^{i \pi \sgnt(D)/4} |\det(D)|^{-1/2} \int_{\RR^\numvars} \e{- \frac{1}{2} \bvec{x}^\top D^{-1} \bvec{x}} \widehat{\psi}(\bvec{x}) \ d \bvec{x} .
\end{align*}
\end{lem}
\begin{proof}
Plancherel's theorem states that 
\begin{align*}
\int_{\RR^\numvars} f(x) \lbar{g(x)} \ d x = \int_{\RR^\numvars} \widehat{f}(y) \lbar{\widehat{g}(y)} \ d y
\end{align*}
if $f, g \in L^1 (\RR^\numvars) \cap L^2 (\RR^\numvars)$. Let $\bvec{z} = (z_1, \ldots , z_\numvars)^\top \in \RR^\numvars$ be such that $z_j > 0$ for each $1 \le j \le \numvars$.
Plancherel's theorem with $f(\bvec{x}) = \psi(\bvec{x})$ and $g(\bvec{x}) = e^{-\sum_{j=1}^\numvars z_j x_j^2}$ gives us 
\begin{align}
\int_{\RR^\numvars} e^{-\sum_{j=1}^\numvars z_j x_j^2} \psi(\bvec{x}) \ d \bvec{x} &= \int_{\RR^\numvars} e^{-\sum_{j=1}^\numvars \pi^2 x_j^2 / z_j} \widehat{\psi}(\bvec{x}) \ d \bvec{x} \prod_{j=1}^\numvars \frac{\pi^{1/2}}{z_j^{1/2}} . \label{eq:ebump Plancherel}
\end{align}
(We used Lemma~\ref{lem:expFouriermultid} to compute the Fourier transform of $e^{-\sum_{j=1}^\numvars z_j x_j^2}$.)

The left-hand side of \eqref{eq:ebump Plancherel} can be extended to a analytic function for all $\bvec{z} \in \CC^\numvars$. 
Using the Lebesgue dominated convergence theorem, the right-hand side can be extended to a continuous function on 
\begin{align*}
S = \{ \bvec{z} \in \CC^\numvars : \bvec{z} \ne \bvec{0}, \Re(z_j) \ge 0 \text{ for all $1 \le j \le \numvars$}\}
\end{align*}
 that is analytic on the interior of $S$. 
It follows from the identity principle that \eqref{eq:ebump Plancherel} holds for all $\bvec{z} \in S$. 
(For the square root function, we use the principal branch cut of the logarithm function along the nonpositive real axis.) 
In particular, \eqref{eq:ebump Plancherel} holds when each $z_j = - i \pi d_j$, where $D = \diag(d_1 , \ldots , d_\numvars)$. Therefore, 
\begin{align*}
\int_{\RR^\numvars} \e{\frac{1}{2} \bvec{x}^\top D \bvec{x}} \psi(\bvec{x}) \ d \bvec{x} &= \int_{\RR^\numvars} e^{-\sum_{j=1}^\numvars \pi^2 x_j^2 / (- i \pi d_j)} \widehat{\psi}(\bvec{x}) \ d \bvec{x} \prod_{j=1}^\numvars \frac{\pi^{1/2}}{(- i \pi d_j)^{1/2}} \\
	&= \int_{\RR^\numvars} e^{- i \pi \sum_{j=1}^\numvars x_j^2 / d_j} \widehat{\psi}(\bvec{x}) \ d \bvec{x} \prod_{j=1}^\numvars (- i d_j)^{-1/2} \\
	&= e^{i \pi \sgnt(D)/4} |\det(D)|^{-1/2} \int_{\RR^\numvars} \e{- \frac{1}{2} \bvec{x}^\top D^{-1} \bvec{x}} \widehat{\psi}(\bvec{x}) \ d \bvec{x} . \qedhere
\end{align*}
\end{proof}

We are now in a position to give a proof for Theorem~\ref{thm:quadoscillatorybound}.
\begin{proof}[Proof of Theorem~\ref{thm:quadoscillatorybound}]
Using the spectral theorem for symmetric matrices, we can write the symmetric matrix $A$ as 
\begin{align*}
A = P^\top D P,
\end{align*}
where $P$ is an orthogonal matrix and $D$ is a diagonal matrix. Therefore, $P^\top = P^{-1}$, and 
\begin{align*}
&\int_{\RR^\numvars} \e{\frac{1}{2} \bvec{x}^\top A \bvec{x} + \bvec{b} \cdot \bvec{x} + c} \psi(\bvec{x}) \ d \bvec{x} \\
	&= \int_{\RR^\numvars} \e{\frac{1}{2} \bvec{x}^\top P^\top D P \bvec{x} + \bvec{b}^\top P^\top P \bvec{x} + c} \psi(P^{-1} P \bvec{x}) \ d \bvec{x} \\
	&= \e{c} \int_{\RR^\numvars} \e{\frac{1}{2} (P \bvec{x})^\top D P \bvec{x} + (P \bvec{b})^\top P \bvec{x}} \psi(P^{-1} P \bvec{x}) \ d \bvec{x} .
\end{align*}

Let $\bvec{y} = P \bvec{x}$. Because $P$ is orthogonal, we know that $|\det(P)| = 1$. Thus,
\begin{align*}
&\int_{\RR^\numvars} \e{\frac{1}{2} \bvec{x}^\top A \bvec{x} + \bvec{b} \cdot \bvec{x} + c} \psi(\bvec{x}) \ d \bvec{x} \\
	&= \e{c} \int_{\RR^\numvars} \e{\frac{1}{2} \bvec{y}^\top D \bvec{y} + (P \bvec{b})^\top \bvec{y}} \psi(P^{-1} \bvec{y}) \frac{1}{|\det(P)|} \ d \bvec{y} \\
	&= \e{c} \int_{\RR^\numvars} \e{\frac{1}{2} \bvec{y}^\top D \bvec{y} + (P \bvec{b})^\top \bvec{y}} \psi(P^{-1} \bvec{y}) \ d \bvec{y} .
\end{align*}
We now complete the square and see that 
\begin{align*}
&\int_{\RR^\numvars} \e{\frac{1}{2} \bvec{x}^\top A \bvec{x} + \bvec{b} \cdot \bvec{x} + c} \psi(\bvec{x}) \ d \bvec{x} \\
	&= \e{c - \frac{1}{2} \bvec{b}^\top P^\top D^{-1} P \bvec{b}} \int_{\RR^\numvars} \e{\frac{1}{2} (\bvec{y} + D^{-1} P \bvec{b})^\top D (\bvec{y} + D^{-1} P \bvec{b})} \psi(P^{-1} \bvec{y}) \ d \bvec{y} .
\end{align*}

Let $\bvec{z} = \bvec{y} + D^{-1} P \bvec{b}$. Then
\begin{align*}
&\int_{\RR^\numvars} \e{\frac{1}{2} \bvec{x}^\top A \bvec{x} + \bvec{b} \cdot \bvec{x} + c} \psi(\bvec{x}) \ d \bvec{x} \\
	&= \e{c - \frac{1}{2} \bvec{b}^\top P^\top D^{-1} P \bvec{b}} \int_{\RR^\numvars} \e{\frac{1}{2} \bvec{z}^\top D \bvec{z}} \psi(P^{-1} \bvec{z} - P^{-1} D^{-1} P \bvec{b}) \ d \bvec{z} . \numberthis \label{eq:quadint completedsquare}
\end{align*}

Now 
\begin{align}
A^{-1} 
	&= P^{-1} D^{-1} P = P^\top D^{-1} P \label{eq:A-1 PTD-1P}
\end{align}
since $P^{-1} = P^\top$.
We apply \eqref{eq:A-1 PTD-1P} to \eqref{eq:quadint completedsquare} and obtain
\begin{align*}
&\int_{\RR^\numvars} \e{\frac{1}{2} \bvec{x}^\top A \bvec{x} + \bvec{b} \cdot \bvec{x} + c} \psi(\bvec{x}) \ d \bvec{x} \\
	&= \e{c - \frac{1}{2} \bvec{b}^\top A^{-1} \bvec{b}} \int_{\RR^\numvars} \e{\frac{1}{2} \bvec{z}^\top D \bvec{z}} \psi(P^{-1} \bvec{z} - A^{-1} \bvec{b}) \ d \bvec{z} \\
	&= \e{c - \frac{1}{2} \bvec{b}^\top A^{-1} \bvec{b}} \int_{\RR^\numvars} \e{\frac{1}{2} \bvec{z}^\top D \bvec{z}} \varphi(\bvec{z}) \ d \bvec{z} , \numberthis \label{eq:quadint nice}
\end{align*}
where $\varphi(\bvec{z}) = \psi(P^{-1} \bvec{z} - A^{-1} \bvec{b})$ for $\bvec{z} \in \RR^\numvars$.

We apply Lemma~\ref{lem:Plancherel expbump} to \eqref{eq:quadint nice} to obtain
\begin{align*}
&\int_{\RR^\numvars} \e{\frac{1}{2} \bvec{x}^\top A \bvec{x} + \bvec{b} \cdot \bvec{x} + c} \psi(\bvec{x}) \ d \bvec{x} \\
	&= \e{c - \frac{1}{2} \bvec{b}^\top A^{-1} \bvec{b}} e^{i \pi \sgnt(D)/4} |\det(D)|^{-1/2} \int_{\RR^\numvars} \e{- \frac{1}{2} \bvec{z}^\top D^{-1} \bvec{z}} \widehat{\varphi}(\bvec{z}) \ d \bvec{z} . \numberthis \label{eq:quadint Plancherel}
\end{align*}

Before we bound \eqref{eq:quadint Plancherel}, we compute $\widehat{\varphi}$ in terms of $A$ and $\bvec{b}$.
\begin{lem} \label{lem:phiFourierTransform}
Suppose that $\varphi(\bvec{z}) = \psi(P^{-1} \bvec{z} - A^{-1} \bvec{b})$, where $A$, $P$, $\bvec{b}$, and $\psi$ are as above. Then
\begin{align}
\widehat{\varphi}(\bvec{w}) = \e{-(P A^{-1} \bvec{b})^\top \bvec{w}} \widehat{\psi}(P^\top \bvec{w}) \label{eq:phiFourierTransform}
\end{align}
for $\bvec{w} \in \RR^\numvars$.
\end{lem}
\begin{proof}[Proof of Lemma~\ref{lem:phiFourierTransform}]
By definition, the Fourier transform of $\varphi$ is
\begin{align*}
\widehat{\varphi}(\bvec{w}) &= \int_{\RR^\numvars} \psi(P^{-1} \bvec{z} - A^{-1} \bvec{b}) \e{- \bvec{z} \cdot \bvec{w}} \ d \bvec{z} .
\end{align*}
Let $\bvec{v} = P^{-1} \bvec{z} - A^{-1} \bvec{b}$. Then $\bvec{z} = P \bvec{v} + P A^{-1} \bvec{b}$, and 
\begin{align*}
\widehat{\varphi}(\bvec{w}) &= \int_{\RR^\numvars} \psi(\bvec{v}) \e{-(P \bvec{v} + P A^{-1} \bvec{b})^\top \bvec{w}} |\det(P)| \ d \bvec{v} \\
	&= \e{-(P A^{-1} \bvec{b})^\top \bvec{w}} \int_{\RR^\numvars} \psi(\bvec{v}) \e{- \bvec{v}^\top P^\top \bvec{w}} \ d \bvec{v} 
\end{align*}
since $P$ is orthogonal and $|\det(P)| = 1$. 
We obtain \eqref{eq:phiFourierTransform} by noting that 
\begin{align*}
\int_{\RR^\numvars} \psi(\bvec{v}) \e{- \bvec{v}^\top P^\top \bvec{w}} \ d \bvec{v} = \widehat{\psi}(P^\top \bvec{w}) . 
\end{align*}
\end{proof}

We are now able to estimate $\int_{\RR^\numvars} \e{\frac{1}{2} \bvec{x}^\top A \bvec{x} + \bvec{b} \cdot \bvec{x} + c} \psi(\bvec{x}) \ d \bvec{x}$. Using Lemma~\ref{lem:phiFourierTransform} in \eqref{eq:quadint Plancherel}, we obtain 
\begin{align*}
&\int_{\RR^\numvars} \e{\frac{1}{2} \bvec{x}^\top A \bvec{x} + \bvec{b} \cdot \bvec{x} + c} \psi(\bvec{x}) \ d \bvec{x} \\
	&= \e{c - \frac{1}{2} \bvec{b}^\top A^{-1} \bvec{b}} e^{i \pi \sgnt(D)/4} |\det(D)|^{-1/2} \\
		&\qquad \times \int_{\RR^\numvars} \e{- \frac{1}{2} \bvec{z}^\top D^{-1} \bvec{z}} \e{-(P A^{-1} \bvec{b})^\top \bvec{z}} \widehat{\psi}(P^\top \bvec{z}) \ d \bvec{z} .
\end{align*}
We take absolute values of both sides and see that
\begin{align*}
\left|\int_{\RR^\numvars} \e{\frac{1}{2} \bvec{x}^\top A \bvec{x} + \bvec{b} \cdot \bvec{x} + c} \psi(\bvec{x}) \ d \bvec{x} \right| 
	&\le |\det(D)|^{-1/2} \int_{\RR^\numvars} \left| \widehat{\psi}(P^\top \bvec{z}) \right| \ d \bvec{z} . \numberthis \label{eq:quadint boundwithconstants}
\end{align*}

Let $\bvec{w} = P^\top \bvec{z}$. Then $\bvec{z} = P \bvec{w}$ (because $P^{-1} = P^\top$), and 
\begin{align*}
\int_{\RR^\numvars} \left| \widehat{\psi}(P^\top \bvec{z}) \right| \ d \bvec{z} &= \int_{\RR^\numvars} \left| \widehat{\psi}(\bvec{w}) \right| |\det(P)| \ d \bvec{w} \\
	&= \int_{\RR^\numvars} \left| \widehat{\psi}(\bvec{w}) \right| \ d \bvec{w} \numberthis \label{eq:psihatL1bound}
\end{align*}
since $|\det(P)| = 1$. Substituting \eqref{eq:psihatL1bound} into \eqref{eq:quadint boundwithconstants}, we obtain
\begin{align*}
\left|\int_{\RR^\numvars} \e{\frac{1}{2} \bvec{x}^\top A \bvec{x} + \bvec{b} \cdot \bvec{x} + c} \psi(\bvec{x}) \ d \bvec{x} \right| &\le |\det(D)|^{-1/2} \int_{\RR^\numvars} \left| \widehat{\psi}(\bvec{w}) \right| \ d \bvec{w} .
\end{align*}
Because $A = P^\top D P$ and $|\det(P)| = 1$, we know that $\det(A) = \det(D)$. Therefore,
\begin{align*}
\left|\int_{\RR^\numvars} \e{\frac{1}{2} \bvec{x}^\top A \bvec{x} + \bvec{b} \cdot \bvec{x} + c} \psi(\bvec{x}) \ d \bvec{x} \right| &\le |\det(A)|^{-1/2} \int_{\RR^\numvars} \left| \widehat{\psi}(\bvec{w}) \right| \ d \bvec{w} ,
\end{align*}
which shows that \eqref{ineq:quadPSP} is true with an implied constant of $\int_{\RR^\numvars} \left| \widehat{\psi}(w) \right| \ d w$.
\end{proof}

\subsection{Applying bounds for oscillatory integrals} \label{sec:applyoscillatoryintbounds}

In this subsection, we apply the results of Subsections~\ref{sec:PNSP} and \ref{sec:oscillatoryintbound} to the archimedean part $\mathcal{I}_{F, \Upspsi} (x, \coordsize, \bvec{r}, q)$.
To make the results of Subsections~\ref{sec:PNSP} and \ref{sec:oscillatoryintbound} applicable to the archimedean part, we normalize our bump function so that it remains same regardless the value of $\coordsize$. We do this by applying a change a variables ($\bvec{m} \mapsto \coordsize \bvec{m}$) in \eqref{eq:IFx\coordsize rq} obtain the following:
\begin{align*}
\mathcal{I}_{F, \Upspsi} (x, \coordsize, \bvec{r}, q) &= \coordsize^\numvars \int_{\RR^\numvars} \e{x F(\coordsize \bvec{m}) - \frac{1}{q} \coordsize \bvec{m} \cdot \bvec{r}} \Upspsi_\coordsize (\coordsize \bvec{m}) \ d \bvec{m} \\
	&= \coordsize^\numvars \int_{\RR^\numvars} \e{x F(\coordsize \bvec{m}) - \frac{1}{q} \coordsize \bvec{m} \cdot \bvec{r}} \Upspsi_1 \left( \frac{\coordsize}{\coordsize} \bvec{m} \right) \ d \bvec{m} \\
	&= \coordsize^\numvars \int_{\RR^\numvars} \e{\coordsize^2 x F(\bvec{m}) - \frac{1}{q} \coordsize \bvec{m} \cdot \bvec{r}} \Upspsi(\bvec{m}) \ d \bvec{m} . \numberthis \label{eq:IFx\coordsize rq const bump}
\end{align*}

The following lemma gives a trivial bound for $\mathcal{I}_{F, \Upspsi} (x, \coordsize, \bvec{r}, q)$. 
\begin{lem} \label{lem:IFx\coordsize rq trivial bound}
Suppose that $x \in \RR$, $\coordsize >0$, $\bvec{r} \in \ZZ^\numvars$, and $q$ is a positive integer. Then 
\begin{align}
\mathcal{I}_{F, \Upspsi} (x, \coordsize, \bvec{r}, q) &\ll_{\Upspsi} \coordsize^\numvars . \label{ineq:IFx\coordsize rq trivial bound}
\end{align}
\end{lem}
\begin{proof}
By \eqref{eq:IFx\coordsize rq const bump}, we know that 
\begin{align*}
|\mathcal{I}_{F, \Upspsi} (x, \coordsize, \bvec{r}, q)| &\le \coordsize^\numvars \int_{\RR^\numvars}| \Upspsi(\bvec{m}) | \ d \bvec{m} \ll_{\Upspsi} \coordsize^\numvars . \qedhere 
\end{align*}
\end{proof}

We will need less trivial bounds for $\mathcal{I}_{F, \Upspsi} (x, \coordsize, \bvec{r}, q)$. 
We can apply Theorem~\ref{thm:quadoscillatorybound} to the integral in \eqref{eq:IFx\coordsize rq const bump} to obtain the following bound on $\mathcal{I}_{F, \Upspsi} (x, \coordsize, \bvec{r}, q)$.
\begin{thm} \label{thm:IFx\coordsize rq PSP bound}
Suppose that $x \in \RR$, $\coordsize >0$, $\bvec{r} \in \ZZ^\numvars$, and $q$ is a positive integer. Then 
\begin{align}
\mathcal{I}_{F, \Upspsi} (x, \coordsize, \bvec{r}, q) &\ll_{\Upspsi} \min\left\{ \coordsize^\numvars , |x|^{-\numvars/2} |\det(A)|^{-1/2} \right\} . \label{ineq:IFx\coordsize rq PSP bound combo}
\end{align}
\end{thm}
\begin{proof}
By Theorem~\ref{thm:quadoscillatorybound}, the integral in \eqref{eq:IFx\coordsize rq const bump} is 
\begin{align*}
\int_{\RR^\numvars} \e{\coordsize^2 x F(\bvec{m}) - \frac{1}{q} \coordsize \bvec{m} \cdot \bvec{r}} \Upspsi(\bvec{m}) \ d \bvec{m} &\ll_{\Upspsi} | \det(\coordsize^2 x A) |^{-1/2} \\
	&= | \coordsize^{2 \numvars} x^\numvars \det(A) |^{-1/2} . \numberthis \label{ineq:IFx\coordsize rq PSP bound1}
\end{align*}
We notice that $\coordsize>0$, so \eqref{ineq:IFx\coordsize rq PSP bound1} implies that 
\begin{align*}
\int_{\RR^\numvars} \e{\coordsize^2 x F(\bvec{m}) - \frac{1}{q} \coordsize \bvec{m} \cdot \bvec{r}} \Upspsi(\bvec{m}) \ d \bvec{m} &\ll_{\Upspsi} \coordsize^{- \numvars} |x|^{-\numvars/2} |\det(A)|^{-1/2} .
\end{align*}
We multiply by $\coordsize^\numvars$ to obtain 
\begin{align}
\mathcal{I}_{F, \Upspsi} (x, \coordsize, \bvec{r}, q) &\ll_{\Upspsi} |x|^{-\numvars/2} |\det(A)|^{-1/2} . \label{ineq:IFx\coordsize rq PSP bound2}
\end{align}

The statement \eqref{ineq:IFx\coordsize rq PSP bound combo} is obtained by taking the minimum of \eqref{ineq:IFx\coordsize rq trivial bound} and \eqref{ineq:IFx\coordsize rq PSP bound2}. 
\end{proof}

To obtain better estimates on sums and integrals involving $\mathcal{I}_{F, \Upspsi} (x, \coordsize, \bvec{r}, q)$, we need to determine when we can apply Theorem~\ref{thm:PNSP1d}. To do this, we use partial derivatives and directional derivatives. We use the notation $\frac{\partial f}{\partial x_j}(\bvec{m})$ to denote the the partial derivative in the $j$th coordinate evaluated at the point $\bvec{m}$. We now define a directional derivative.

\begin{defn}[Directional derivative] \label{defn:directionalderiv}
For a unit vector $\bvec{u} \in \RR^\numvars$ and a differentiable function $f \colon \RR^\numvars \to \RR$, 
define the directional derivative $\nabla_{\bvec{u}} f$ of $f$ along $\bvec{u}$ to be
\begin{align*}
\nabla_{\bvec{u}} f = \bvec{u} \cdot (\nabla f) ,
\end{align*}
where $\nabla f$ is the gradient of $f$.
\end{defn}
\begin{rmk}
The quantity $\frac{\partial f}{\partial x_j}(\bvec{m})$ is equal to $\nabla_{\bvec{e}_j} f(\bvec{m})$, where $\bvec{e}_j \in \RR^\numvars$ is the unit vector whose $j$th entry is equal to $1$ and is the only nonzero entry of $\bvec{e}_j$.
\end{rmk}

Directional derivatives allow us to take repeatedly the derivative of a function $f$ in a particular direction. This is needed for Theorem~\ref{thm:PNSP1d} to apply. 
We use directional derivatives, partial derivatives, and Theorem~\ref{thm:PNSP1d} to prove the following theorem.
\begin{thm} \label{thm:IFx\coordsize rq PNSP bound in terms of r}
Suppose that there exists an integer $j$ with $1 \le j \le \numvars$ such that 
\begin{align}
| r_j | \ge q \coordsize |x| \sigma_1 (\Upsuppvar_\Upspsi + 1) . \label{ineq:rjPNSP cond}
\end{align}
Then 
\begin{align}
\mathcal{I}_{F, \Upspsi} (x, \coordsize, \bvec{r}, q) &\ll_{M, \Upspsi} \min \left\{ \coordsize^\numvars , \coordsize^{\numvars-M} \left( \frac{1}{q (\Upsuppvar_\Upspsi + 1) \sqrt{\numvars}} \| \bvec{r} \| \right)^{-M} \right\}
\end{align}
for all $M \ge 0$. 
\end{thm}

The proof of Theorem~\ref{thm:IFx\coordsize rq PNSP bound in terms of r} uses the following lemma that gives an upper bound for $\bvec{u}^\top B \bvec{u}$ when $B$ is a symmetric matrix and $\bvec{u}$ is a unit vector. Proofs of this lemma can be found in linear algebra textbooks, such as \cite{LeonLinAlg}. (See Theorem~7.4.3 in \cite{LeonLinAlg}.)
\begin{lem} \label{lem:wtBw upperbound}
Suppose $\numvars$ is a positive integer. 
If $B \in \Mat_\numvars (\RR)$ is a symmetric $\numvars \times \numvars$ matrix and $\bvec{w} \in \RR^\numvars$, then
\begin{align*}
| \bvec{w}^\top B \bvec{w} | \le \sigma_1 \| \bvec{w} \|^2 .
\end{align*}
\end{lem}

Now that we have Lemma~\ref{lem:wtBw upperbound}, we prove Theorem~\ref{thm:IFx\coordsize rq PNSP bound in terms of r}.
\begin{proof}[Proof of Theorem~\ref{thm:IFx\coordsize rq PNSP bound in terms of r}]
Suppose that $\bvec{r}$ has an entry $r_j$ that satisfies \eqref{ineq:rjPNSP cond}. Set $j$ with $1 \le j \le \numvars$ to be such that $r_j \ge r_k$ for all $1 \le k \le \numvars$. Then $r_j$ satisfies \eqref{ineq:rjPNSP cond} and 
\begin{align}
| r_j | &\ge \frac{1}{\sqrt{\numvars}} \| \bvec{r} \| . \label{ineq:rj&r}
\end{align}

Lemma~\ref{lem:IFx\coordsize rq trivial bound} says that 
\begin{align*}
\mathcal{I}_{F, \Upspsi} (x, \coordsize, \bvec{r}, q) &\ll_{\Upspsi} \coordsize^\numvars ,
\end{align*}
so it suffices to prove that 
\begin{align}
\mathcal{I}_{F, \Upspsi} (x, \coordsize, \bvec{r}, q) &\ll_{M, \Upspsi} \coordsize^{\numvars-M} \left( \frac{1}{q (\Upsuppvar_\Upspsi + 1) \sqrt{\numvars}} \| \bvec{r} \| \right)^{-M} . \label{ineq:IFx\coordsize rq PNSP bound1}
\end{align}

From \eqref{eq:IFx\coordsize rq const bump}, we have 
\begin{align*}
\mathcal{I}_{F, \Upspsi} (x, \coordsize, \bvec{r}, q) &= \coordsize^\numvars \int_{\RR^\numvars} \e{f(\bvec{m})} \Upspsi(\bvec{m}) \ d \bvec{m} , 
\end{align*}
where 
\begin{align}
f(\bvec{m}) &= \coordsize^2 x F(\bvec{m}) - \frac{1}{q} \coordsize \bvec{m} \cdot \bvec{r} . \label{eq:f in PNSParg}
\end{align}
Therefore, to prove \eqref{ineq:IFx\coordsize rq PNSP bound1}, it suffices to show that
\begin{align}
\int_{\RR^\numvars} \e{f(\bvec{m})} \Upspsi(\bvec{m}) \ d \bvec{m} &\ll_{M, \Upspsi} \left( \frac{\coordsize}{q (\Upsuppvar_\Upspsi + 1) \sqrt{\numvars}} \| \bvec{r} \| \right)^{-M} . \label{ineq:efbumpPNSPbound1}
\end{align}

We will use partial derivatives, directional derivatives, and Theorem~\ref{thm:PNSP1d} to prove \eqref{ineq:efbumpPNSPbound1}.
In order to apply Theorem~\ref{thm:PNSP1d}, we show that for all $\bvec{m} \in \supp(\Upspsi)$, we have 
\begin{align}
\left| \frac{\partial f}{\partial x_j}(\bvec{m}) \right| \ge \frac{\coordsize}{q (\Upsuppvar_\Upspsi + 1) \sqrt{\numvars}} \| \bvec{r} \| \label{ineq:PNSPffirstderiv bound}
\end{align}
and 
\begin{align*}
\left| \frac{\partial^k f}{\partial x_j^k}(\bvec{m}) \right| \le \left| \frac{\partial f}{\partial x_j}(\bvec{m}) \right|
\end{align*}
 for all $k \ge 2$.
Because $f$ is a quadratic polynomial in terms of the $m_j$, it suffices to show that for all $\bvec{m} \in \supp(\Upspsi)$, we have \eqref{ineq:PNSPffirstderiv bound} holding and 
\begin{align}
\left| \frac{\partial^2 f}{\partial x_j^2}(\bvec{m}) \right| \le \left| \frac{\partial f}{\partial x_j}(\bvec{m}) \right| . \label{ineq:PNSPfderivs bound}
\end{align}

Let $\bvec{u} \in \RR^\numvars$ be a unit vector. We would like to compute the directional derivative $\nabla_{\bvec{u}} f$. In order to do this, we need to know the gradient of $f$, which is stated in the next lemma that can be proved with a standard gradient computation.
\begin{lem} \label{lem:gradf in PNSParg}
Let $f$ be as defined in \eqref{eq:f in PNSParg}. 
The gradient of $f$ is 
\begin{align}
\nabla f(\bvec{m}) &= \coordsize^2 x A \bvec{m} - \frac{1}{q} \coordsize \bvec{r} . \label{eq:gradf in PNSParg}
\end{align}
In particular, the partial derivative $\frac{\partial f}{\partial x_j}(\bvec{m})$ in the $j$th coordinate is
\begin{align}
\frac{\partial f}{\partial x_j}(\bvec{m}) &= \coordsize^2 x \sum_{k=1}^{\numvars} a_{j k} m_k - \frac{1}{q} \coordsize r_j . \label{eq:jthpartialf in PNSParg}
\end{align}
\end{lem}
%

In light of Lemma~\ref{lem:gradf in PNSParg}, the directional derivative $\nabla_{\bvec{u}} f$ is 
\begin{align*}
\nabla_{\bvec{u}} f (\bvec{m}) &= \bvec{u} \cdot \left( \coordsize^2 x A \bvec{m} - \frac{1}{q} \coordsize \bvec{r} \right) \\
	&= \coordsize^2 x \bvec{u}^\top A \bvec{m} - \frac{1}{q} \coordsize \bvec{u} \cdot \bvec{r} . 
\end{align*}
Using the linearity of the gradient, we obtain 
\begin{align*}
\nabla (\nabla_{\bvec{u}} f) (\bvec{m}) &= \coordsize^2 x A \bvec{u} .
\end{align*}
Therefore, 
\begin{align}
(\nabla_{\bvec{u}})^2 f (\bvec{m}) &= \coordsize^2 x \bvec{u}^\top A \bvec{u} . \label{eq:dirderiv2uf}
\end{align}

Towards showing that \eqref{ineq:PNSPfderivs bound} holds, we prove an upper bound for $\left| (\nabla_{\bvec{u}})^2 f (\bvec{m}) \right|$ that does not depend on $\bvec{u}$ or $\bvec{m}$.
\begin{lem} \label{lem:dirderiv2uf upperbound}
For $f$ as in \eqref{eq:f in PNSParg}, we have 
\begin{align}
\left| (\nabla_{\bvec{u}})^2 f (\bvec{m}) \right| \le \coordsize^2 |x| \sigma_1 . \label{ineq:dirderiv2uf upperbound}
\end{align}
In particular,
\begin{align}
\left| \frac{\partial^2 f}{\partial x_j^2}(\bvec{m}) \right| \le \coordsize^2 |x| \sigma_1 . \label{ineq:partial2f upperbound}
\end{align}
\end{lem}
\begin{proof}[Proof of Lemma~\ref{lem:dirderiv2uf upperbound}]
An application of Lemma~\ref{lem:wtBw upperbound} 
to \eqref{eq:dirderiv2uf} proves \eqref{ineq:dirderiv2uf upperbound} since $\bvec{u}$ is a unit vector.
We obtain \eqref{ineq:partial2f upperbound} by noting that $\frac{\partial^2 f}{\partial x_j^2}(\bvec{m})$ equals $(\nabla_{\bvec{e}_j})^2 f (\bvec{m})$.
\end{proof}

We now begin to compute an lower bound for $\left| \frac{\partial f}{\partial x_j}(\bvec{m}) \right|$. 
By using the triangle inequality with \eqref{eq:jthpartialf in PNSParg} in Lemma~\ref{lem:gradf in PNSParg}, we obtain
\begin{align}
\left| \frac{\partial f}{\partial x_j}(\bvec{m}) \right| &\ge \frac{1}{q} \coordsize | r_j | - \coordsize^2 |x| \left| \sum_{k=1}^{\numvars} a_{j k} m_k \right| . \label{ineq:jthpartialf lowerbound1}
\end{align}
To effectively use this lower bound for $\left| \frac{\partial f}{\partial x_j}(\bvec{m}) \right|$, we need an upper bound for \\
$\left| \sum_{k=1}^{\numvars} a_{j k} m_k \right|$. The following lemma is a step towards finding such an upper bound.
\begin{lem} \label{lem:|Am| upperbound}
For $\bvec{m} \in \supp(\Upspsi)$, then
\begin{align}
\left\| A \bvec{m} \right\| &\le \sigma_1 \Upsuppvar_\Upspsi . \label{ineq:|Am| upperbound1}
\end{align}
\end{lem}
\begin{proof}[Proof of Lemma~\ref{lem:|Am| upperbound}]
By definition of the Eurclidean norm, we have
\begin{align*}
\left\| A \bvec{m} \right\|^2 &= \bvec{m}^\top A^\top A \bvec{m} .
\end{align*}
Because $A^\top A$ is a symmetric matrix, Lemma~\ref{lem:wtBw upperbound} applies, and we obtain 
\begin{align}
\left\| A \bvec{m} \right\|^2 &\le \| \bvec{m} \|^2 \sigma_1^2 . \label{ineq:|Am|^2 upperbound}
\end{align}
By taking square roots of both sides of \eqref{ineq:|Am|^2 upperbound}, we obtain 
\begin{align}
\left\| A \bvec{m} \right\| &\le \sigma_1 \| \bvec{m} \| . \label{ineq:|Am| upperbound2}
\end{align}
Because $\bvec{m} \in \supp(\Upspsi)$, we know that $\| \bvec{m} \| \le \Upsuppvar_\Upspsi$. Applying this to \eqref{ineq:|Am| upperbound2}, we obtain \eqref{ineq:|Am| upperbound1}.
\end{proof}

Since $\sum_{k=1}^{\numvars} a_{j k} m_k$ is the $j$th entry of the vector $A \bvec{m}$, we use the previous lemma to give an upper bound for $\left| \sum_{k=1}^{\numvars} a_{j k} m_k \right|$. 
We state this upper bound in the next lemma, which follows from the previous lemma and the fact that 
\begin{align*}
\left| \sum_{k=1}^{\numvars} a_{j k} m_k \right| \le \left\| A \bvec{m} \right\| .
\end{align*}
\begin{lem} \label{lem:|Amentry| upperbound}
For $\bvec{m} \in \supp(\Upspsi)$, then
\begin{align}
\left| \sum_{k=1}^{\numvars} a_{j k} m_k \right| &\le \sigma_1 \Upsuppvar_\Upspsi . \label{ineq:|Amentry| upperbound}
\end{align}
\end{lem}

We now apply Lemma~\ref{lem:|Amentry| upperbound} to \eqref{ineq:jthpartialf lowerbound1} to obtain
\begin{align}
\left| \frac{\partial f}{\partial x_j}(\bvec{m}) \right| &\ge \frac{1}{q} \coordsize | r_j | - \coordsize^2 |x| \sigma_1 \Upsuppvar_\Upspsi \label{ineq:jthpartialf lowerbound2}
\end{align}
for all $\bvec{m} \in \supp(\Upspsi)$.
We are now in a position to prove that \eqref{ineq:PNSPffirstderiv bound} and \eqref{ineq:PNSPfderivs bound} hold for all $\bvec{m} \in \supp(\Upspsi)$.
\begin{lem} \label{lem:PNSPffirstderiv bound}
For all $\bvec{m} \in \supp(\Upspsi)$, the statement \eqref{ineq:PNSPffirstderiv bound} holds.
\end{lem}
\begin{proof}[Proof of Lemma~\ref{lem:PNSPffirstderiv bound}]
Using \eqref{ineq:rjPNSP cond} in \eqref{ineq:jthpartialf lowerbound2}, we obtain
\begin{align*}
\left| \frac{\partial f}{\partial x_j}(\bvec{m}) \right| &\ge \frac{1}{q} \coordsize | r_j | - \frac{\Upsuppvar_\Upspsi}{q (\Upsuppvar_\Upspsi + 1)} \coordsize | r_j | 
	= \frac{1}{q (\Upsuppvar_\Upspsi + 1)} \coordsize | r_j | . \numberthis \label{ineq:jthpartialf lowerbound3}
\end{align*}
By applying \eqref{ineq:rj&r} to \eqref{ineq:jthpartialf lowerbound3}, we obtain \eqref{ineq:PNSPffirstderiv bound}.
\end{proof}

We now prove that \eqref{ineq:PNSPfderivs bound} holds for all $\bvec{m} \in \supp(\Upspsi)$.
\begin{lem} \label{lem:PNSPfderivs bound}
For all $\bvec{m} \in \supp(\Upspsi)$, the statement \eqref{ineq:PNSPfderivs bound} holds.
\end{lem}
\begin{proof}[Proof of Lemma~\ref{lem:PNSPffirstderiv bound}]
By applying \eqref{ineq:rjPNSP cond} to \eqref{ineq:jthpartialf lowerbound2}, we find that 
\begin{align*}
\left| \frac{\partial f}{\partial x_j}(\bvec{m}) \right| &\ge \coordsize^2 |x| (\Upsuppvar_\Upspsi + 1) \sigma_1 - \coordsize^2 |x| \sigma_1 \Upsuppvar_\Upspsi 
	= \coordsize^2 |x| \sigma_1 . 
\end{align*}
Now \eqref{ineq:partial2f upperbound} in Lemma~\ref{lem:dirderiv2uf upperbound} says that $\left| \frac{\partial^2 f}{\partial x_j^2}(\bvec{m}) \right| \le \coordsize^2 |x| \sigma_1$, so we obtain \eqref{ineq:PNSPfderivs bound}.
\end{proof}

Lemmas~\ref{lem:PNSPffirstderiv bound} and \ref{lem:PNSPfderivs bound} state that \eqref{ineq:PNSPffirstderiv bound} and \eqref{ineq:PNSPfderivs bound} hold for all $\bvec{m} \in \supp(\Upspsi)$. This is sufficient to apply Theorem~\ref{thm:PNSP1d} and obtain for all $M \ge 0$, 
\begin{align*}
&\int_{\RR^\numvars} \e{f(\bvec{m})} \Upspsi(\bvec{m}) \ d \bvec{m} \\
	&= \int_{\RR} \cdots \int_{\RR} \e{f(\bvec{m})} \Upspsi(\bvec{m}) \ d m_j \ d m_1 \ d m_2 \cdots \ d m_{j-1} \ d m_{j+1} \ d m_{j+2} \cdots \ d m_\numvars \\
	&\ll_{M, \Upspsi} \int_{S_{\Upspsi , j}} \left( \frac{\coordsize}{q (\Upsuppvar_\Upspsi + 1) \sqrt{\numvars}} \| \bvec{r} \| \right)^{-M} \ d m_1 \ d m_2 \cdots \ d m_{j-1} \ d m_{j+1} \ d m_{j+2} \cdots \ d m_\numvars , \numberthis \label{ineq:efbumpPNSPbound inter}
\end{align*}
where $S_{\Upspsi , j} \subseteq \RR^{\numvars-1}$ is the set of $(m_1 , m_2 , \ldots , m_{j-1} , m_{j+1} , m_{j+2} , \ldots , m_{\numvars})^\top$ in which there exists $m_j \in \RR$ such that $(m_1 , m_2 , \ldots , m_{\numvars})^\top \in \supp(\Upspsi)$. Because $\Upspsi$ has compact support, the set $S_{\Upspsi , j}$ is bounded and \eqref{ineq:efbumpPNSPbound inter} implies \eqref{ineq:efbumpPNSPbound1}.
\end{proof}

For our purposes, we will want to apply the principle of nonstationary phase outside of an $\numvars$-dimensional ball as opposed to outside of an $\numvars$-dimensional cube. Therefore, we have the following corollary.
\begin{cor} \label{cor:IFx\coordsize rq PNSP bound in terms of r}
If 
\begin{align}
\| \bvec{r} \| \ge q \coordsize |x| \sigma_1 (\Upsuppvar_\Upspsi + 1) \sqrt{\numvars} , \label{ineq:rPNSP cond}
\end{align}
 then 
\begin{align}
\mathcal{I}_{F, \Upspsi} (x, \coordsize, \bvec{r}, q) &\ll_{M, \Upspsi} \min \left\{ \coordsize^\numvars , \coordsize^{\numvars-M} \left( \frac{1}{q (\Upsuppvar_\Upspsi + 1) \sqrt{\numvars}} \| \bvec{r} \| \right)^{-M} \right\}
\end{align}
for all $M \ge 0$. 
\end{cor}
\begin{proof}
If $\bvec{r}$ satisfies \eqref{ineq:rPNSP cond}, then there exists an integer $j$ with $1 \le j \le \numvars$ such that $r_j$ satisfies \eqref{ineq:rjPNSP cond}. 
An application of Theorem~\ref{thm:IFx\coordsize rq PNSP bound in terms of r} gives the result of this corollary.
\end{proof}

\section{Putting estimates together} \label{chapter:puttogether}

In this section, we use results from previous sections to prove Theorem~\ref{thm:R\coordsize\numrepresented asymp1} and Corollaries~\ref{cor:R\coordsize\numrepresented asymp2} and \ref{cor:repnum F asymp}. 
We begin this section by splitting up the weighted representation number $\repnum_{F, \Upspsi, \coordsize} (\numrepresented)$ into a main term and some error terms. 
From \eqref{eq:R\coordsize \numrepresented with Trq\numrepresented x}, we see that 
\begin{align}
\repnum_{F, \Upspsi, \coordsize} (\numrepresented) &= M_{F, \Upspsi, \coordsize} (\numrepresented) + E_{F, \Upspsi, \coordsize, 1} (\numrepresented) + E_{F, \Upspsi, \coordsize, 2} (\numrepresented) + E_{F, \Upspsi, \coordsize, 3} (\numrepresented) , \label{eq:R\coordsize \numrepresented split}
\end{align}
where
\begin{align}
M_{F, \Upspsi, \coordsize} (\numrepresented) &= 2 \Re\left( \sum_{1 \le q \le Q} \frac{1}{q^\numvars} \int_{0}^{\frac{1}{q (q + Q)}} \e{-\numrepresented x} \mathcal{I}_{F, \Upspsi} (x, \coordsize, \bvec{0}, q) T_{\bvec{0}} (q, \numrepresented; x) \ d x \right) , \label{eq:M\coordsize \numrepresented} 
\end{align}
\begin{align*}
E_{F, \Upspsi, \coordsize, 1} (\numrepresented) &= 2 \Re\Bigg( \sum_{1 \le q \le Q} \frac{1}{q^\numvars} \int_{0}^{\frac{1}{q (q + Q)}} \e{-\numrepresented x} \\
		&\qquad\qquad\qquad \times \sum_{\substack{\bvec{r} \in \ZZ^\numvars \\ 0 < \| \bvec{r} \| \le q \coordsize |x| \sigma_1 (\Upsuppvar_\Upspsi + 1) \sqrt{\numvars}}} \mathcal{I}_{F, \Upspsi} (x, \coordsize, \bvec{r}, q) T_{\bvec{r}} (q, \numrepresented; x) \ d x \Bigg) , \numberthis \label{eq:E\coordsize1n} 
\end{align*}
\begin{align*}
E_{F, \Upspsi, \coordsize, 2} (\numrepresented) &= 2 \Re\Bigg( \sum_{1 \le q \le Q} \frac{1}{q^\numvars} \int_{\frac{1}{q (q + Q)}}^{\frac{1}{q Q}} \e{-\numrepresented x}  \\
		&\qquad\qquad\qquad \times \sum_{\substack{\bvec{r} \in \ZZ^\numvars \\ \| \bvec{r} \| \le q \coordsize |x| \sigma_1 (\Upsuppvar_\Upspsi + 1) \sqrt{\numvars}}} \mathcal{I}_{F, \Upspsi} (x, \coordsize, \bvec{r}, q) T_{\bvec{r}} (q, \numrepresented; x) \ d x \Bigg) , \numberthis \label{eq:E\coordsize2n} 
\end{align*}
and 
\begin{align*}
E_{F, \Upspsi, \coordsize, 3} (\numrepresented) &= 2 \Re\Bigg( \sum_{1 \le q \le Q} \frac{1}{q^\numvars} \int_{0}^{\frac{1}{q Q}} \e{-\numrepresented x} \\
		&\qquad\qquad\qquad \times \sum_{\substack{\bvec{r} \in \ZZ^\numvars \\ \| \bvec{r} \| > q \coordsize |x| \sigma_1 (\Upsuppvar_\Upspsi + 1) \sqrt{\numvars}}} \mathcal{I}_{F, \Upspsi} (x, \coordsize, \bvec{r}, q) T_{\bvec{r}} (q, \numrepresented; x) \ d x \Bigg) . \numberthis \label{eq:E\coordsize3}
\end{align*}
We call $M_{F, \Upspsi, \coordsize} (\numrepresented)$ the main term of $\repnum_{F, \Upspsi, \coordsize} (\numrepresented)$. We call $E_{F, \Upspsi, \coordsize, 1} (\numrepresented)$, $E_{F, \Upspsi, \coordsize, 2} (\numrepresented)$, and $E_{F, \Upspsi, \coordsize, 3} (\numrepresented)$ the error terms of $\repnum_{F, \Upspsi, \coordsize} (\numrepresented)$.

In this section, we will provide an asymptotic for the main term $M_{F, \Upspsi, \coordsize} (\numrepresented)$ and upper bounds for the absolute values of the error terms. To do this, we will first prove some more results that will help us prove Theorem~\ref{thm:R\coordsize\numrepresented asymp1} and Corollaries~\ref{cor:R\coordsize\numrepresented asymp2} and \ref{cor:repnum F asymp}.

\subsection{Stating some supporting results} \label{sec:suppresults}

In this subsection, we state some lemmas that will be used to provide an asymptotic for the main term $M_{F, \Upspsi, \coordsize} (\numrepresented)$ and upper bounds for the absolute values of the error terms. 

\subsubsection{An upper bound for the absolute value of a particular sum}

The sum in the following lemma will come up multiple times in our estimates. (The sum is related to our estimate of the sum $T_{\bvec{r}} (q, \numrepresented; x)$.) The lemma provides provides an upper bound for the absolute value of this sum. 
\begin{lem} \label{lem:q1gcdsum+eps}
Let $Q \ge 1$ and let $C$ and $\numrepresented$ be nonzero integers. 
For an integer $q$, we split $q$ into $q = q_0 q_1$ such that $q_0$ is the largest factor of $q$ having all of its prime divisors dividing $C$ so that $\gcd(q_1 , C) = 1$. Then 
\begin{align}
\sum_{1 \le q \le Q} (\gcd(\numrepresented, q_1))^{1/2} q_1^{-1/2} \tau(q) \log(2q) \ll_\varepsilon Q^{1/2 + \varepsilon} \tau(\numrepresented) \prod_{p \mid C} ( 1 - p^{-1/2} )^{-1} \label{ineq:q1gcdsum+eps bound}
\end{align}
for any $\varepsilon > 0$.
\end{lem}
\begin{rmk}
In our applications of Lemma~\ref{lem:q1gcdsum+eps}, the integer $C$ is equal to $2 \det(A)$.
\end{rmk}

The proof of Lemma~\ref{lem:q1gcdsum+eps} requires another lemma, which we now state and prove.
\begin{lem} \label{lem:q1gcdsum}
Let $Q \ge 1$ and let $C$ and $\numrepresented$ be nonzero integers. 
For an integer $q$, we split $q$ into $q = q_0 q_1$ such that $q_0$ is the largest factor of $q$ having all of its prime divisors dividing $C$ so that $\gcd(q_1 , C) = 1$. Then 
\begin{align}
\sum_{1 \le q \le Q} q_1^{-1/2} (\gcd(\numrepresented, q_1))^{1/2} \ll Q^{1/2} \tau(\numrepresented) \prod_{p \mid C} ( 1 - p^{-1/2} )^{-1} . \label{ineq:q1gcdsum bound}
\end{align}
\end{lem}
\begin{proof}
Throughout this proof, we have $q_0$ such that if a prime $p$ divides $q_0$ then $p$ divides $C$.

Observe that 
\begin{align*}
\sum_{1 \le q \le Q} q_1^{-1/2} (\gcd(\numrepresented, q_1))^{1/2} &\le \sum_{1 \le q \le Q} q_1^{-1/2} \sum_{\substack{d \mid \numrepresented \\ d \mid q_1}} d^{1/2} \\
	&\le \sum_{1 \le q_0 \le Q} \sum_{1 \le q_1 \le Q / q_0} q_1^{-1/2} \sum_{\substack{d \mid \numrepresented \\ d \mid q_1}} d^{1/2} \\
	&= \sum_{d \mid \numrepresented} d^{1/2} \sum_{1 \le q_0 \le Q} \sum_{\substack{1 \le q_1 \le Q / q_0 \\ q_1 \equiv 0 \pmod{d}}} q_1^{-1/2} \numberthis \label{ineq:q1gcdsum bound1}
\end{align*}
by switching the order of summation.

Let $q_2$ be $q_1/d$. Then
\begin{align*}
\sum_{d \mid \numrepresented} d^{1/2} \sum_{1 \le q_0 \le Q} \sum_{\substack{1 \le q_1 \le Q / q_0 \\ q_1 \equiv 0 \pmod{d}}} q_1^{-1/2} &= \sum_{d \mid \numrepresented} d^{1/2} \sum_{1 \le q_0 \le Q} \sum_{1 \le q_2 \le Q / (q_0 d)} (q_2 d)^{-1/2} \\
	&= \sum_{d \mid \numrepresented} \sum_{1 \le q_0 \le Q} \sum_{1 \le q_2 \le Q / (q_0 d)} q_2^{-1/2} .
\end{align*}
Substituting this into \eqref{ineq:q1gcdsum bound1}, we have
\begin{align*}
\sum_{1 \le q \le Q} q_1^{-1/2} (\gcd(\numrepresented, q_1))^{1/2} &\le \sum_{d \mid \numrepresented} \sum_{1 \le q_0 \le Q} \sum_{1 \le q_2 \le Q / (q_0 d)} q_2^{-1/2} \\
	&\le \sum_{d \mid \numrepresented} \sum_{1 \le q_0 \le Q} \sum_{1 \le q_2 \le Q / q_0} q_2^{-1/2} .
\end{align*}

Using part~(b) Theorem~3.2 of \cite{ApostolAnalNumThy}, we obtain
\begin{align*}
\sum_{1 \le q \le Q} q_1^{-1/2} (\gcd(\numrepresented, q_1))^{1/2} &\ll \sum_{d \mid \numrepresented} \sum_{1 \le q_0 \le Q} \left( \frac{Q}{q_0} \right)^{1/2} \\
	&= Q^{1/2} \sum_{d \mid \numrepresented} \sum_{1 \le q_0 \le Q} q_0^{-1/2} \\
	&\le Q^{1/2} \sum_{d \mid \numrepresented} \sum_{q_0 > 0} q_0^{-1/2} \\
	&= Q^{1/2} \tau(\numrepresented) \sum_{q_0 > 0} q_0^{-1/2} . \numberthis \label{ineq:q1gcdsum bound2}
\end{align*}
We see that \eqref{ineq:q1gcdsum bound} follows from \eqref{ineq:q1gcdsum bound2}, because 
\begin{align*}
\sum_{q_0 > 0} q_0^{-1/2} &= \prod_{p \mid C} \sum_{j=0}^\infty (p^{-1/2})^j \\
	&= \prod_{p \mid C} ( 1 - p^{-1/2} )^{-1} . \qedhere
\end{align*}
\end{proof}

The proof of Lemma~\ref{lem:q1gcdsum+eps} follows quickly from Lemma~\ref{lem:q1gcdsum}
\begin{proof}[Proof of Lemma~\ref{lem:q1gcdsum+eps}]
For any $\varepsilon > 0$, we have $\tau(q) \ll_\varepsilon q^\varepsilon \le Q^\varepsilon$ and $\log(2q) \ll_\varepsilon (2q)^\varepsilon \le (2Q)^\varepsilon \ll_\varepsilon Q^\varepsilon$. 
Thus, 
\begin{align*}
\sum_{1 \le q \le Q} (\gcd(\numrepresented, q_1))^{1/2} q_1^{-1/2} \tau(q) \log(2q) \ll_{\varepsilon} Q^\varepsilon \sum_{1 \le q \le Q} (\gcd(\numrepresented, q_1))^{1/2} q_1^{-1/2} . 
\end{align*}
By applying Lemma~\ref{lem:q1gcdsum}, we obtain the result of Lemma~\ref{lem:q1gcdsum+eps}.
\end{proof}

\subsubsection{The volume of an $\numvars$-dimensional ball, integer lattice point counting, and sums over integer lattice points} \label{subsec:volball intlatticeptcounting sumsoverintlatticepts}

This subsection contains information about the volume of an $\numvars$-dimensional ball, integer lattice point counting, and sums over integer lattice points. Throughout this subsection (Subsection~\ref{subsec:volball intlatticeptcounting sumsoverintlatticepts}), the number $\numvars$ is only required to be a positive integer.

We begin by stating (without proof) the following result about the volume of an $\numvars$-dimensional ball that can be found in a number of sources (including, for example, Section~2.C of Chapter~21 of \cite{ConwaySloaneSpherePackings}).
\begin{lem} \label{lem:\numvars-ballvol}
The volume of an $\numvars$-dimensional ball of radius $R$ is 
\begin{align}
\frac{\pi^{\numvars/2}}{\Gamma \left( \numvars/2 + 1 \right)} R^\numvars ,
\end{align}
where $\numvars$ is a positive integer.
\end{lem}

We are primarily concerned with $\numvars$-dimensional balls centered at the origin.
Let $B_\numvars (R)$ be the closed $\numvars$-dimensional ball centered at the origin with radius $R$. This $\numvars$-dimensional ball is defined by 
\begin{align}
B_\numvars (R) &= \{ \bvec{x} \in \RR^\numvars : \| \bvec{x} \| \le R \} . \label{eq:closedsballorigin}
\end{align}
Let $B^o_\numvars (R)$ be the open $\numvars$-dimensional ball centered at the origin with radius $R$. This $\numvars$-dimensional ball is defined by 
\begin{align}
B^o_\numvars (R) &= \{ \bvec{x} \in \RR^\numvars : \| \bvec{x} \| < R \} . \label{eq:opensballorigin}
\end{align}

For a Lebesgue measurable subset $W$ of $\RR^\numvars$, let $\Vol_\numvars (W)$ be the $\numvars$-dimensional volume of $W$. 
Then Lemma~\ref{lem:\numvars-ballvol} implies that 
\begin{align}
\Vol_\numvars (B_\numvars (R)) = \Vol_\numvars (B^o_\numvars (R)) &= \frac{\pi^{\numvars/2}}{\Gamma \left( \numvars/2 + 1 \right)} R^\numvars \label{eq:\numvars-balloriginvol}
\end{align}
since the the boundary of an $\numvars$-dimensional ball has zero volume.

We will also need to count the number of integer lattices points in an $\numvars$-dimensional ball centered at the origin. We first state a result that can be proven with a geometric argument (originally due to Gauss~\cite{GaussCircleProbWerke} for counting lattice points inside a circle).
\begin{lem} \label{lem:latticeptsinballR>=1}
Let $\numvars$ be a positive integer and $R \ge 1$. 
Then the number of integer lattice points in $B_\numvars (R)$ is 
\begin{align}
\left| \{ \bvec{m} \in \ZZ^\numvars : \| \bvec{m} \| \le R \} \right| 
	&= \frac{\pi^{\numvars/2}}{\Gamma \left( \numvars/2 + 1 \right)} R^\numvars + O_\numvars (R^{\numvars - 1}) . \label{eq:latticeptsinballsufflarge}
\end{align}
\end{lem}
\begin{rmk}
The error term in \eqref{eq:latticeptsinballsufflarge} is not the best possible error term for $\numvars \ge 2$. 
For example, for $\numvars = 2$, Huxley~\cite{HuxleyIntPtsExpSumsRZF} obtained an error term of $O(R^{131/208})$. (This is not necessarily the best result for $\numvars = 2$; there is a preprint by Bourgain and Watt~\cite{BourgainWattMeanSquareZetaCircleDivisorProbRevisited} that claims a better error term for $\numvars = 2$.) 
For $\numvars = 3$, Heath-Brown~\cite{H-BLatticePtsInSphere} obtained an error term of $O(R^{21/16})$.
For $\numvars \ge 4$, the best known error term is $O_\numvars (R^{\numvars - 2})$, which can be proved using a classical formula for the number of representations of an integer as the sum of four squares. (See, for instance, \cite{FrickerLatticePts}.)
Given all of this, the error term in \eqref{eq:latticeptsinballsufflarge} is sufficient for our purposes.
\end{rmk}
\begin{rmk} \label{rmk:latticptsinsmallballerror}
Lemma~\ref{lem:latticeptsinballR>=1} requires that $R \ge 1$. A lower bound like this is required, because for fixed $\numvars$, we have
\begin{align}
\lim_{R \to 0^+} R^\numvars = \lim_{R \to 0^+} R^{\numvars - 1} = 0 .
\end{align}
However, since the zero vector is contained in the set $\{ \bvec{m} \in \ZZ^\numvars : \| \bvec{m} \| \le R \}$, we have
\begin{align}
\left| \{ \bvec{m} \in \ZZ^\numvars : \| \bvec{m} \| \le R \} \right| \ge 1 \label{ineq:latticeptsinball triviallowerbound}
\end{align}
for all positive $R$. Therefore, we cannot have \eqref{eq:latticeptsinballsufflarge} be true for all $R > 0$. 
\end{rmk}

As noted in Remark~\ref{rmk:latticptsinsmallballerror}, we cannot have \eqref{eq:latticeptsinballsufflarge} be true for all $R > 0$. However, we do want to have an upper bound for the number of points in $B_\numvars (R)$ that is true for all $R > 0$.
We first state a result that gives an exact count for the number of points in $B_\numvars (R)$ if $0 < R < 1$.
\begin{lem} \label{lem:latticeptsinballR<1}
Let $\numvars$ be a positive integer and $0 < R < 1$. 
Then the number of integer lattice points in $B_\numvars (R)$ is 
\begin{align*}
\left| \{ \bvec{m} \in \ZZ^\numvars : \| \bvec{m} \| \le R \} \right| &= 1 . 
\end{align*}
\end{lem}
\begin{proof}
The zero vector is in $\{ \bvec{m} \in \ZZ^\numvars : \| \bvec{m} \| \le R \}$. 
If $\bvec{m} \in \ZZ^\numvars$ is a nonzero vector, then $\| \bvec{m} \| \ge 1 > R$. Therefore, the only vector in $\{ \bvec{m} \in \ZZ^\numvars : \| \bvec{m} \| \le R \}$ is the zero vector, and $\left| \{ \bvec{m} \in \ZZ^\numvars : \| \bvec{m} \| \le R \} \right| = 1$.
\end{proof}

We use Lemmas~\ref{lem:latticeptsinballR>=1} and \ref{lem:latticeptsinballR<1} to obtain the following upper bound for the number of points in $B_\numvars (R)$.
\begin{thm} \label{thm:latticeptsinball upperbound}
Let $\numvars$ be a positive integer and $R > 0$. 
Then the number of integer lattice points in $B_\numvars (R)$ is 
\begin{align*}
\left| \{ \bvec{m} \in \ZZ^\numvars : \| \bvec{m} \| \le R \} \right| &\ll_{\numvars} R^\numvars + 1.
\end{align*}
\end{thm}
\begin{proof}
If $R \ge 1$, then
\begin{align*}
\left| \{ \bvec{m} \in \ZZ^\numvars : \| \bvec{m} \| \le R \} \right| &\ll_{\numvars} R^\numvars \le R^\numvars + 1
\end{align*}
by Lemma~\ref{lem:latticeptsinballR>=1}.

If $0 < R < 1$, then 
\begin{align*}
\left| \{ \bvec{m} \in \ZZ^\numvars : \| \bvec{m} \| \le R \} \right| &= 1 \le R^\numvars + 1
\end{align*}
by Lemma~\ref{lem:latticeptsinballR<1}.
\end{proof}

For $R > 0$, sometimes we will want to have an upper bound for the number of integer lattice points $\bvec{m} \in \ZZ^\numvars$ satisfying $0 < \| \bvec{m} \| \le R$. Such an upper bound is stated in the next corollary.
\begin{cor} \label{cor:latticeptsinpuncball upperbound}
Let $\numvars$ be a positive integer and $R > 0$. 
Then the number of nonzero integer lattice points in $B_\numvars (R)$ is 
\begin{align*}
\left| \{ \bvec{m} \in \ZZ^\numvars : 0 < \| \bvec{m} \| \le R \} \right| &\ll_{\numvars} R^\numvars .
\end{align*}
\end{cor}
\begin{proof}
If $R \ge 1$, then the result of the corollary follows from Lemma~\ref{lem:latticeptsinballR>=1}.

If $0 < R < 1$, then 
\begin{align*}
\left| \{ \bvec{m} \in \ZZ^\numvars : 0 < \| \bvec{m} \| \le R \} \right| &= 0 \le R^\numvars . \qedhere
\end{align*}
\end{proof}

We will have summations involving the Euclidean norm of vectors. The following theorem provides an upper bound for such a sum.
\begin{thm} \label{thm:latticesumbound B>=1}
Suppose that $B \ge 1$ and $M > \numvars > 0$. Then
\begin{align*}
\sum_{\substack{\bvec{r} \in \ZZ^\numvars \\ \| \bvec{r} \| > B}} \| \bvec{r} \|^{-M} \ll_{\numvars} \left( 1 + \frac{4}{M (M - \numvars)} \right) B^{\numvars-M} . 
\end{align*}
\end{thm}
\begin{proof}
Define $\mathcal{L}$ to be the lattice sum 
\begin{align*}
\mathcal{L} = \sum_{\substack{\bvec{r} \in \ZZ^\numvars \\ \| \bvec{r} \| > B}} \| \bvec{r} \|^{-M} & = \sum_{\substack{\bvec{r} \in \ZZ^\numvars \\ \| \bvec{r} \|^2 > B^2}} (\| \bvec{r} \|^2)^{-M/2} .
\end{align*}
For a nonnegative integer $j$, let 
\begin{align*}
a(j) &= \left| \{ \bvec{r} \in \ZZ^\numvars : \| \bvec{r} \|^2 = j \} \right|.
\end{align*}
Then
\begin{align*}
\mathcal{L} = \sum_{j > B^2} a(j) j^{-M/2} .
\end{align*}
For $R \ge 1$, let $\mathcal{L}(R)$ be
\begin{align*}
\mathcal{L}(R) &= \sum_{\substack{\bvec{r} \in \ZZ^\numvars \\ B < \| \bvec{r} \| \le R}} \| \bvec{r} \|^{-M} = \sum_{B^2 < j \le R^2} a(j) j^{-M/2} 
\end{align*}
so that $\lim_{R \to \infty} \mathcal{L}(R) = \mathcal{L}$.

We will use Abel's summation formula, also called Abel's identity. (See Theorem~4.2 of \cite{ApostolAnalNumThy}.) Therefore, it is useful to define the following sum: 
For $y \ge 0$, let
\begin{align}
S(y) &= \sum_{j=0}^{\lfloor y \rfloor} a(j) .
\end{align}
The definition of $a(j)$ implies that 
\begin{align*}
S(y) &= \left| \{ \bvec{r} \in \ZZ^\numvars : \| \bvec{r} \|^2 \le y \} \right| \\
	&= \left| \{ \bvec{r} \in \ZZ^\numvars : \| \bvec{r} \| \le \sqrt{y} \} \right| . \numberthis \label{eq:Sytolatticecount}
\end{align*}
Therefore, if $y \ge 1$, we know from Lemma~\ref{lem:latticeptsinballR>=1} that
\begin{align}
S(y) &\ll_\numvars y^{\numvars/2} . \label{ineq:Sylarge upperbound}
\end{align}

By Abel's summation formula,
\begin{align}
\mathcal{L}(R) &= S(R^2) R^{-M} - S(B^2) B^{-M} + \frac{2}{M} \int_{B^2}^{R^2} S(t) t^{-M/2 - 1} \ d t . \label{eq:LRsumbyparts}
\end{align}
Using \eqref{ineq:Sylarge upperbound} in \eqref{eq:LRsumbyparts}, we find that
\begin{align*}
\mathcal{L}(R) &\ll_\numvars  R^{\numvars - M} + B^{\numvars - M} + \frac{2}{M} \int_{B^2}^{R^2} t^{(\numvars - M)/2 - 1} \ d t \\
	&= R^{\numvars - M} + B^{\numvars - M} +  \frac{4}{M ( M - \numvars )} \left( B^{\numvars - M} - R^{\numvars - M} \right) . \numberthis \label{ineq:LR upperbound1}
\end{align*}
Taking the limits in \eqref{ineq:LR upperbound1} as $R \to \infty$, we find that 
\begin{align*}
\mathcal{L} &\ll_\numvars  B^{\numvars - M} +  \frac{4}{M ( M - \numvars )} B^{\numvars - M}
\end{align*}
since $M > \numvars > 0$.
\end{proof}

Instead of using the previous theorem, we will use a corollary of it for ease of use.
\begin{cor} \label{cor:latticesumbound B>=1}
Suppose that $B \ge 1$ and $M \ge \numvars + 1$. Then
\begin{align}
\sum_{\substack{\bvec{r} \in \ZZ^\numvars \\ \| \bvec{r} \| > B}} \| \bvec{r} \|^{-M} \ll_{\numvars} B^{\numvars-M} . \label{ineq:latticesumbound B>=1 cor}
\end{align}
\end{cor}
\begin{proof}
Because $M \ge \numvars + 1$, we have
\begin{align*}
1 + \frac{4}{M (M - \numvars)} \le 1 + \frac{4}{\numvars+1} ,
\end{align*}
so 
\begin{align*}
\left( 1 + \frac{4}{M (M - \numvars)} \right) B^{\numvars-M} \ll_{\numvars} B^{\numvars-M} .
\end{align*}
The result of the corollary follows from Theorem~\ref{thm:latticesumbound B>=1}.
\end{proof}

Because such a sum will arise, we provide an upper bound for $\sum_{\substack{\bvec{r} \in \ZZ^\numvars \\ \| \bvec{r} \| > B}} \| \bvec{r} \|^{-M}$ when $B \ge 0$.
\begin{cor} \label{cor:latticesumbound B>=0}
Suppose that $B \ge 0$ and $M \ge \numvars + 1$. Then
\begin{align}
\sum_{\substack{\bvec{r} \in \ZZ^\numvars \\ \| \bvec{r} \| > B}} \| \bvec{r} \|^{-M} \ll_{\numvars} 1 . \label{ineq:latticesumbound B>=0 cor}
\end{align}
\end{cor}
\begin{proof}
If $B \ge 1$, then it follows from Corollary~\ref{cor:latticesumbound B>=1} that
\begin{align*}
\sum_{\substack{\bvec{r} \in \ZZ^\numvars \\ \| \bvec{r} \| > B}} \| \bvec{r} \|^{-M} \ll_{\numvars} B^{\numvars-M} \le 1
\end{align*}
since $M \ge \numvars + 1$.

Now suppose that $0 \le B < 1$. Then
\begin{align}
\sum_{\substack{\bvec{r} \in \ZZ^\numvars \\ \| \bvec{r} \| > B}} \| \bvec{r} \|^{-M} &= \left| \{ \bvec{r} \in \ZZ^\numvars : \| \bvec{r} \| = 1 \} \right| + \sum_{\substack{\bvec{r} \in \ZZ^\numvars \\ \| \bvec{r} \| > 1}} \| \bvec{r} \|^{-M} . \label{eq:latticesum 0<=B<1}
\end{align}
Now 
\begin{align}
\left| \{ \bvec{r} \in \ZZ^\numvars : \| \bvec{r} \| = 1 \} \right| & = 2 \numvars \label{eq:|r|=1}
\end{align}
since the vectors $\bvec{r} \in \ZZ^\numvars$ with $\| \bvec{r} \| = 1$ are the vectors with exactly one nonzero entry and that nonzero entry is either $1$ or $-1$. 
From Corollary~\ref{cor:latticesumbound B>=1}, we know that 
\begin{align}
\sum_{\substack{\bvec{r} \in \ZZ^\numvars \\ \| \bvec{r} \| > 1}} \| \bvec{r} \|^{-M} \ll_{\numvars} 1 . \label{ineq:latticesumbound B=1}
\end{align}
Substituting \eqref{eq:|r|=1} and \eqref{ineq:latticesumbound B=1} into \eqref{eq:latticesum 0<=B<1}, we obtain \eqref{ineq:latticesumbound B>=0 cor}.
\end{proof}

\subsubsection{The compactness of the preimage of a positive definite quadratic form} \label{subsec:compactpreimageF}

In this subsection, we prove that the set $V = \{ \bvec{m} \in \RR^\numvars : F(\bvec{m}) = \numrepresented \}$ is compact when $F$ is a positive definite quadratic form and $\numrepresented$ is a positive real number.
Before we do this, we prove that if $\bvec{m}$ is in $V$, then there are bounds on $\| \bvec{m} \|$.
We begin by proving that if $F$ is a nonsingular quadratic form and  $F(\bvec{m}) = \numrepresented$, then there is a lower bound on $\| \bvec{m} \|$.
\begin{lem} \label{lem:preimageFlowerbound}
Suppose that $F$ is a nonsingular quadratic form in $\numvars \ge 1$ variables. Suppose that $\numrepresented$ is a positive real number. 
Let $A \in \Mat_\numvars (\RR)$ be the Hessian matrix of $F$.
Let $\sigma_1$ be the largest singular value of $A$.
Suppose that $\bvec{m} \in \RR^\numvars$ satisfies $F(\bvec{m}) = \numrepresented$. 
Then
\begin{align}
\| \bvec{m} \| \ge \sqrt{\frac{2 \numrepresented}{\sigma_1}}  . \label{ineq:m in preimageF lower bound}
\end{align}
\end{lem}
\begin{proof}
 By Lemma~\ref{lem:wtBw upperbound}, we have 
\begin{align}
\numrepresented &= F(\bvec{m}) = \frac{1}{2} \bvec{m}^\top A \bvec{m} \le \frac{1}{2} \sigma_1 \| \bvec{m} \|^2 . \label{ineq:m sigmalowerbound}
\end{align}
Solving for $\| \bvec{m} \|$ in \eqref{ineq:m sigmalowerbound}, we obtain \eqref{ineq:m in preimageF lower bound}.
\end{proof}

We now prove an upper and lower bound for  $\| \bvec{m} \|$ when $F$ is a positive definite quadratic form and $F(\bvec{m}) = \numrepresented$.
\begin{lem} \label{lem:preimageFbounded}
Suppose that $F$ is a positive definite quadratic form in $\numvars \ge 1$ variables. Suppose that $\numrepresented$ is a positive real number. 
Let $A \in \Mat_\numvars (\RR)$ be the Hessian matrix of $F$.
Let $\lambda_\numvars$ be the smallest eigenvalue of $A$, and let $\lambda_1$ be the largest eigenvalue of $A$.
Suppose that $\bvec{m} \in \RR^\numvars$ satisfies $F(\bvec{m}) = \numrepresented$. 
Then
\begin{align}
\sqrt{\frac{2 \numrepresented}{\lambda_1}} \le \| \bvec{m} \| \le \sqrt{\frac{2 \numrepresented}{\lambda_\numvars}} . \label{ineqs:m in preimageF bounds}
\end{align}
\end{lem}
\begin{rmk}
Other bounds for the size of a real solution of $F(\bvec{m}) = \numrepresented$ have been found when $F$ is a positive definite integral quadratic form. For example, let $f_{j k}$ be the coefficient of $m_j m_k$ in $F$, where $1 \le j \le k \le \numvars$. Kornhauser~\cite[Lemma~10]{KornhauserSmallSolnsQuadEqns5+} showed that if $F(\bvec{m}) = \numrepresented$, then
\begin{align*}
\| \bvec{m} \| < 2 (4 \numvars H)^{(\numvars - 1)/2} \numrepresented^{1/2} ,
\end{align*}
where $H = \max_{1 \le j \le k \le \numvars} |f_{j k}|$.
(For a proof of this result, see the proof of Lemma~10 in \cite{KornhauserBoundsSmallestIntSolnQuadEqns}.)
\end{rmk}
\begin{rmk}
The condition that $F$ is positive definite is crucial in Lemma~\ref{lem:preimageFbounded}. For example, the set $\{ (m_1 , m_2)^\top \in \RR^2 : m_1^2 - m_2^2 = 1 \}$ is a hyperbola and is not bounded.
\end{rmk}
\begin{proof}[Proof of Lemma~\ref{lem:preimageFbounded}]
Because $A$ is positive definite, the largest singular value of $A$ is $\lambda_1$. That is, $\sigma_1 = \lambda_1$.
Thus, the first inequality in \eqref{ineqs:m in preimageF bounds} follows from Lemma~\ref{lem:preimageFlowerbound}.

We now prove the second inequality in \eqref{ineqs:m in preimageF bounds}. 
Using the spectral theorem for symmetric matrices, we can write the matrix $A$ as 
\begin{align*}
A = P^\top D P,
\end{align*}
where $P$ is an orthogonal matrix and $D = \diag(\lambda_1 , \ldots , \lambda_\numvars)$ is a diagonal matrix. Note that $\{ \lambda_j \}_{j=1}^\numvars$ is the set of eigenvalues of $A$.
Then
\begin{align}
F(\bvec{m}) &= \frac{1}{2} \bvec{m}^\top P^\top D P \bvec{m} . \label{eq:F(m) diag}
\end{align}

Let $\bvec{v} = P \bvec{m}$. Note that $\| \bvec{v} \| = \| \bvec{m} \|$ since $P$ is orthogonal. 
Substituting $\bvec{v} = P \bvec{m}$ into \eqref{eq:F(m) diag}, we see that 
\begin{align}
F(\bvec{m}) &= \frac{1}{2} \bvec{v}^\top D \bvec{v} 
	= \frac{1}{2} \sum_{j=1}^\numvars \lambda_j v_j^2 . \label{eq:vtDv/2}
\end{align}

Since $A$ is positive definite, each eigenvalue $\lambda_j$ is positive. Without loss of generality, we assume that $\lambda_1 \ge \lambda_2 \ge \cdots \ge \lambda_{\numvars-1} \ge \lambda_\numvars > 0$. 
Therefore, \eqref{eq:vtDv/2} implies that 
\begin{align*}
F(\bvec{m}) &\ge \frac{1}{2} \sum_{j=1}^\numvars \lambda_\numvars v_j^2 = \frac{\lambda_\numvars}{2} \| \bvec{v} \|^2 . \numberthis \label{ineq:F(m) lowerbound}
\end{align*}

Solving \eqref{ineq:F(m) lowerbound} for $\| \bvec{v} \|$, we find that 
\begin{align*}
\| \bvec{v} \| \le \sqrt{\frac{2 F(\bvec{m})}{\lambda_\numvars}} .
\end{align*}
Because $\| \bvec{v} \| = \| \bvec{m} \|$ and $F(\bvec{m}) = \numrepresented$, we obtain the second inequality in \eqref{ineqs:m in preimageF bounds}.
\end{proof}

Lemma~\ref{lem:preimageFbounded} shows that $V$ is bounded. We now use the Heine--Borel theorem to show that $V$ compact. 
\begin{thm} \label{thm:preimageFcompact}
Suppose that $F$ is a positive definite quadratic form in $\numvars \ge 1$ variables. Suppose that $\numrepresented$ is a positive real number. 
Then the set 
\begin{align}
V &= \{ \bvec{m} \in \RR^\numvars : F(\bvec{m}) = \numrepresented \}
\end{align}
is compact.
\end{thm}
\begin{proof}
Because $F$ is continuous and $\{ \numrepresented \}$ is a closed set, the preimage $F^{-1}(\{ \numrepresented \}) = V$ is a closed set. 
Lemma~\ref{lem:preimageFbounded} says that $V$ is bounded. 
Therefore, since $V$ is closed and bounded, the Heine--Borel theorem tells us that $V$ is compact.
\end{proof}

\subsection{Analyzing the main term}

Given the previous sections, we now analyze the main term $M_{F, \Upspsi, \coordsize} (\numrepresented)$ of our weighted representation number. 
Using \eqref{eq:Trq\numrepresented x with gausssumr}, \eqref{eq:gausssumr}, and \eqref{eq:IFx\coordsize rq const bump}, we expand \eqref{eq:M\coordsize \numrepresented} to discover that 
\begin{align*}
M_{F, \Upspsi, \coordsize} (\numrepresented) &= 2 \Re\left( \sum_{1 \le q \le Q} \frac{1}{q^\numvars} \int_{0}^{\frac{1}{q (q + Q)}} \e{-\numrepresented x} \coordsize^\numvars \int_{\RR^\numvars} \e{\coordsize^2 x F(\bvec{m})} \Upspsi(\bvec{m}) \ d \bvec{m} \right. \\
		&\qquad\qquad \left. \times \sum_{\substack{Q < d \le q+Q \\ \gcd(d, q) = 1}} \e{\numrepresented \frac{\modmultinv{d}}{q}} \sum_{\bvec{h} \in (\ZZ/q\ZZ)^\numvars} \e{\frac{- \modmultinv{d}}{q} F(\bvec{h})} \ d x \right) . \numberthis \label{eq:M\coordsize \numrepresented expr1}
\end{align*}
We use the fact that $2 \Re(z) = z + \lbar{z}$ for any $z \in \CC$ to expand \eqref{eq:M\coordsize \numrepresented expr1} and obtain
\begin{align*}
M_{F, \Upspsi, \coordsize} (\numrepresented) 
	&= \sum_{1 \le q \le Q} \frac{1}{q^\numvars} \int_{0}^{\frac{1}{q (q + Q)}} \e{-\numrepresented x} \coordsize^\numvars \int_{\RR^\numvars} \e{\coordsize^2 x F(\bvec{m})} \Upspsi(\bvec{m}) \ d \bvec{m} \\
		&\qquad\qquad \times \sum_{\substack{Q < d \le q+Q \\ \gcd(d, q) = 1}} \e{\numrepresented \frac{\modmultinv{d}}{q}} \sum_{\bvec{h} \in (\ZZ/q\ZZ)^\numvars} \e{\frac{- \modmultinv{d}}{q} F(\bvec{h})} \ d x \\
		&\qquad + \sum_{1 \le q \le Q} \frac{1}{q^\numvars} \int_{0}^{\frac{1}{q (q + Q)}} \e{\numrepresented x} \coordsize^\numvars \int_{\RR^\numvars} \e{- \coordsize^2 x F(\bvec{m})} \Upspsi(\bvec{m}) \ d \bvec{m} \\
		&\qquad\qquad \times \sum_{\substack{Q < d \le q+Q \\ \gcd(d, q) = 1}} \e{- \numrepresented \frac{\modmultinv{d}}{q}} \sum_{\bvec{h} \in (\ZZ/q\ZZ)^\numvars} \e{\frac{\modmultinv{d}}{q} F(\bvec{h})} \ d x . \numberthis \label{eq:M\coordsize \numrepresented expand1}
\end{align*}

Now, by mapping $d$ to $-d$, we have
\begin{align*}
&\sum_{\substack{Q < d \le q+Q \\ \gcd(d, q) = 1}} \e{\numrepresented \frac{\modmultinv{d}}{q}} \sum_{\bvec{h} \in (\ZZ/q\ZZ)^\numvars} \e{\frac{\modmultinv{- d}}{q} F(\bvec{h})} \\
&= \sum_{\substack{-Q > d \ge -(q+Q) \\ \gcd(d, q) = 1}} \e{- \numrepresented \frac{\modmultinv{d}}{q}} \sum_{\bvec{h} \in (\ZZ/q\ZZ)^\numvars} \e{\frac{\modmultinv{d}}{q} F(\bvec{h})} . \numberthis \label{eq:T0q\numrepresented x main conjugate 1}
\end{align*}
Because $\e{\frac{\cdot}{q}}$ is periodic modulo $q$, we can sum $d$ over any reduced residue system and obtain the same value for \eqref{eq:T0qnx main conjugate 1}. Therefore, we rewrite \eqref{eq:T0q\numrepresented x main conjugate 1} as 
\begin{align}
\sum_{\substack{Q < d \le q+Q \\ \gcd(d, q) = 1}} \e{\numrepresented \frac{\modmultinv{d}}{q}} \sum_{\bvec{h} \in (\ZZ/q\ZZ)^\numvars} \e{\frac{- \modmultinv{d}}{q} F(\bvec{h})} &= \sum_{\substack{Q < d \le q+Q \\ \gcd(d, q) = 1}} \e{- \numrepresented \frac{\modmultinv{d}}{q}} \sum_{\bvec{h} \in (\ZZ/q\ZZ)^\numvars} \e{\frac{\modmultinv{d}}{q} F(\bvec{h})} . \label{eq:T0q\numrepresented x main conjugate 2}
\end{align}
Substituting this into \eqref{eq:M\coordsize \numrepresented expand1}, we obtain
\begin{align*}
M_{F, \Upspsi, \coordsize} (\numrepresented) &= \sum_{1 \le q \le Q} \frac{1}{q^\numvars} \int_{0}^{\frac{1}{q (q + Q)}} \e{-\numrepresented x} \coordsize^\numvars \int_{\RR^\numvars} \e{\coordsize^2 x F(\bvec{m})} \Upspsi(\bvec{m}) \ d \bvec{m} \\
		&\qquad\qquad \times \sum_{\substack{Q < d \le q+Q \\ \gcd(d, q) = 1}} \e{- \numrepresented \frac{\modmultinv{d}}{q}} \sum_{\bvec{h} \in (\ZZ/q\ZZ)^\numvars} \e{\frac{\modmultinv{d}}{q} F(\bvec{h})} \ d x \\
		&\qquad + \sum_{1 \le q \le Q} \frac{1}{q^\numvars} \int_{0}^{\frac{1}{q (q + Q)}} \e{\numrepresented x} \coordsize^\numvars \int_{\RR^\numvars} \e{- \coordsize^2 x F(\bvec{m})} \Upspsi(\bvec{m}) \ d \bvec{m} \\
		&\qquad\qquad \times \sum_{\substack{Q < d \le q+Q \\ \gcd(d, q) = 1}} \e{- \numrepresented \frac{\modmultinv{d}}{q}} \sum_{\bvec{h} \in (\ZZ/q\ZZ)^\numvars} \e{\frac{\modmultinv{d}}{q} F(\bvec{h})} \ d x . \numberthis \label{eq:M\coordsize \numrepresented expand2}
\end{align*}
By mapping $x$ to $-x$ in the second integral in $x$ and simplifying, we see that 
\begin{align*}
M_{F, \Upspsi, \coordsize} (\numrepresented) 
	&= \sum_{1 \le q \le Q} \frac{1}{q^\numvars} \sum_{\substack{Q < d \le q+Q \\ \gcd(d, q) = 1}} \e{- \numrepresented \frac{\modmultinv{d}}{q}} \sum_{\bvec{h} \in (\ZZ/q\ZZ)^\numvars} \e{\frac{\modmultinv{d}}{q} F(\bvec{h})} \\
		&\qquad\qquad \times \int_{-\frac{1}{q (q + Q)}}^{\frac{1}{q (q + Q)}} \e{-\numrepresented x} \coordsize^\numvars \int_{\RR^\numvars} \e{\coordsize^2 x F(\bvec{m})} \Upspsi(\bvec{m}) \ d \bvec{m} \ d x . \numberthis \label{eq:M\coordsize \numrepresented expand3}
\end{align*}
Because the sum over $d$ is a complete sum that only depends on $d$ modulo $q$, we can map $\modmultinv{d}$ to $d$ in \eqref{eq:M\coordsize \numrepresented expand3} and obtain 
\begin{align*}
M_{F, \Upspsi, \coordsize} (\numrepresented) 
	&= \sum_{1 \le q \le Q} \frac{1}{q^\numvars} \sum_{d \in (\ZZ/q\ZZ)^\times} \sum_{\bvec{h} \in (\ZZ/q\ZZ)^\numvars} \e{\frac{d}{q} \left( F(\bvec{h}) - \numrepresented \right)} \\
		&\qquad\qquad \times \int_{-\frac{1}{q (q + Q)}}^{\frac{1}{q (q + Q)}} \e{-\numrepresented x} \coordsize^\numvars \int_{\RR^\numvars} \e{\coordsize^2 x F(\bvec{m})} \Upspsi(\bvec{m}) \ d \bvec{m} \ d x . \numberthis \label{eq:M\coordsize \numrepresented expand4}
\end{align*}

To provide an asymptotic for $M_{F, \Upspsi, \coordsize} (\numrepresented)$ with an appropriate error term, we will need an upper bound for the absolute value of the sum 
\begin{align*}
\sum_{d \in (\ZZ/q\ZZ)^\times} \sum_{\bvec{h} \in (\ZZ/q\ZZ)^\numvars} \e{\frac{d}{q} \left( F(\bvec{h}) - \numrepresented \right)}. 
\end{align*}
The next lemma gives such an upper bound.
\begin{lem} \label{lem:dhmodqsum upperbound}
If $q$ is a positive integer and $\numrepresented$ is an integer, then
\begin{align*}
&\sum_{d \in (\ZZ/q\ZZ)^\times} \sum_{\bvec{h} \in (\ZZ/q\ZZ)^\numvars} \e{\frac{d}{q} \left( F(\bvec{h}) - \numrepresented \right)} \\
	&\ll (\gcd(\lvl,q_0))^{\numvars/2} (\gcd(\numrepresented, q_1))^{1/2} q_0^{1/2} q^{(\numvars+1)/2} \tau(q) \log(2q) . \numberthis \label{ineq:dhmodqsum upperbound}
\end{align*}
\end{lem}
\begin{proof}
If $x < \frac{1}{q (q + Q)}$, then 
\begin{align}
T_{\bvec{0}} (q, \numrepresented; x) &= \sum_{\substack{Q < d \le q+Q \\ \gcd(d, q) = 1}} \e{\numrepresented \frac{\modmultinv{d}}{q}} \sum_{\bvec{h} \in (\ZZ/q\ZZ)^\numvars} \e{\frac{ - \modmultinv{d}}{q} F(\bvec{h})} \label{eq:T0q\numrepresented smallx 1} \\
	&= \sum_{d \in (\ZZ/q\ZZ)^\times} \sum_{\bvec{h} \in (\ZZ/q\ZZ)^\numvars} \e{\frac{d}{q} \left( F(\bvec{h}) - \numrepresented \right)} . \label{eq:T0q\numrepresented smallx 2}
\end{align}
Therefore, we can use Lemma~\ref{lem:Trq\numrepresented x bound} to provide an upper bound for the absolute value of $\sum_{d \in (\ZZ/q\ZZ)^\times} \sum_{\bvec{h} \in (\ZZ/q\ZZ)^\numvars} \e{\frac{d}{q} \left( F(\bvec{h}) - \numrepresented \right)}$ and obtain \eqref{ineq:dhmodqsum upperbound}.
\end{proof}

The rest of this subsection is devoted to providing an asymptotic for $M_{F, \Upspsi, \coordsize} (\numrepresented)$.

\subsubsection{Extending to the singular integral} \label{subsec:extendsingint}

For a bump function $\Upspsi \in C_c^\infty (\RR^\numvars)$, positive real numbers $\coordsize$ and $B$, and a real number $\numrepresented$, 
we define the \emph{truncated singular integral} $J_{F, \Upspsi} (\numrepresented, \coordsize; B)$ to be 
\begin{align}
J_{F, \Upspsi} (\numrepresented, \coordsize; B) &= \int_{-B}^{B} \e{-\numrepresented x} \coordsize^\numvars \int_{\RR^\numvars} \e{\coordsize^2 x F(\bvec{m})} \Upspsi(\bvec{m}) \ d \bvec{m} \ d x . \label{eq:truncsingint def}
\end{align}
Notice that the truncated singular integral $J_{F, \Upspsi} \left( \numrepresented, \coordsize; \frac{1}{q (q + Q)} \right)$ appears in \eqref{eq:M\coordsize \numrepresented expand4} so that
\begin{align}
M_{F, \Upspsi, \coordsize} (\numrepresented) 
	&= \sum_{1 \le q \le Q} \frac{1}{q^\numvars} \sum_{d \in (\ZZ/q\ZZ)^\times} \sum_{\bvec{h} \in (\ZZ/q\ZZ)^\numvars} \e{\frac{d}{q} \left( F(\bvec{h}) - \numrepresented \right)} J_{F, \Upspsi} \left( \numrepresented, \coordsize; \frac{1}{q (q + Q)} \right) . \numberthis \label{eq:M\coordsize \numrepresented with truncsingint}
\end{align}
The first step towards providing an asymptotic for $M_{F, \Upspsi, \coordsize} (\numrepresented)$ is to extend (up to some acceptable error term) the truncated singular integral $J_{F, \Upspsi} \left( \numrepresented, \coordsize; \frac{1}{q (q + Q)} \right)$
 to the \emph{singular integral} 
\begin{align}
J_{F, \Upspsi} (\numrepresented, \coordsize) &= \coordsize^\numvars \int_{-\infty}^{\infty} \e{-\numrepresented x} \int_{\RR^\numvars} \e{\coordsize^2 x F(\bvec{m})} \Upspsi(\bvec{m}) \ d \bvec{m} \ d x , \label{eq:singint def}
\end{align}
where $\Upspsi \in C_c^\infty (\RR^\numvars)$, $\coordsize > 0$, and $\numrepresented \in \RR$.
We do this in the following lemma.
\begin{lem} \label{lem:truncsingintB to singint}
For a bump function $\Upspsi \in C_c^\infty (\RR^\numvars)$, positive real numbers $\coordsize$ and $B$, and a real number $\numrepresented$, we have 
\begin{align}
J_{F, \Upspsi} (\numrepresented, \coordsize; B) &= J_{F, \Upspsi} (\numrepresented, \coordsize) +  O_{\Upspsi , \numvars} \left( |\det(A)|^{-1/2} B^{1 - \frac{\numvars}{2}} \right) \label{eq:truncsingintB to singint}
\end{align}
\end{lem}
\begin{proof}
By \eqref{eq:IFx\coordsize rq const bump}, the difference between the singular integral and the truncated singular integral $J_{F, \Upspsi} (\numrepresented, \coordsize; B)$ is equal to 
\begin{align}
J_{F, \Upspsi} (\numrepresented, \coordsize) - J_{F, \Upspsi} (\numrepresented, \coordsize; B) &= \int_{|x| > B} \e{-\numrepresented x} \mathcal{I}_{F, \Upspsi} (x, \coordsize, \bvec{0}, q) \ d x . \label{eq:singintdiff}
\end{align}
We apply Theorem~\ref{thm:IFx\coordsize rq PSP bound} to this difference and obtain 
\begin{align*}
J_{F, \Upspsi} (\numrepresented, \coordsize) - J_{F, \Upspsi} (\numrepresented, \coordsize; B) &\ll_{\Upspsi} \int_{|x| > B} |x|^{-\numvars/2} |\det(A)|^{-1/2} \ d x \\
	&= 2 |\det(A)|^{-1/2} \int_{B}^\infty x^{-\numvars/2} \ d x \\
	&= \frac{2}{1 - \frac{\numvars}{2}} |\det(A)|^{-1/2} B^{1 - \frac{\numvars}{2}} . 
\end{align*}
This gives \eqref{eq:truncsingintB to singint} with an implied constant of $\frac{2}{1 - \frac{\numvars}{2}}$.
\end{proof}

We use Lemma~\ref{lem:truncsingintB to singint} with $B = \frac{1}{q (q + Q)}$ to conclude that 
\begin{align*}
J_{F, \Upspsi} \left( \numrepresented, \coordsize; \frac{1}{q (q + Q)} \right) &= J_{F, \Upspsi} (\numrepresented, \coordsize) + O_{\Upspsi , \numvars} \left( |\det(A)|^{-1/2} (q (q + Q))^{\numvars/2 - 1} \right). \numberthis \label{eq:truncsingint to singint}
\end{align*}
Substituting this into \eqref{eq:M\coordsize \numrepresented with truncsingint}, we see that 
\begin{align*}
M_{F, \Upspsi, \coordsize} (\numrepresented) 
	&= \sum_{1 \le q \le Q} \frac{1}{q^\numvars} \sum_{d \in (\ZZ/q\ZZ)^\times} \sum_{\bvec{h} \in (\ZZ/q\ZZ)^\numvars} \e{\frac{d}{q} \left( F(\bvec{h}) - \numrepresented \right)} J_{F, \Upspsi} (\numrepresented, \coordsize) \\
		&\qquad + O_{\Upspsi , \numvars} \Bigg( \Bigg| \sum_{1 \le q \le Q} \frac{1}{q^\numvars} \sum_{d \in (\ZZ/q\ZZ)^\times} \sum_{\bvec{h} \in (\ZZ/q\ZZ)^\numvars} \e{\frac{d}{q} \left( F(\bvec{h}) - \numrepresented \right)} \\
		&\qquad\qquad\qquad\qquad\qquad\qquad \times |\det(A)|^{-1/2} (q (q + Q))^{\numvars/2 - 1} \Bigg| \Bigg) . \numberthis \label{eq:M\coordsize \numrepresented with singint1}
\end{align*}

By Lemma~\ref{lem:dhmodqsum upperbound} and the fact that $\numvars \ge 4$, we have
\begin{align*}
&\sum_{1 \le q \le Q} \frac{1}{q^\numvars} \sum_{d \in (\ZZ/q\ZZ)^\times} \sum_{\bvec{h} \in (\ZZ/q\ZZ)^\numvars} \e{\frac{d}{q} \left( F(\bvec{h}) - \numrepresented \right)} |\det(A)|^{-1/2} (q (q + Q))^{\numvars/2 - 1} \\
	&\ll \sum_{1 \le q \le Q} (\gcd(\lvl,q_0))^{\numvars/2} (\gcd(\numrepresented, q_1))^{1/2} q_0^{1/2} q^{(1 - \numvars) / 2} \\
		&\qquad\qquad \times \tau(q) \log(2q) |\det(A)|^{-1/2} (q (q + Q))^{\numvars/2 - 1} \\
	&\le \lvl^{\numvars/2} |\det(A)|^{-1/2} (2 Q)^{\numvars/2 - 1} \sum_{1 \le q \le Q} (\gcd(\numrepresented, q_1))^{1/2} q_1^{-1/2} \tau(q) \log(2q) . \numberthis \label{ineq:extendingsinginterror1}
\end{align*}
We apply Lemma~\ref{lem:q1gcdsum+eps} to \eqref{ineq:extendingsinginterror1} to obtain 
\begin{align*}
&\sum_{1 \le q \le Q} \frac{1}{q^\numvars} \sum_{d \in (\ZZ/q\ZZ)^\times} \sum_{\bvec{h} \in (\ZZ/q\ZZ)^\numvars} \e{\frac{d}{q} \left( F(\bvec{h}) - \numrepresented \right)} |\det(A)|^{-1/2} (q (q + Q))^{\numvars/2 - 1} \\
	&\ll_\varepsilon \lvl^{\numvars/2} |\det(A)|^{-1/2} 2^{\numvars/2 - 1} Q^{(\numvars - 1) / 2 + \varepsilon} \tau(\numrepresented) \prod_{p \mid 2 \det(A)} ( 1 - p^{-1/2} )^{-1} \numberthis \label{ineq:extendingsinginterror2}
\end{align*}
for any $\varepsilon > 0$.

Substituting \eqref{ineq:extendingsinginterror2} into \eqref{eq:M\coordsize \numrepresented with singint1}, we conclude that
\begin{align*}
M_{F, \Upspsi, \coordsize} (\numrepresented) 
	&= \sum_{1 \le q \le Q} \frac{1}{q^\numvars} \sum_{d \in (\ZZ/q\ZZ)^\times} \sum_{\bvec{h} \in (\ZZ/q\ZZ)^\numvars} \e{\frac{d}{q} \left( F(\bvec{h}) - \numrepresented \right)} J_{F, \Upspsi} (\numrepresented, \coordsize) \\
		&\qquad + O_{\Upspsi , \numvars , \varepsilon} \left(\lvl^{\numvars/2} |\det(A)|^{-1/2} Q^{(\numvars - 1) / 2 + \varepsilon} \tau(\numrepresented) \prod_{p \mid 2 \det(A)} ( 1 - p^{-1/2} )^{-1} \right) . \numberthis \label{eq:M\coordsize \numrepresented with singint2}
\end{align*}

\subsubsection{Evaluating the singular integral} \label{subsec:evalsingint}

In this subsection, we evaluate the singular integral under certain conditions. 
Only in this subsection (Subsection~\ref{subsec:evalsingint}), the nonsingular form~$F$ in $\numvars$ variables might not be integral. All quantities used in this subsection that involve $F$ make sense if $F$ is not integral. 

We first apply to \eqref{eq:singint def} the change of variables $x \mapsto x / \coordsize^2$ to obtain 
\begin{align}
J_{F, \Upspsi} (\numrepresented, \coordsize) &= \coordsize^{\numvars - 2} \int_{-\infty}^{\infty} \int_{\RR^\numvars} \e{x \left( F(\bvec{m}) - \frac{\numrepresented}{\coordsize^2} \right)} \Upspsi(\bvec{m}) \ d \bvec{m} \ d x . \label{eq:singintnormalized}
\end{align}
Let $\tilde{\sigma}_{F, \Upspsi, \infty} (\numrepresented, \coordsize)$ be the quantity
\begin{align}
\tilde{\sigma}_{F, \Upspsi, \infty} (\numrepresented, \coordsize) &= \int_{-\infty}^{\infty} \int_{\RR^\numvars} \e{x \left( F(\bvec{m}) - \frac{\numrepresented}{\coordsize^2} \right)} \Upspsi(\bvec{m}) \ d \bvec{m} \ d x \label{eq:archmediandensity def}
\end{align}
so that 
\begin{align}
J_{F, \Upspsi} (\numrepresented, \coordsize) &= \tilde{\sigma}_{F, \Upspsi, \infty} (\numrepresented, \coordsize) \coordsize^{\numvars - 2} . \label{eq:singintnormalized with density}
\end{align}
We now prove the following theorem about $\tilde{\sigma}_{F, \Upspsi, \infty} (\numrepresented, \coordsize)$.
\begin{thm} \label{thm:archmediandensitytovol}
For a bump function $\psi \in C_c^\infty (\RR)$, a real number $\numrepresented$,  and a positive real number $\coordsize$, we have
\begin{align}
\tilde{\sigma}_{F, \Upspsi, \infty} (\numrepresented, \coordsize) &= \sigma_{F, \Upspsi, \infty} (\numrepresented, \coordsize) , \label{eq:singint archmediandensity vol}
\end{align}
where $\sigma_{F, \Upspsi, \infty} (\numrepresented, \coordsize)$ is 
the real factor defined in \eqref{eq:archmediandensity vol}. 
\end{thm}
\begin{proof}
We use tent functions to create continuous approximations to the indicator function $\Ind{|x| < \varepsilon}$, where $\varepsilon > 0$.

For $x \in \RR$, define the tent function $\mathfrak{t}$ by
\begin{align}
\mathfrak{t} (x) &= \max \{ 0 , 1 - |x| \} . \label{eq:tent}
\end{align}
For nonzero $x \in \RR$, define the $\sinc^2$ function by
\begin{align}
\sinc^2 (x) &= \left( \frac{\sin(\pi x)}{\pi x} \right)^2 . \label{eq:sinc}
\end{align}
Also, set $\sinc^2 (0) = 1$ so that $\sinc^2$ function is a continuous function on $\RR$.

It is well-known that the tent function $\mathfrak{t}$ is the Fourier transform of the $\sinc^2$ function and that the $\sinc^2$ function is the Fourier transform of the tent function $\mathfrak{t}$.
(See, for example, Appendix~2 of \cite{KammlerFourierAnal}.)

For $\eta > 0$, we define the function $\mathfrak{t}_{\eta}$ by
\begin{align}
\mathfrak{t}_{\eta} (x) &= \max \left\{ 0 , 1 - \frac{|x|}{\eta} \right\} . \label{eq:tenteta}
\end{align}
Using part~b of Theorem~8.22 in \cite{FollandRealAnalysis} about a scaling property for the Fourier transform, we find that the Fourier transform of $\mathfrak{t}_{\eta}$ is 
\begin{align}
w_{\eta} (x) &= \eta \left( \frac{\sin(\pi \eta x)}{\pi \eta x} \right)^2 \label{eq:weta}
\end{align}
and the Fourier transform of $w_{\eta}$ is $\mathfrak{t}_{\eta}$.

If $\eta, \delta > 0$, then we define the function $\mathfrak{T}_{\eta, \delta}$ by 
\begin{align}
\mathfrak{T}_{\eta, \delta} (x) &= \left( 1 + \frac{\eta}{\delta} \right) \mathfrak{t}_{\eta + \delta} (x) - \frac{\eta}{\delta} \mathfrak{t}_{\eta} (x) \label{eq:Tetadelta}
\end{align}
for $x \in \RR$.
After some manipulations, we find that
\begin{align}
\mathfrak{T}_{\eta, \delta} (x) &= 
	\begin{cases}
	1	&\text{if $|x| \le \eta$,} \\
	1 - \frac{|x| - \eta}{\delta}	&\text{if $\eta < |x| < \eta + \delta$,} \\
	0	&\text{if $|x| \ge \eta + \delta$.}
	\end{cases}
\label{eq:Tetadeltacases}
\end{align}

Let $0 < \varepsilon < 1$. For $x \in \RR$, define $\mathfrak{T}_{\varepsilon}^+$ and $\mathfrak{T}_{\varepsilon}^-$ to be
\begin{align}
\mathfrak{T}_{\varepsilon}^- &= \mathfrak{T}_{\varepsilon - \varepsilon^2, \varepsilon^2} = \varepsilon^{-1} \mathfrak{t}_{\varepsilon} + \left( 1 - \varepsilon^{-1} \right) \mathfrak{t}_{\varepsilon - \varepsilon^2} \label{eq:Teps-}
\end{align}
and
\begin{align}
\mathfrak{T}_{\varepsilon}^+ &= \mathfrak{T}_{\varepsilon, \varepsilon^2} = \left( 1 + \varepsilon^{-1} \right) \mathfrak{t}_{\varepsilon + \varepsilon^2} - \varepsilon^{-1} \mathfrak{t}_{\varepsilon} . \label{eq:Teps+}
\end{align}
Using \eqref{eq:Tetadeltacases}, we observe that 
\begin{align}
\mathfrak{T}_{\varepsilon}^- (x) \le \Ind{|x| < \varepsilon} \le \mathfrak{T}_{\varepsilon}^+ (x) \label{ineqs:indTeps bounds}
\end{align}
for all $x \in \RR$.
Therefore, $\mathfrak{T}_{\varepsilon}^- (x)$ provides a lower bound for $\Ind{|x| < \varepsilon}$, and $\mathfrak{T}_{\varepsilon}^+ (x)$ provides a upper bound for $\Ind{|x| < \varepsilon}$.

Let 
\begin{align*}
\beta_{F , \Upspsi , \numrepresented , \coordsize} (\varepsilon) &= \int_{\left| F(\bvec{m}) - \frac{\numrepresented}{\coordsize^2} \right| < \varepsilon} \Upspsi(\bvec{m}) \ d \bvec{m} .
\end{align*}
Since 
\begin{align*}
\int_{\left| F(\bvec{m}) - \frac{\numrepresented}{\coordsize^2} \right| < \varepsilon} \Upspsi(\bvec{m}) \ d \bvec{m} &= \int_{\RR^\numvars} \Ind{\left| F(\bvec{m}) - \frac{\numrepresented}{\coordsize^2} \right| < \varepsilon} \Upspsi(\bvec{m}) \ d \bvec{m} , 
\end{align*}
the inequalities in \eqref{ineqs:indTeps bounds} imply that 
\begin{align}
\int_{\RR^\numvars} \mathfrak{T}_{\varepsilon}^- \left( F(\bvec{m}) - \frac{\numrepresented}{\coordsize^2} \right) \Upspsi(\bvec{m}) \ d \bvec{m} 
	&\le \beta_{F , \Upspsi , \numrepresented , \coordsize} (\varepsilon)
	\le \int_{\RR^\numvars} \mathfrak{T}_{\varepsilon}^+ \left( F(\bvec{m}) - \frac{\numrepresented}{\coordsize^2} \right) \Upspsi(\bvec{m}) \ d \bvec{m} . \label{ineqs:intindTeps bounds}
\end{align}

Using \eqref{eq:Teps-} and \eqref{eq:Teps+}, we manipulate some of the integrals that appear in \eqref{ineqs:intindTeps bounds} and see that 
\begin{align}
\int_{\RR^\numvars} \mathfrak{T}_{\varepsilon}^- \left( F(\bvec{m}) - \frac{\numrepresented}{\coordsize^2} \right) \Upspsi(\bvec{m}) \ d \bvec{m} &= U_{F, \numrepresented, \coordsize} (\varepsilon) - (\varepsilon - 1)^2 U_{F, \numrepresented, \coordsize} (\varepsilon - \varepsilon^2) \label{eq:intTeps- with Ueta}
\end{align}
and 
\begin{align}
\int_{\RR^\numvars} \mathfrak{T}_{\varepsilon}^+ \left( F(\bvec{m}) - \frac{\numrepresented}{\coordsize^2} \right) \Upspsi(\bvec{m}) \ d \bvec{m} 
	&= (\varepsilon + 1)^2 U_{F, \numrepresented, \coordsize} (\varepsilon + \varepsilon^2) - U_{F, \numrepresented, \coordsize} (\varepsilon) , \label{eq:intTeps+ with Ueta}
\end{align}
where
\begin{align}
U_{F, \numrepresented, \coordsize} (\eta) &= \int_{\RR^\numvars} \eta^{-1} \mathfrak{t}_{\eta} \left( F(\bvec{m}) - \frac{\numrepresented}{\coordsize^2} \right) \Upspsi(\bvec{m}) \ d \bvec{m} \label{eq:Ueta}
\end{align}
for $\eta > 0$.

In light of \eqref{eq:intTeps- with Ueta} and \eqref{eq:intTeps+ with Ueta}, we provide an asymptotic for $U_{F, \numrepresented, \coordsize} (\eta)$ when $\eta$ is sufficiently close to $\varepsilon$.
\begin{lem} \label{lem:Uetaeps asymp}
Let $\varepsilon > 0$. 
For $\eta \in \RR$ with $| \eta - \varepsilon | < \varepsilon$, we have 
\begin{align}
U_{F, \numrepresented, \coordsize} (\eta) &= \tilde{\sigma}_{F, \Upspsi, \infty} (\numrepresented, \coordsize) + O_{F , \Upspsi} \left( \varepsilon^{(2 \numvars - 4) / (\numvars + 4)} \right) . \label{eq:Ueta sigmatilde asymp}
\end{align}
\end{lem}
\begin{proof}[Proof of Lemma~\ref{lem:Uetaeps asymp}]
Using the inverse Fourier transform and the fact that $\widehat{\mathfrak{t}}_{\eta} = w_{\eta}$, we find that
\begin{align*}
U_{F, \numrepresented, \coordsize} (\eta) &= \int_{\RR^\numvars} \int_{-\infty}^{\infty} \eta^{-1} w_{\eta} (x) \e{x \left( F(\bvec{m}) - \frac{\numrepresented}{\coordsize^2} \right)} \Upspsi(\bvec{m}) \ d x \ d \bvec{m} .
\end{align*}
Applying the Fubini-Tonelli theorem, we see that 
\begin{align}
U_{F, \numrepresented, \coordsize} (\eta) &= \int_{-\infty}^{\infty} \int_{\RR^\numvars} \eta^{-1} w_{\eta} (x) \e{x \left( F(\bvec{m}) - \frac{\numrepresented}{\coordsize^2} \right)} \Upspsi(\bvec{m}) \ d \bvec{m} \ d x . \label{eq:Ueta with weta}
\end{align}
Therefore, 
\begin{align*}
&U_{F, \numrepresented, \coordsize} (\eta) - \tilde{\sigma}_{F, \Upspsi, \infty} (\numrepresented, \coordsize) \\
&= \int_{-\infty}^{\infty} \left( \eta^{-1} w_{\eta} (x) - 1 \right) \int_{\RR^\numvars} \e{x \left( F(\bvec{m}) - \frac{\numrepresented}{\coordsize^2} \right)} \Upspsi(\bvec{m}) \ d \bvec{m} \ d x . \numberthis \label{eq:Ueta minus sigmatilde}
\end{align*}
In order to prove \eqref{eq:Ueta sigmatilde asymp}, it suffices to prove \eqref{eq:Ueta minus sigmatilde} is $O_{F , \Upspsi} \left( \varepsilon^{(2 \numvars - 4) / (\numvars + 4)} \right)$. To do this, we split up the integral over $x$ into two regions: one in which $|x| \le \varepsilon^\delta$ and the other in which $|x| > \varepsilon^\delta$, where $\delta \in \RR$ will be chosen later. Let $\mathfrak{D} = \left[ - \varepsilon^\delta , \varepsilon^\delta \right]$ so that the two regions under consideration for $x$ are $\mathfrak{D}$ and $\RR \setminus \mathfrak{D}$.

Observe that 
\begin{align}
0 \le 1 - \eta^{-1} w_{\eta} (x) \le 1 . \label{ineqs:1-weta bounds}
\end{align}
From a Taylor series approximation of $w_{\eta} (x)$, we see that 
\begin{align}
0 \le 1 - \eta^{-1} w_{\eta} (x) \ll \min \{ 1 , \eta^2 x^2 \} \ll \min \{ 1 , \varepsilon^2 x^2 \} \label{ineqs:1-weta taylor bounds}
\end{align}
since $0 < \eta < 2 \varepsilon$.

Now
\begin{align*}
&\left| \int_{\mathfrak{D}} \left( \eta^{-1} w_{\eta} (x) - 1 \right) \int_{\RR^\numvars} \e{x \left( F(\bvec{m}) - \frac{\numrepresented}{\coordsize^2} \right)} \Upspsi(\bvec{m}) \ d \bvec{m} \ d x \right| \\
	&\le \int_{\mathfrak{D}} \left| \eta^{-1} w_{\eta} (x) - 1 \right| \int_{\RR^\numvars} | \Upspsi(\bvec{m}) | \ d \bvec{m} \ d x . \numberthis \label{ineq:Ueta minus sigmatilde int|x|small upperbound1}
\end{align*}
Using \eqref{ineqs:1-weta taylor bounds}, we obtain
\begin{align*}
&\left| \int_{\mathfrak{D}} \left( \eta^{-1} w_{\eta} (x) - 1 \right) \int_{\RR^\numvars} \e{x \left( F(\bvec{m}) - \frac{\numrepresented}{\coordsize^2} \right)} |\Upspsi(\bvec{m})| \ d \bvec{m} \ d x \right| \\
	&\ll \int_{- \varepsilon^\delta}^{\varepsilon^\delta} \varepsilon^2 x^2 \int_{\RR^\numvars} |\Upspsi(\bvec{m})| \ d \bvec{m} \ d x \\
	&\ll_{\Upspsi} \varepsilon^{2 + 3 \delta} . \numberthis \label{ineq:Ueta minus sigmatilde int|x|small upperbound2}
\end{align*}

Now we look at the region in which $|x| > \varepsilon^\delta$. By \eqref{ineqs:1-weta bounds}, we have
\begin{align*}
&\left| \int_{\RR \setminus \mathfrak{D}} \left( \eta^{-1} w_{\eta} (x) - 1 \right) \int_{\RR^\numvars} \e{x \left( F(\bvec{m}) - \frac{\numrepresented}{\coordsize^2} \right)} \Upspsi(\bvec{m}) \ d \bvec{m} \ d x \right| \\
	&\le \int_{\RR \setminus \mathfrak{D}} \left| \int_{\RR^\numvars} \e{x \left( F(\bvec{m}) - \frac{\numrepresented}{\coordsize^2} \right)} \Upspsi(\bvec{m}) \ d \bvec{m} \right| \ d x . \numberthis \label{ineq:Ueta minus sigmatilde int|x|large upperbound1}
\end{align*}
We apply Theorem~\ref{thm:quadoscillatorybound} to \eqref{ineq:Ueta minus sigmatilde int|x|large upperbound1} to obtain
\begin{align*}
&\left| \int_{\RR \setminus \mathfrak{D}} \left( \eta^{-1} w_{\eta} (x) - 1 \right) \int_{\RR^\numvars} \e{x \left( F(\bvec{m}) - \frac{\numrepresented}{\coordsize^2} \right)} \Upspsi(\bvec{m}) \ d \bvec{m} \ d x \right| \\
	&\ll_{\Upspsi} \int_{\RR \setminus \mathfrak{D}} |x|^{-\numvars/2} |\det(A)|^{-1/2} \ d x . \numberthis \label{ineq:Ueta minus sigmatilde int|x|large upperbound2}
\end{align*}
Now 
\begin{align*}
\int_{\RR \setminus \mathfrak{D}} |x|^{-\numvars/2} |\det(A)|^{-1/2} \ d x &= 2 |\det(A)|^{-1/2} \int_{\varepsilon^\delta}^{\infty} x^{-\numvars/2} \ d x \\
	&\ll_{F} \varepsilon^{\delta (1 - \numvars/2)} . \numberthis \label{ineq:Ueta minus sigmatilde int|x|large upperbound3}
\end{align*}
We substitute \eqref{ineq:Ueta minus sigmatilde int|x|large upperbound3} into \eqref{ineq:Ueta minus sigmatilde int|x|large upperbound2} to obtain
\begin{align*}
&\int_{\RR \setminus \mathfrak{D}} \left( \eta^{-1} w_{\eta} (x) - 1 \right) \int_{\RR^\numvars} \e{x \left( F(\bvec{m}) - \frac{\numrepresented}{\coordsize^2} \right)} \Upspsi(\bvec{m}) \ d \bvec{m} \ d x \\
&\ll_{F , \Upspsi} \varepsilon^{\delta (1 - \numvars/2)} . \numberthis \label{ineq:Ueta minus sigmatilde int|x|large upperbound4}
\end{align*}

Combining \eqref{ineq:Ueta minus sigmatilde int|x|small upperbound2} and \eqref{ineq:Ueta minus sigmatilde int|x|large upperbound4} with \eqref{eq:Ueta minus sigmatilde}, we see that
\begin{align}
U_{F, \numrepresented, \coordsize} (\eta) - \tilde{\sigma}_{F, \Upspsi, \infty} (\numrepresented, \coordsize) &\ll_{F , \Upspsi} \varepsilon^{2 + 3 \delta} + \varepsilon^{\delta (1 - \numvars/2)} . \label{ineq:Ueta minus sigmatilde upperbound}
\end{align}
%
By choosing $\delta$ to be $- 4 / (\numvars + 4)$ in \eqref{ineq:Ueta minus sigmatilde upperbound}, we obtain \eqref{eq:Ueta sigmatilde asymp}.
\end{proof}

Because $\varepsilon^2 < \varepsilon$ when $0 < \varepsilon < 1$, Lemma~\ref{lem:Uetaeps asymp} applies to all instances of $U_{F, \numrepresented, \coordsize}$ in \eqref{eq:intTeps- with Ueta} and \eqref{eq:intTeps+ with Ueta}. By applying Lemma~\ref{lem:Uetaeps asymp} to \eqref{eq:intTeps- with Ueta} and \eqref{eq:intTeps+ with Ueta}, we find that 
\begin{align}
\int_{\RR^\numvars} \mathfrak{T}_{\varepsilon}^- \left( F(\bvec{m}) - \frac{\numrepresented}{\coordsize^2} \right) \Upspsi(\bvec{m}) \ d \bvec{m} &= (2 \varepsilon - \varepsilon^2) \left( \tilde{\sigma}_{F, \Upspsi, \infty} (\numrepresented, \coordsize) + O_{F , \Upspsi , \numvars} \left( \varepsilon^{(2 \numvars - 4) / (\numvars + 4)} \right) \right) \label{eq:intTeps- asymp}
\end{align}
and 
\begin{align}
\int_{\RR^\numvars} \mathfrak{T}_{\varepsilon}^+ \left( F(\bvec{m}) - \frac{\numrepresented}{\coordsize^2} \right) \Upspsi(\bvec{m}) \ d \bvec{m} 
	 &= (2 \varepsilon + \varepsilon^2) \left( \tilde{\sigma}_{F, \Upspsi, \infty} (\numrepresented, \coordsize) + O_{F , \Upspsi , \numvars} \left( \varepsilon^{(2 \numvars - 4) / (\numvars + 4)} \right) \right) . \label{eq:intTeps+ asymp}
\end{align}
By substituting \eqref{eq:intTeps- asymp} and \eqref{eq:intTeps+ asymp} into \eqref{ineqs:intindTeps bounds} and then dividing by $2 \varepsilon$, we obtain 
\begin{align*}
&\left( 1 - \frac{\varepsilon}{2} \right) \left( \tilde{\sigma}_{F, \Upspsi, \infty} (\numrepresented, \coordsize) + O_{F , \Upspsi , \numvars} \left( \varepsilon^{(2 \numvars - 4) / (\numvars + 4)} \right) \right) \\
	&\le \frac{1}{2 \varepsilon} \int_{\left| F(\bvec{m}) - \frac{\numrepresented}{\coordsize^2} \right| < \varepsilon} \Upspsi(\bvec{m}) \ d \bvec{m} \\
	&\le \left( 1 + \frac{\varepsilon}{2} \right) \left( \tilde{\sigma}_{F, \Upspsi, \infty} (\numrepresented, \coordsize) + O_{F , \Upspsi , \numvars} \left( \varepsilon^{(2 \numvars - 4) / (\numvars + 4)} \right) \right) . \numberthis \label{ineqs:intindTeps asymp bounds}
\end{align*}
Because $\numvars \ge 4$, we observe that $\frac{2 \numvars - 4}{\numvars + 4} \ge \frac{1}{2}$, so $\lim_{\varepsilon \to 0^+} \varepsilon^{(2 \numvars - 4) / (\numvars + 4)} = 0$.
Thus, by taking limits in \eqref{ineqs:intindTeps asymp bounds} as $\varepsilon \to 0^+$, we conclude that $\tilde{\sigma}_{F, \Upspsi, \infty} (\numrepresented, \coordsize) \le \sigma_{F, \Upspsi, \infty} (\numrepresented, \coordsize) \le \tilde{\sigma}_{F, \Upspsi, \infty} (\numrepresented, \coordsize)$. This implies \eqref{eq:singint archmediandensity vol}.
\end{proof}

Because we have determined that $\tilde{\sigma}_{F, \Upspsi, \infty} (\numrepresented, \coordsize) = \sigma_{F, \Upspsi, \infty} (\numrepresented, \coordsize)$ with Theorem~\ref{thm:archmediandensitytovol}, we now turn our attention to evaluating $\sigma_{F, \Upspsi, \infty} (\numrepresented, \coordsize)$. 
We first provide an upper bound for $\sigma_{F, \Upspsi, \infty} (\numrepresented, \coordsize)$. 
In order to do this, we first compute the density of real solutions to $ F(\bvec{m}) = \numrepresented$ when $F$ is a positive definite quadratic form in the following theorem. 
\begin{thm} \label{thm:limitshellvol}
Suppose that $F$ is a positive definite quadratic form in $\numvars \ge 1$ variables. Let $A$ be the Hessian matrix of $F$. Suppose that $\numrepresented$ is a real positive number and that $c$ is a real number. Then
\begin{align}
\lim_{\varepsilon \to 0^+} \frac{1}{2 \varepsilon} \int_{\left| F(\bvec{m}) - \numrepresented \right| < \varepsilon} c \ d \bvec{m} &= \frac{(2 \pi)^{\numvars/2} c}{\Gamma(\numvars/2) \sqrt{\det(A)}} \numrepresented^{\numvars/2 - 1} . \label{eq:limitshellvol}
\end{align}
\end{thm}
\begin{proof}
The integral on the left-hand side of \eqref{eq:limitshellvol} is 
\begin{align*}
\int_{\left| F(\bvec{m}) - \numrepresented \right| < \varepsilon} c \ d \bvec{m} &= c \int_{F(\bvec{m}) < \numrepresented + \varepsilon} 1 \ d \bvec{m} - c \int_{F(\bvec{m}) \le \numrepresented - \varepsilon} 1 \ d \bvec{m} \\
	&= c \int_{\bvec{m}^\top A \bvec{m} < 2 (\numrepresented + \varepsilon)} 1 \ d \bvec{m} - c \int_{\bvec{m}^\top A \bvec{m} \le 2 (\numrepresented - \varepsilon)} 1 \ d \bvec{m} . \numberthis \label{eq:limitshellvolsplit}
\end{align*}

The Cholesky decomposition of the positive definite matrix $A$ gives us 
\begin{align*}
A = B B^\top ,
\end{align*}
where $B$ is a lower triangular matrix with positive diagonal entries.
Note that $\det(B) = \sqrt{\det(A)} > 0$.

We perform the change of variables of $\bvec{m} \mapsto (B^\top)^{-1} \bvec{m}$ in \eqref{eq:limitshellvolsplit} to obtain
\begin{align*}
\int_{\left| F(\bvec{m}) - \numrepresented \right| < \varepsilon} c \ d \bvec{m} 
	&= \frac{c}{| \det(B) |} \int_{((B^\top)^{-1} \bvec{m})^\top (B B^\top) (B^\top)^{-1}\bvec{m} < 2 (\numrepresented + \varepsilon)} 1 \ d \bvec{m} \\
		&\qquad - \frac{c}{| \det(B) |} \int_{((B^\top)^{-1} \bvec{m})^\top (B B^\top) (B^\top)^{-1}\bvec{m} \le 2 (\numrepresented - \varepsilon)} 1 \ d \bvec{m} \\
	&= \frac{c}{| \det(B) |} \left( \int_{\bvec{m}^\top \bvec{m} < 2 (\numrepresented + \varepsilon)} 1 \ d \bvec{m} - \int_{\bvec{m}^\top \bvec{m} \le 2 (\numrepresented - \varepsilon)} 1 \ d \bvec{m} \right) . \numberthis \label{eq:limtshellvolsplit intermed1}
\end{align*}
Because $\det(B) = \sqrt{\det(A)} > 0$ and $\bvec{m}^\top \bvec{m} = \| \bvec{m} \|^2$, we find that \eqref{eq:limtshellvolsplit intermed1} is equivalent to 
\begin{align*}
\int_{\left| F(\bvec{m}) - \numrepresented \right| < \varepsilon} c \ d \bvec{m} &= \frac{c}{\sqrt{\det(A)}} \left( \int_{\| \bvec{m} \|^2 < 2 (\numrepresented + \varepsilon)} 1 \ d \bvec{m} - \int_{\| \bvec{m} \|^2 \le 2 (\numrepresented - \varepsilon)} 1 \ d \bvec{m} \right) \\
	&= \frac{c}{\sqrt{\det(A)}} \left( \int_{\| \bvec{m} \| < \sqrt{2 (\numrepresented + \varepsilon)}} 1 \ d \bvec{m} - \int_{\| \bvec{m} \| \le \sqrt{2 (\numrepresented - \varepsilon)}} 1 \ d \bvec{m} \right) \\
	&= \frac{c}{\sqrt{\det(A)}} \left( \Vol_\numvars \left( B^o_\numvars \left( \sqrt{2 (\numrepresented + \varepsilon)} \right) \right) - \Vol_\numvars \left( B_\numvars \left( \sqrt{2 (\numrepresented - \varepsilon)} \right) \right) \right) . \numberthis \label{eq:limitshellvolsplit vols}
\end{align*}
By \eqref{eq:\numvars-balloriginvol}, we have 
\begin{align}
\int_{\left| F(\bvec{m}) - \numrepresented \right| < \varepsilon} c \ d \bvec{m} &= \frac{c}{\sqrt{\det(A)}} \left( \frac{\pi^{\numvars/2}}{\Gamma \left( \numvars/2 + 1 \right)} (2 (\numrepresented + \varepsilon))^{\numvars / 2} - \frac{\pi^{\numvars/2}}{\Gamma \left( \numvars/2 + 1 \right)} (2 (\numrepresented - \varepsilon))^{\numvars / 2} \right) \label{eq:limitshellvolsplit volseval}
\end{align}
Therefore,
\begin{align}
\lim_{\varepsilon \to 0^+} \frac{1}{2 \varepsilon} \int_{\left| F(\bvec{m}) - \numrepresented \right| < \varepsilon} c \ d \bvec{m} &= \frac{\pi^{\numvars/2} c}{\Gamma \left( \numvars/2 + 1 \right) \sqrt{\det(A)}} \lim_{\varepsilon \to 0^+} \frac{(2 (\numrepresented + \varepsilon))^{\numvars / 2} - (2 (\numrepresented - \varepsilon))^{\numvars / 2}}{2 \varepsilon} . \label{limtshellvolsplit symmetricderiv}
\end{align}
We recognize that the limit on the right-hand side of \eqref{limtshellvolsplit symmetricderiv} is the (symmetric) derivative of $(2 x)^{\numvars / 2}$ evaluated at $x = \numrepresented$. The derivative of $(2 x)^{\numvars / 2}$ is $\numvars (2 x)^{\numvars / 2 - 1}$, so 
\begin{align*}
\lim_{\varepsilon \to 0^+} \frac{1}{2 \varepsilon} \int_{\left| F(\bvec{m}) - \numrepresented \right| < \varepsilon} c \ d \bvec{m} &= \frac{\frac{\numvars}{2} (2 \pi)^{\numvars/2} c}{\Gamma \left( \numvars/2 + 1 \right) \sqrt{\det(A)}} \numrepresented^{\numvars/2 - 1} . 
\end{align*}
By noticing that $\Gamma(\numvars/2 + 1) = \frac{\numvars}{2} \Gamma(\numvars/2)$ and simplifying, we obtain \eqref{eq:limitshellvol}.
\end{proof}

We now provide an upper bound for the absolute value of $\sigma_{F, \Upspsi, \infty} (\numrepresented, \coordsize)$.
\begin{thm} \label{thm:sigma upperbound nonsing}
Suppose that $F$ is a nonsingular quadratic form in $\numvars \ge 1$ variables. Let $A$ be the Hessian matrix of $F$. Suppose that $\numrepresented$ and $\coordsize$ are positive real numbers. Let $\numposevals$ be the number of positive eigenvalues of $A$. Suppose that $\Upspsi \in C_c^\infty (\RR^\numvars)$. Then 
\begin{align*}
&| \sigma_{F, \Upspsi, \infty} (\numrepresented, \coordsize) | \\
&\le 
\begin{cases}
0 &\text{if $\numposevals = 0$,} \\
\displaystyle{(2 \Upsuppvar_\Upspsi)^{\numvars - 1} \frac{(2 \pi)^{1/2} \coordsize}{\Gamma(1/2) \left( \lambda_1 \numrepresented \right)^{1/2}} \max_{\bvec{m} \in \RR^\numvars} | \Upspsi(\bvec{m}) |} &\text{if $\numposevals = 1$,} \\
\displaystyle{(2 \Upsuppvar_\Upspsi)^{\numvars - \numposevals} \frac{(2 \pi)^{\numposevals/2}}{\Gamma(\numposevals/2) \left( \prod_{j=1}^\numposevals \lambda_j \right)^{1/2}} \left(\frac{\numrepresented}{\coordsize^2} - \frac{{\Upsuppvar_\Upspsi}^2}{2} \sum_{j=\numposevals+1}^\numvars \lambda_j \right)^{\numposevals/2 - 1} \max_{\bvec{m} \in \RR^\numvars} | \Upspsi(\bvec{m}) |} &\text{if $\numposevals > 1$,}
\end{cases}
 \numberthis \label{ineq:sigma upperbound nonsing}
\end{align*}
where $\lambda_1, \lambda_2, \ldots, \lambda_\numposevals$ are the positive eigenvalues of $A$ and $\lambda_{\numposevals+1}, \lambda_{\numposevals+2}, \ldots, \lambda_\numvars$ are the negative eigenvalues of $A$.
\end{thm}
 \begin{proof}
 Suppose $\numposevals = 0$. Then $F$ is negative definite. Because $\numrepresented$ and $\coordsize$ are positive, this implies that $\sigma_{F, \Upspsi, \infty} (\numrepresented, \coordsize) = 0$, and \eqref{ineq:sigma upperbound nonsing} follows.

Suppose that $\numposevals \ge 1$.
By taking absolute values of both sides of \eqref{eq:archmediandensity vol}, we obtain 
\begin{align}
| \sigma_{F, \Upspsi, \infty} (\numrepresented, \coordsize) | &\le \lim_{\varepsilon \to 0^+} \frac{1}{2 \varepsilon} \int_{\left| F(\bvec{m}) - \frac{\numrepresented}{\coordsize^2} \right| < \varepsilon} | \Upspsi(\bvec{m}) | \ d \bvec{m} . \label{ineq:sigma upperbound nonsing1}
\end{align}
 
Using the spectral theorem for symmetric matrices, we can write the symmetric matrix $A$ as 
\begin{align*}
A = P^\top D P,
\end{align*}
where $P$ is an orthogonal matrix and $D = \diag(\lambda_1 , \ldots , \lambda_\numvars)$ is a diagonal matrix with the eigenvalues of $A$ as diagonal entries.
Without loss of generality, we assume that $\lambda_1 , \lambda_2 , \ldots, \lambda_\numposevals$ are positive and $\lambda_{\numposevals + 1} , \lambda_{\numposevals + 2} , \ldots, \lambda_\numvars$ are negative.

Let $\bvec{k} = P \bvec{m}$. Because $P$ is orthogonal, we know that $| \det(P) | = 1$.
Therefore, 
\begin{align*}
\lim_{\varepsilon \to 0^+} \frac{1}{2 \varepsilon} \int_{\left| F(\bvec{m}) - \frac{\numrepresented}{\coordsize^2} \right| < \varepsilon} | \Upspsi(\bvec{m}) | \ d \bvec{m} &= \lim_{\varepsilon \to 0^+} \frac{1}{2 \varepsilon} \int_{\left| \frac{1}{2} \bvec{m}^\top P^\top D P \bvec{m} - \frac{\numrepresented}{\coordsize^2} \right| < \varepsilon} | \Upspsi(\bvec{m}) | \ d \bvec{m} \\
	&= \lim_{\varepsilon \to 0^+} \frac{1}{2 \varepsilon} \int_{\left| \frac{1}{2} \bvec{k}^\top D  \bvec{k} - \frac{\numrepresented}{\coordsize^2} \right| < \varepsilon} | \Upspsi(P^{-1} \bvec{k}) | \frac{1}{| \det(P) |} \ d \bvec{k} \\
	&= \lim_{\varepsilon \to 0^+} \frac{1}{2 \varepsilon} \int_{\left| \frac{1}{2} \sum_{j=1}^\numvars \lambda_j {k_j}^2 - \frac{\numrepresented}{\coordsize^2} \right| < \varepsilon} | \Upspsi(P^{-1} \bvec{k}) | \ d \bvec{k}  . \numberthis \label{ineq:sigma upperbound nonsing2}
\end{align*}
Rearranging the terms in \eqref{ineq:sigma upperbound nonsing2}, we obtain 
\begin{align}
\lim_{\varepsilon \to 0^+} \frac{1}{2 \varepsilon} \int_{\left| F(\bvec{m}) - \frac{\numrepresented}{\coordsize^2} \right| < \varepsilon} | \Upspsi(\bvec{m}) | \ d \bvec{m} &= \lim_{\varepsilon \to 0^+} \frac{1}{2 \varepsilon} \int_{\left| \frac{1}{2} \sum_{j=1}^\numposevals \lambda_j {k_j}^2 - \left(\frac{\numrepresented}{\coordsize^2} - \frac{1}{2} \sum_{j=\numposevals+1}^\numvars \lambda_j {k_j}^2\right) \right| < \varepsilon} | \Upspsi(P^{-1} \bvec{k}) | \ d \bvec{k} .  \label{ineq:sigma upperbound nonsing3}
\end{align}

Let $\bvec{x} = (k_1, k_2, \ldots, k_\numposevals)^\top$ and $\bvec{y} = (k_{\numposevals+1}, k_2, \ldots, k_\numvars)^\top$ so that $\bvec{k} = \begin{pmatrix} \bvec{x} \\ \bvec{y} \end{pmatrix}$. Since $\supp(\Upspsi) \subseteq B_\numvars (\Upsuppvar_\Upspsi)$ and $P$ is orthogonal, we have 
\begin{align*}
&\lim_{\varepsilon \to 0^+} \frac{1}{2 \varepsilon} \int_{\left| \frac{1}{2} \sum_{j=1}^\numposevals \lambda_j {k_j}^2 - \left(\frac{\numrepresented}{\coordsize^2} - \frac{1}{2} \sum_{j=\numposevals+1}^\numvars \lambda_j {k_j}^2\right) \right| < \varepsilon} | \Upspsi(P^{-1} \bvec{k}) | \ d \bvec{k} \\
	&\le \int_{[-\Upsuppvar_\Upspsi , \Upsuppvar_\Upspsi]^{\numvars - \numposevals}} \ \lim_{\varepsilon \to 0^+} \frac{1}{2 \varepsilon} \int_{\left| \frac{1}{2} \sum_{j=1}^\numposevals \lambda_j {k_j}^2 - \left(\frac{\numrepresented}{\coordsize^2} - \frac{1}{2} \sum_{j=\numposevals+1}^\numvars \lambda_j {k_j}^2\right) \right| < \varepsilon} \max_{\bvec{w} \in \RR^\numvars} | \Upspsi(\bvec{w}) | \ d \bvec{x} \ d \bvec{y} .  \numberthis \label{ineq:sigma upperbound nonsing4}
\end{align*}

We apply Theorem~\ref{thm:limitshellvol} to \eqref{ineq:sigma upperbound nonsing4} to obtain 
\begin{align*}
&\lim_{\varepsilon \to 0^+} \frac{1}{2 \varepsilon} \int_{\left| \frac{1}{2} \sum_{j=1}^\numposevals \lambda_j {k_j}^2 - \left(\frac{\numrepresented}{\coordsize^2} - \frac{1}{2} \sum_{j=\numposevals+1}^\numvars \lambda_j {k_j}^2\right) \right| < \varepsilon} | \Upspsi(P^{-1} \bvec{k}) | \ d \bvec{k} \\
	&\le \int_{[-\Upsuppvar_\Upspsi , \Upsuppvar_\Upspsi]^{\numvars - \numposevals}} \frac{(2 \pi)^{\numposevals/2}}{\Gamma(\numposevals/2) \sqrt{\prod_{j=1}^\numposevals \lambda_j}} \left(\frac{\numrepresented}{\coordsize^2} - \frac{1}{2} \sum_{j=\numposevals+1}^\numvars \lambda_j {k_j}^2\right)^{\numposevals/2 - 1} \max_{\bvec{w} \in \RR^\numvars} | \Upspsi(\bvec{w}) | \ d \bvec{y} .  \numberthis \label{ineq:sigma upperbound nonsing5}
\end{align*}
since the Hessian matrix of the positive definite quadratic form $\frac{1}{2} \sum_{j=1}^\numposevals \lambda_j {k_j}^2$ is $\diag( \lambda_1, \ldots, \lambda_\numposevals)$ (which has a determinant of $\prod_{j=1}^\numposevals \lambda_j $).

If $\numposevals = 1$, then 
\begin{align*}
&\int_{[-\Upsuppvar_\Upspsi , \Upsuppvar_\Upspsi]^{\numvars - \numposevals}} \frac{(2 \pi)^{\numposevals/2}}{\Gamma(\numposevals/2) \sqrt{\prod_{j=1}^\numposevals \lambda_j}} \left(\frac{\numrepresented}{\coordsize^2} - \frac{1}{2} \sum_{j=\numposevals+1}^\numvars \lambda_j {k_j}^2\right)^{\numposevals/2 - 1} \max_{\bvec{w} \in \RR^\numvars} | \Upspsi(\bvec{w}) | \ d \bvec{y} \\
	&\le \int_{[-\Upsuppvar_\Upspsi , \Upsuppvar_\Upspsi]^{\numvars - 1}} \frac{(2 \pi)^{1/2}}{\Gamma(1/2) \sqrt{\lambda_1}} \left(\frac{\numrepresented}{\coordsize^2} \right)^{-1/2} \max_{\bvec{w} \in \RR^\numvars} | \Upspsi(\bvec{w}) | \ d \bvec{y} \\
	&\le (2 \Upsuppvar_\Upspsi)^{\numvars - 1} \frac{(2 \pi)^{1/2} \coordsize}{\Gamma(1/2) \left( \lambda_1 \numrepresented \right)^{1/2}} \max_{\bvec{w} \in \RR^\numvars} | \Upspsi(\bvec{w}) | .  \numberthis \label{ineq:sigma upperbound nonsing6}
\end{align*}

If $\numposevals > 1$, then 
\begin{align*}
&\int_{[-\Upsuppvar_\Upspsi , \Upsuppvar_\Upspsi]^{\numvars - \numposevals}} \frac{(2 \pi)^{\numposevals/2}}{\Gamma(\numposevals/2) \sqrt{\prod_{j=1}^\numposevals \lambda_j}} \left(\frac{\numrepresented}{\coordsize^2} - \frac{1}{2} \sum_{j=\numposevals+1}^\numvars \lambda_j {k_j}^2\right)^{\numposevals/2 - 1} \max_{\bvec{w} \in \RR^\numvars} | \Upspsi(\bvec{w}) | \ d \bvec{y} \\
	&\le \int_{[-\Upsuppvar_\Upspsi , \Upsuppvar_\Upspsi]^{\numvars - \numposevals}} \frac{(2 \pi)^{\numposevals/2}}{\Gamma(\numposevals/2) \left( \prod_{j=1}^\numposevals \lambda_j \right)^{1/2}} \left(\frac{\numrepresented}{\coordsize^2} - \frac{1}{2} \sum_{j=\numposevals+1}^\numvars \lambda_j {\Upsuppvar_\Upspsi}^2\right)^{\numposevals/2 - 1} \max_{\bvec{w} \in \RR^\numvars} | \Upspsi(\bvec{w}) | \ d \bvec{y} \\
	&\le (2 \Upsuppvar_\Upspsi)^{\numvars - \numposevals} \frac{(2 \pi)^{\numposevals/2}}{\Gamma(\numposevals/2) \left( \prod_{j=1}^\numposevals \lambda_j \right)^{1/2}} \left(\frac{\numrepresented}{\coordsize^2} - \frac{{\Upsuppvar_\Upspsi}^2}{2} \sum_{j=\numposevals+1}^\numvars \lambda_j \right)^{\numposevals/2 - 1} \max_{\bvec{w} \in \RR^\numvars} | \Upspsi(\bvec{w}) | .  \numberthis \label{ineq:sigma upperbound nonsing7}
\end{align*}

We obtain \eqref{ineq:sigma upperbound nonsing} by combining \eqref{ineq:sigma upperbound nonsing1}, \eqref{ineq:sigma upperbound nonsing3}, \eqref{ineq:sigma upperbound nonsing5}, \eqref{ineq:sigma upperbound nonsing6}, and \eqref{ineq:sigma upperbound nonsing7}.
\end{proof}

If $F$ is a positive definite form, then Theorem~\ref{thm:sigma upperbound nonsing} greatly simplifies since $\numposevals = \numvars$ and the determinant of a matrix is equal to the product of its eigenvalues. We state the result in the following corollary.
\begin{cor} \label{cor:sigma upperbound}
Suppose that $F$ is a positive definite quadratic form in $\numvars \ge 1$ variables. Let $A$ be the Hessian matrix of $F$. Suppose that $\numrepresented$ and $\coordsize$ are positive real numbers. Suppose that $\Upspsi \in C_c^\infty (\RR^\numvars)$. Then 
\begin{align}
| \sigma_{F, \Upspsi, \infty} (\numrepresented, \coordsize) | &\le \frac{(2 \pi)^{\numvars/2}}{\Gamma(\numvars/2) \sqrt{\det(A)}} \numrepresented^{\numvars/2 - 1} \coordsize^{2 - \numvars} \max_{\bvec{m} \in \RR^\numvars} |\Upspsi(\bvec{m})| . \label{ineq:sigma upperbound}
\end{align}
\end{cor}

Now that we have upper bounds for the absolute value of $\sigma_{F, \Upspsi, \infty} (\numrepresented, \coordsize)$, we give upper bounds for the absolute value of the singular integral $J_{F, \Upspsi} (\numrepresented, \coordsize)$. The next upper bound follows immediately from \eqref{eq:singintnormalized with density}, Theorem~\ref{thm:archmediandensitytovol}, and Theorem~\ref{thm:sigma upperbound nonsing}.
\begin{cor} \label{cor:singint upperbound nonsing}
Suppose that $F$ is a nonsingular quadratic form in $\numvars \ge 1$ variables. Let $A$ be the Hessian matrix of $F$. Suppose that $\numrepresented$ and $\coordsize$ are positive real numbers. Let $\numposevals$ be the number of positive eigenvalues of $A$. Suppose that $\Upspsi \in C_c^\infty (\RR^\numvars)$. Then 
\begin{align*}
&| J_{F, \Upspsi, \infty} (\numrepresented, \coordsize) | \\
&\le 
\begin{cases}
0 &\text{if $\numposevals = 0$,} \\
\displaystyle{(2 \Upsuppvar_\Upspsi)^{\numvars - 1} \frac{(2 \pi)^{1/2} \coordsize^{\numvars-1}}{\Gamma(1/2) \left( \lambda_1 \numrepresented \right)^{1/2}} \max_{\bvec{m} \in \RR^\numvars} | \Upspsi(\bvec{m}) |} &\text{if $\numposevals = 1$,} \\
\displaystyle{(2 \Upsuppvar_\Upspsi)^{\numvars - \numposevals} \frac{(2 \pi)^{\numposevals/2} \coordsize^{\numvars-2}}{\Gamma(\numposevals/2) \left( \prod_{j=1}^\numposevals \lambda_j \right)^{1/2}} \left(\frac{\numrepresented}{\coordsize^2} - \frac{{\Upsuppvar_\Upspsi}^2}{2} \sum_{j=\numposevals+1}^\numvars \lambda_j \right)^{\numposevals/2 - 1} \max_{\bvec{m} \in \RR^\numvars} | \Upspsi(\bvec{m}) |} &\text{if $\numposevals > 1$,}
\end{cases}
\end{align*}
where $\lambda_1, \lambda_2, \ldots, \lambda_\numposevals$ are the positive eigenvalues of $A$ and $\lambda_{\numposevals+1}, \lambda_{\numposevals+2}, \ldots, \lambda_\numvars$ are the negative eigenvalues of $A$.
\end{cor}

Analogous to the upper bounds for the absolute value of $\sigma_{F, \Upspsi, \infty} (\numrepresented, \coordsize)$, Corollary~\ref{cor:singint upperbound nonsing} greatly simplifies if $F$ is a positive definite form. 
The next corollary follows immediately from \eqref{eq:singintnormalized with density}, Theorem~\ref{thm:archmediandensitytovol}, and Corollary~\ref{cor:sigma upperbound}.
\begin{cor} \label{cor:singint upperbound}
Suppose that $F$ is a positive definite quadratic form in $\numvars \ge 1$ variables. Let $A$ be the Hessian matrix of $F$. Suppose that $\numrepresented$ and $\coordsize$ are real positive numbers. Suppose that $\Upspsi \in C_c^\infty (\RR^\numvars)$. Then 
\begin{align*}
| J_{F, \Upspsi} (\numrepresented, \coordsize) | \le \frac{(2 \pi)^{\numvars/2}}{\Gamma(\numvars/2) \sqrt{\det(A)}} \numrepresented^{\numvars/2 - 1} \max_{\bvec{m} \in \RR^\numvars} | \Upspsi(\bvec{m}) | . 
\end{align*}
\end{cor}

We would like to obtain an equality for the singular integral $J_{F, \Upspsi} (\numrepresented, \coordsize)$ under certain conditions. One such condition is requiring that $\Upspsi$ is equal to $1$ on the set $V_\coordsize = \{ \bvec{m} \in \RR^\numvars : F(\bvec{m}) = \numrepresented / \coordsize^2 \}$. 
Before we can give an equality under such a condition, we prove a lemma that allows us to give an upper bound on how close $V_1$ is to a point $\bvec{m} \in \RR^\numvars$ satisfying $| F(\bvec{m}) - \numrepresented | < \varepsilon$ when $F$ is positive definite.
\begin{lem} \label{lem:preimagecloseness}
Suppose that $F$ is a positive definite quadratic form in $\numvars$ variables. Suppose that $\bvec{m} \in \RR^\numvars$. Suppose that $\numrepresented$ is a real positive number and $0 < \varepsilon < \numrepresented$. Let $V = \{ \bvec{v} \in \RR^\numvars : F(\bvec{v}) = \numrepresented \}$. 
Let $\lambda_\numvars$ be the smallest eigenvalue of the Hessian matrix of $F$. 
Then 
the inequality $| F(\bvec{m}) - \numrepresented | < \varepsilon$ implies that there exists $\bvec{v} \in V$ such that 
\begin{align}
\| \bvec{m} - \bvec{v} \| &< \frac{\varepsilon}{\numrepresented - \varepsilon} \sqrt{\frac{\numrepresented + \varepsilon}{2 \lambda_\numvars}} . \label{ineq:preimagecloseness}
\end{align}
\end{lem}
\begin{proof}
Suppose that $| F(\bvec{m}) - \numrepresented | < \varepsilon$. This is equivalent to the statement 
\begin{align}
\numrepresented - \varepsilon < F(\bvec{m}) < \numrepresented + \varepsilon . \label{ineq:F(m)nepsilonbounds}
\end{align}
 Using Lemma~\ref{lem:preimageFbounded}, we find that 
\begin{align}
\| \bvec{m} \| < \sqrt{\frac{2 (\numrepresented + \varepsilon)}{\lambda_\numvars}} . \label{ineq:m upperbound}
\end{align}

Let 
\begin{align*}
\bvec{v} &= \sqrt{\frac{\numrepresented}{F(\bvec{m})}} \, \bvec{m} .
\end{align*}
Because $F$ is a quadratic form,
\begin{align*}
F(\bvec{v}) &= F \left( \sqrt{\frac{\numrepresented}{F(\bvec{m})}} \, \bvec{m} \right) \\
	&= \frac{\numrepresented}{F(\bvec{m})} F(\bvec{m}) 
	= \numrepresented .
\end{align*}
Therefore, $\bvec{v}$ is in $V$. 

We will prove the lemma by showing that \eqref{ineq:preimagecloseness} is satisfied by $\bvec{m}$ and this particular choice of $\bvec{v}$. 
Now
\begin{align*}
\| \bvec{m} - \bvec{v} \| &= \left\| \bvec{m} - \sqrt{\frac{\numrepresented}{F(\bvec{m})}} \, \bvec{m} \right\| 
	= \left| 1 - \sqrt{\frac{\numrepresented}{F(\bvec{m})}} \right| \| \bvec{m} \| \\
	&= \left| \frac{\sqrt{F(\bvec{m})} - \sqrt{\numrepresented}}{\sqrt{F(\bvec{m})}} \right| \| \bvec{m} \| \\
	&= \left| \frac{F(\bvec{m}) - \numrepresented}{\sqrt{F(\bvec{m})} (\sqrt{F(\bvec{m})} + \sqrt{\numrepresented})} \right| \| \bvec{m} \| . \numberthis \label{ineq:preimagecloseness intermed}
\end{align*}
By \eqref{ineq:F(m)nepsilonbounds}, \eqref{ineq:m upperbound}, and the fact that $| F(\bvec{m}) - \numrepresented | < \varepsilon$, we have
\begin{align*}
\left| \frac{F(\bvec{m}) - \numrepresented}{\sqrt{F(\bvec{m})} (\sqrt{F(\bvec{m})} + \sqrt{\numrepresented})} \right| \| \bvec{m} \| &< \frac{\varepsilon}{\sqrt{\numrepresented - \varepsilon} (\sqrt{\numrepresented - \varepsilon} + \sqrt{\numrepresented})} \sqrt{\frac{2 (\numrepresented + \varepsilon)}{\lambda_\numvars}} \\
	&< \frac{\varepsilon}{2 (\numrepresented - \varepsilon)} \sqrt{\frac{2 (\numrepresented + \varepsilon)}{\lambda_\numvars}}
\end{align*}
Combining this with \eqref{ineq:preimagecloseness intermed}, we obtain \eqref{ineq:preimagecloseness}.
\end{proof}

We use the previous lemma to prove the following theorem that gives an equality for $\sigma_{F, \Upspsi, \infty} (\numrepresented, \coordsize)$ when $F$ is positive definite and $\Upspsi$ is equal to $1$ on the set $V_\coordsize$.
\begin{thm} \label{thm:sigma asymp Upsilon=1 eval}
Suppose that $F$ is a positive definite quadratic form in $\numvars$ variables. Let $A$ be the Hessian matrix of $F$. Suppose that $\numrepresented$ and $\coordsize$ are real positive numbers. Suppose that $\Upspsi \in C_c^\infty (\RR^\numvars)$. 
Suppose that $\Upspsi(\bvec{m}) = 1$ whenever $F(\bvec{m}) = \numrepresented / \coordsize^2$. Then 
\begin{align}
\sigma_{F, \Upspsi, \infty} (\numrepresented, \coordsize) &= \frac{(2 \pi)^{\numvars/2}}{\Gamma(\numvars/2) \sqrt{\det(A)}} \numrepresented^{\numvars/2 - 1} \coordsize^{2 - \numvars} . \label{eq:sigma Upsilon=1 eval}
\end{align}
\begin{proof}
Since $V_\coordsize = \{ \bvec{m} \in \RR^\numvars : F(\bvec{m}) = \numrepresented / \coordsize^2 \}$, we have $\Upspsi(\bvec{m}) = 1$ if $\bvec{m} \in V_\coordsize$. From Theorem~\ref{thm:preimageFcompact}, we know that $V_\coordsize$ is compact.

Let $\eta > 0$. Because $\Upspsi$ is continuous and $V_\coordsize$ is compact, there exists $\delta > 0$ such that if $\dist( \bvec{m} , V_\coordsize ) < \delta$, then $| \Upspsi(\bvec{m}) - 1 | < \eta$. 

Observe that 
\begin{align*}
\lim_{\varepsilon \to 0^+} \frac{\varepsilon}{\frac{\numrepresented}{\coordsize^2} - \varepsilon} \sqrt{\frac{\frac{\numrepresented}{\coordsize^2} + \varepsilon}{2 \lambda_\numvars}} = 0 ,
\end{align*}
where $\lambda_\numvars$ is the smallest eigenvalue of $A$.
Thus, there exist infinitely many $\varepsilon > 0$ such that 
\begin{align}
\frac{\varepsilon}{\frac{\numrepresented}{\coordsize^2} - \varepsilon} \sqrt{\frac{\frac{\numrepresented}{\coordsize^2} + \varepsilon}{2 \lambda_1}} \le \delta . \label{ineq:epsdeltabound}
\end{align}

Choose $\varepsilon > 0$ so that \eqref{ineq:epsdeltabound} holds and $\varepsilon < \numrepresented / \coordsize^2$. Then Lemma~\ref{lem:preimagecloseness} implies that $\dist( \bvec{m} , V_\coordsize ) < \delta$ for all $\bvec{m} \in \RR^\numvars$ satisfying $\left| F(\bvec{m}) - \frac{\numrepresented}{\coordsize^2} \right| < \varepsilon$. Therefore, for all $\bvec{m} \in \RR^\numvars$ satisfying $\left| F(\bvec{m}) - \frac{\numrepresented}{\coordsize^2} \right| < \varepsilon$, we have $| \Upspsi(\bvec{m}) - 1 | < \eta$. Thus,
\begin{align}
\frac{1}{2 \varepsilon} \int_{\left| F(\bvec{m}) - \frac{\numrepresented}{\coordsize^2} \right| < \varepsilon} (1 - \eta) \ d \bvec{m} 
&\le \frac{1}{2 \varepsilon} \beta_{F , \Upspsi , \numrepresented , \coordsize} (\varepsilon) 
\le \frac{1}{2 \varepsilon} \int_{\left| F(\bvec{m}) - \frac{\numrepresented}{\coordsize^2} \right| < \varepsilon} (1 + \eta) \ d \bvec{m} , \label{ineqs:sigma Upsilon=1 etabounds}
\end{align}
where
\begin{align*}
\beta_{F , \Upspsi , \numrepresented , \coordsize} (\varepsilon) &= \int_{\left| F(\bvec{m}) - \frac{\numrepresented}{\coordsize^2} \right| < \varepsilon} \Upspsi(\bvec{m}) \ d \bvec{m} .
\end{align*}

Taking limits in \eqref{ineqs:sigma Upsilon=1 etabounds} as $\varepsilon \to 0^+$, we obtain
\begin{align}
\lim_{\varepsilon \to 0^+} \frac{1}{2 \varepsilon} \int_{\left| F(\bvec{m}) - \frac{\numrepresented}{\coordsize^2} \right| < \varepsilon} (1 - \eta) \ d \bvec{m} 
&\le \sigma_{F, \Upspsi, \infty} (\numrepresented, \coordsize) 
\le \lim_{\varepsilon \to 0^+} \frac{1}{2 \varepsilon} \int_{\left| F(\bvec{m}) - \frac{\numrepresented}{\coordsize^2} \right| < \varepsilon} (1 + \eta) \ d \bvec{m}  \label{ineqs:sigma Upsilon=1 etabounds lim}
\end{align}
since 
\begin{align*}
\sigma_{F, \Upspsi, \infty} (\numrepresented, \coordsize)  &= \lim_{\varepsilon \to 0^+} \frac{1}{2 \varepsilon} \beta_{F , \Upspsi , \numrepresented , \coordsize} (\varepsilon) .
\end{align*}
We apply Theorem~\ref{thm:limitshellvol} to \eqref{ineqs:sigma Upsilon=1 etabounds lim} to obtain
\begin{align}
\frac{(2 \pi)^{\numvars/2} (1 - \eta)}{\Gamma(\numvars/2) \sqrt{\det(A)}} \numrepresented^{\numvars/2 - 1} \coordsize^{2 - \numvars} 
&\le \sigma_{F, \Upspsi, \infty} (\numrepresented, \coordsize) 
\le \frac{(2 \pi)^{\numvars/2} (1 + \eta)}{\Gamma(\numvars/2) \sqrt{\det(A)}} \numrepresented^{\numvars/2 - 1} \coordsize^{2 - \numvars} . \label{ineqs:sigma Upsilon=1 etabounds2}
\end{align}
Because $\eta > 0$ was arbitrary, we obtain \eqref{eq:sigma Upsilon=1 eval} from \eqref{ineqs:sigma Upsilon=1 etabounds2}.
\end{proof}
\end{thm}

If $F$ is positive definite, the following corollary gives an equality for the singular integral $J_{F, \Upspsi} (\numrepresented, \coordsize)$ if $\Upspsi(\bvec{m}) = 1$ whenever $F(\bvec{m}) = \numrepresented / \coordsize^2$. It follows immediately from \eqref{eq:singintnormalized with density}, Theorem~\ref{thm:archmediandensitytovol}, and Theorem~\ref{thm:sigma asymp Upsilon=1 eval}.
\begin{cor} \label{cor:singint asymp Upsilon=1 eval}
Suppose that $F$ is a positive definite quadratic form in $\numvars$ variables. Let $A$ be the Hessian matrix of $F$. Suppose that $\numrepresented$ and $\coordsize$ are real positive numbers. Suppose that $\Upspsi \in C_c^\infty (\RR^\numvars)$. 
Suppose that $\Upspsi(\bvec{m}) = 1$ whenever $F(\bvec{m}) = \numrepresented / \coordsize^2$. Then 
\begin{align*}
J_{F, \Upspsi} (\numrepresented, \coordsize) &= \frac{(2 \pi)^{\numvars/2}}{\Gamma(\numvars/2) \sqrt{\det(A)}} \numrepresented^{\numvars/2 - 1} . 
\end{align*}
\end{cor}

\subsubsection{Extending to the singular series} \label{subsec:extendsingseries}

From \eqref{eq:M\coordsize \numrepresented with singint2}, we know that 
\begin{align*}
M_{F, \Upspsi, \coordsize} (\numrepresented) 
	&= \mathfrak{S}_F (\numrepresented; Q) J_{F, \Upspsi} (\numrepresented, \coordsize) \\
		&\qquad + O_{\Upspsi , \numvars , \varepsilon} \left(\lvl^{\numvars/2} |\det(A)|^{-1/2} Q^{(\numvars - 1) / 2 + \varepsilon} \tau(\numrepresented) \prod_{p \mid 2 \det(A)} ( 1 - p^{-1/2} )^{-1} \right) , \numberthis \label{eq:M\coordsize \numrepresented with singint and truncsingseries}
\end{align*}
where $\mathfrak{S}_F (\numrepresented; Q)$ is the \emph{truncated singular series}
\begin{align}
\mathfrak{S}_F (\numrepresented; Q) &= \sum_{1 \le q \le Q} \frac{1}{q^\numvars} \sum_{d \in (\ZZ/q\ZZ)^\times} \sum_{\bvec{h} \in (\ZZ/q\ZZ)^\numvars} \e{\frac{d}{q} \left( F(\bvec{h}) - \numrepresented \right)} . \label{eq:truncsingseries}
\end{align}

In the following lemma, 
we extend (up to some acceptable error term) the truncated singular series to the singular series $\mathfrak{S}_F (\numrepresented)$ as defined in \eqref{eq:singseries}. 
\begin{lem} \label{lem:truncsingseries to singseries}
For $Q \ge 1$, the truncated singular series $\mathfrak{S}_F (\numrepresented; Q)$ is 
\begin{align}
\mathfrak{S}_F (\numrepresented; Q) &= \mathfrak{S}_F (\numrepresented) + O_{\numvars , \varepsilon} \left( \lvl^{\numvars/2} Q^{(3 - \numvars) / 2 + \varepsilon} \tau(\numrepresented) \prod_{p \mid 2 \det(A)} ( 1 - p^{-1/2} )^{-1} \right) \label{eq:truncsingseries to singseries}
\end{align}
for any $\varepsilon > 0$.
\end{lem}
\begin{proof}
We apply Lemma~\ref{lem:dhmodqsum upperbound} to find that 
\begin{align*}
\mathfrak{S}_F (\numrepresented) - \mathfrak{S}_F (\numrepresented; Q) &= \sum_{q > Q} \frac{1}{q^\numvars} \sum_{d \in (\ZZ/q\ZZ)^\times} \sum_{\bvec{h} \in (\ZZ/q\ZZ)^\numvars} \e{\frac{d}{q} \left( F(\bvec{h}) - \numrepresented \right)} \\
	&\ll \sum_{q > Q} (\gcd(\lvl,q_0))^{\numvars/2} (\gcd(\numrepresented, q_1))^{1/2} q_0^{1/2} q^{(1 - \numvars) / 2} \tau(q) \log(2q) \\
	&\le \lvl^{\numvars/2} \sum_{q > Q} (\gcd(\numrepresented, q_1))^{1/2} q_1^{-1/2} q^{1 - \numvars/2} \tau(q) \log(2q) . \numberthis \label{ineq:singseriestailbound1}
\end{align*}
We now dyadically decompose the sum to obtain
\begin{align*}
&\sum_{q > Q} (\gcd(\numrepresented, q_1))^{1/2} q_1^{-1/2} q^{1 - \numvars/2} \tau(q) \log(2q) \\
	&= \sum_{k = 0}^\infty \sum_{2^k Q < q \le 2^{k+1} Q} (\gcd(\numrepresented, q_1))^{1/2} q_1^{-1/2} q^{1 - \numvars/2} \tau(q) \log(2q) .
\end{align*}
Because $\numvars \ge 4$, we find that 
\begin{align*}
&\sum_{q > Q} (\gcd(\numrepresented, q_1))^{1/2} q_1^{-1/2} q^{1 - \numvars/2} \tau(q) \log(2q) \\
	&\le \sum_{k = 0}^\infty (2^k Q)^{1 - \numvars/2} \sum_{2^k Q < q \le 2^{k+1} Q} (\gcd(\numrepresented, q_1))^{1/2} q_1^{-1/2} \tau(q) \log(2q) \\
	&\le \sum_{k = 0}^\infty (2^k Q)^{1 - \numvars/2} \sum_{1 \le q \le 2^{k+1} Q} (\gcd(\numrepresented, q_1))^{1/2} q_1^{-1/2} \tau(q) \log(2q) . \numberthis \label{ineq:singseriesdyadictailbound1}
\end{align*}
We apply Lemma~\ref{lem:q1gcdsum+eps} to \eqref{ineq:singseriesdyadictailbound1} to obtain 
\begin{align*}
&\sum_{q > Q} (\gcd(\numrepresented, q_1))^{1/2} q_1^{-1/2} q^{1 - \numvars/2} \tau(q) \log(2q) \\
&\ll_\varepsilon \sum_{k = 0}^\infty (2^k Q)^{1 - \numvars/2} (2^{k+1} Q)^{1/2 + \varepsilon} \tau(\numrepresented) \prod_{p \mid 2 \det(A)} ( 1 - p^{-1/2} )^{-1} \\
&= 2^{1/2 + \varepsilon} Q^{(3 - \numvars) / 2 + \varepsilon} \tau(\numrepresented) \left( \prod_{p \mid 2 \det(A)} ( 1 - p^{-1/2} )^{-1} \right) \sum_{k = 0}^\infty \left( 2^{(3 - \numvars) / 2 + \varepsilon} \right)^k \numberthis \label{ineq:singseriesdyadictailbound2} 
\end{align*}
for any $\varepsilon > 0$. The series $\sum_{k = 0}^\infty \left( 2^{(3 - \numvars) / 2 + \varepsilon} \right)^k$ is a geometric series and converges if $\varepsilon < (\numvars - 3)/2$.
If $\varepsilon < (\numvars - 3)/2$, then 
\begin{align}
\sum_{k = 0}^\infty \left( 2^{(3 - \numvars) / 2 + \varepsilon} \right)^k &= \left( 1 - 2^{(3 - \numvars) / 2 + \varepsilon} \right)^{-1} . \label{eq:geomseries2pow}
\end{align}

Substituting \eqref{eq:geomseries2pow} into \eqref{ineq:singseriesdyadictailbound2}, we obtain 
\begin{align*}
&\sum_{q > Q} (\gcd(\numrepresented, q_1))^{1/2} q_1^{-1/2} q^{1 - \numvars/2} \tau(q) \log(2q) \\
&\ll_\varepsilon 2^{1/2 + \varepsilon} Q^{(3 - \numvars) / 2 + \varepsilon} \tau(\numrepresented) \left( 1 - 2^{(3 - \numvars) / 2 + \varepsilon} \right)^{-1} \prod_{p \mid 2 \det(A)} ( 1 - p^{-1/2} )^{-1} \\
&\ll_{\numvars , \varepsilon} Q^{(3 - \numvars) / 2 + \varepsilon} \tau(\numrepresented) \prod_{p \mid 2 \det(A)} ( 1 - p^{-1/2} )^{-1} \numberthis \label{ineq:singseriesdyadictailbound3}
\end{align*}
for any $\varepsilon \in \RR$ satisfying $0 < \varepsilon < (\numvars - 3)/2$.

Applying \eqref{ineq:singseriesdyadictailbound3} to \eqref{ineq:singseriestailbound1}, we conclude that 
\eqref{eq:truncsingseries to singseries} holds 
for any $\varepsilon \in \RR$ satisfying $0 < \varepsilon < (\numvars - 3)/2$. 
Because $Q \ge 1$, we notice that \eqref{eq:truncsingseries to singseries} is true for all $\varepsilon > 0$, so we only require that $\varepsilon > 0$.
\end{proof}

Applying Lemma~\ref{lem:truncsingseries to singseries} to \eqref{eq:M\coordsize \numrepresented with singint and truncsingseries}, 
we find that 
\begin{align*}
M_{F, \Upspsi, \coordsize} (\numrepresented) 
	&= \mathfrak{S}_F (\numrepresented) J_{F, \Upspsi} (\numrepresented, \coordsize) \\
		&\qquad + O_\varepsilon \left( | J_{F, \Upspsi} (\numrepresented, \coordsize) | \lvl^{\numvars/2} Q^{(3 - \numvars) / 2 + \varepsilon} \tau(\numrepresented) \prod_{p \mid 2 \det(A)} ( 1 - p^{-1/2} )^{-1} \right) \\
		&\qquad + O_{\Upspsi , \numvars , \varepsilon} \left(\lvl^{\numvars/2} |\det(A)|^{-1/2} Q^{(\numvars - 1) / 2 + \varepsilon} \tau(\numrepresented) \prod_{p \mid 2 \det(A)} ( 1 - p^{-1/2} )^{-1} \right) 
\end{align*}
for any $\varepsilon > 0$.
We apply Corollary~\ref{cor:singint upperbound nonsing} to obtain for any $\varepsilon > 0$ that
\begin{align*}
M_{F, \Upspsi, \coordsize} (\numrepresented) 
	&= \mathfrak{S}_F (\numrepresented) J_{F, \Upspsi} (\numrepresented, \coordsize) \\
		&\quad + O_{\Upspsi , \numvars , \varepsilon} \left( \lvl^{\numvars/2} \frac{\coordsize^{\numvars-2}}{\Gamma(\numposevals/2) \left( \prod_{j=1}^\numposevals \lambda_j \right)^{1/2}} \left(\frac{\numrepresented}{\coordsize^2} - \frac{{\Upsuppvar_\Upspsi}^2}{2} \Ind{\numposevals > 1} \sum_{j=\numposevals+1}^\numvars \lambda_j \right)^{\numposevals/2 - 1} \right. \\
				&\qquad\qquad\qquad \left. \times Q^{(3 - \numvars) / 2 + \varepsilon} \tau(\numrepresented) \prod_{p \mid 2 \det(A)} ( 1 - p^{-1/2} )^{-1} \right) \\
		&\quad + O_{\Upspsi , \numvars , \varepsilon} \left(\lvl^{\numvars/2} |\det(A)|^{-1/2} Q^{(\numvars - 1) / 2 + \varepsilon} \tau(\numrepresented) \prod_{p \mid 2 \det(A)} ( 1 - p^{-1/2} )^{-1} \right) , \numberthis \label{eq:M\coordsize \numrepresented with singint and singseries}
\end{align*}
where $\lambda_1, \lambda_2, \ldots, \lambda_\numposevals$ are the positive eigenvalues of $A$ and $\lambda_{\numposevals+1}, \lambda_{\numposevals+2}, \ldots, \lambda_\numvars$ are the negative eigenvalues of $A$.

We will not further analyze the singular series in this paper. 
We do mention that the singular series is a product of $p$-adic densities, so the singular series is said to contain local information for the representation number. 
For more information about the singular series, please see Section~11.5 in \cite{IwaniecAutomorphicForms}.

\subsection{Analyzing the error term $E_{F, \Upspsi, \coordsize, 1} (\numrepresented)$}

We now analyze the error term $E_{F, \Upspsi, \coordsize, 1} (\numrepresented)$. We apply Lemma~\ref{lem:Trq\numrepresented x bound} and Theorem~\ref{thm:IFx\coordsize rq PSP bound} to \eqref{eq:E\coordsize1n} to obtain
\begin{align*}
E_{F, \Upspsi, \coordsize, 1} (\numrepresented) &\ll_{\Upspsi} \sum_{1 \le q \le Q} (\gcd(\lvl,q_0))^{\numvars/2} (\gcd(\numrepresented, q_1))^{1/2} q_0^{1/2} q^{(1 - \numvars) / 2} \tau(q) \log(2q) \\
		&\qquad \times \int_{0}^{\frac{1}{q (q + Q)}} \min\left\{ \coordsize^\numvars , |x|^{-\numvars/2} |\det(A)|^{-1/2} \right\} \sum_{\substack{\bvec{r} \in \ZZ^\numvars \\ 0 < \| \bvec{r} \| \le q \coordsize |x| \sigma_1 (\Upsuppvar_\Upspsi + 1) \sqrt{\numvars}}} 1 \ d x . \numberthis \label{ineq:E\coordsize1 bound1}
\end{align*}

Using Corollary~\ref{cor:latticeptsinpuncball upperbound}, we obtain 
\begin{align*}
\sum_{\substack{\bvec{r} \in \ZZ^\numvars \\ 0 < \| \bvec{r} \| \le q \coordsize |x| \sigma_1 (\Upsuppvar_\Upspsi + 1) \sqrt{\numvars}}} 1 &= \left| \{ \bvec{r} \in \ZZ^\numvars : 0 < \| \bvec{r} \| \le q \coordsize |x| \sigma_1 (\Upsuppvar_\Upspsi + 1) \sqrt{\numvars} \} \right| \\
	&\ll_\numvars (q \coordsize |x| \sigma_1 (\Upsuppvar_\Upspsi + 1) \sqrt{\numvars})^\numvars .
\end{align*}
Using this in \eqref{ineq:E\coordsize1 bound1}, we find that 
\begin{align*}
E_{F, \Upspsi, \coordsize, 1} (\numrepresented) &\ll_{\Upspsi , \numvars} (\coordsize \sigma_1 (\Upsuppvar_\Upspsi + 1) \sqrt{\numvars})^\numvars \sum_{1 \le q \le Q} (\gcd(\lvl,q_0))^{\numvars/2} (\gcd(\numrepresented, q_1))^{1/2} q_0^{1/2} q^{(\numvars+1)/2} \\
		&\qquad\qquad \times \tau(q) \log(2q) \int_{0}^{\frac{1}{q (q + Q)}} \min\left\{ \coordsize^\numvars |x|^\numvars , |x|^{\numvars/2} |\det(A)|^{-1/2} \right\} \ d x \\
	&\le (\coordsize \sigma_1 (\Upsuppvar_\Upspsi + 1) \sqrt{\numvars})^\numvars \lvl^{\numvars/2} \sum_{1 \le q \le Q} (\gcd(\numrepresented, q_1))^{1/2} q_0^{1/2} q^{(\numvars+1)/2} \tau(q) \log(2q) \\
		&\qquad\qquad \times |\det(A)|^{-1/2} \int_{0}^{\frac{1}{q (q + Q)}} |x|^{\numvars/2} \ d x \\
	&= (\coordsize \sigma_1 (\Upsuppvar_\Upspsi + 1) \sqrt{\numvars})^\numvars \lvl^{\numvars/2} \sum_{1 \le q \le Q} (\gcd(\numrepresented, q_1))^{1/2} q_0^{1/2} q^{-1/2} \tau(q) \log(2q) \\
		&\qquad\qquad \times |\det(A)|^{-1/2} \frac{1}{\frac{\numvars}{2} + 1} (q + Q)^{-\numvars/2 - 1} \\
	&= (\coordsize \sigma_1 (\Upsuppvar_\Upspsi + 1) \sqrt{\numvars})^\numvars \lvl^{\numvars/2} \sum_{1 \le q \le Q} (\gcd(\numrepresented, q_1))^{1/2} q_1^{-1/2} \tau(q) \log(2q) \\
		&\qquad\qquad \times |\det(A)|^{-1/2} \frac{1}{\frac{\numvars}{2} + 1} (q + Q)^{-\numvars/2 - 1} . \numberthis \label{ineq:E\coordsize1 bound2}
\end{align*}
Since $q > 0$, we know that $(q + Q)^{- \numvars/2 - 1} < Q^{- \numvars/2 - 1}$. Therefore, 
\begin{align*}
E_{F, \Upspsi, \coordsize, 1} (\numrepresented) &\ll_{\Upspsi , \numvars} (\coordsize \sigma_1 (\Upsuppvar_\Upspsi + 1))^\numvars \lvl^{\numvars/2} |\det(A)|^{-1/2} Q^{-\numvars/2 - 1} \\
		&\qquad\qquad \times \sum_{1 \le q \le Q} (\gcd(\numrepresented, q_1))^{1/2} q_1^{-1/2} \tau(q) \log(2q) . \numberthis \label{ineq:E\coordsize1 bound3}
\end{align*}
By applying Lemma~\ref{lem:q1gcdsum+eps} to \eqref{ineq:E\coordsize1 bound3}, we conclude that 
\begin{align*}
E_{F, \Upspsi, \coordsize, 1} (\numrepresented) &\ll_{\Upspsi , \numvars , \varepsilon} \coordsize^\numvars \sigma_1^\numvars \lvl^{\numvars/2} |\det(A)|^{-1/2} Q^{-\numvars/2 - 1/2 + \varepsilon} \tau(\numrepresented) \prod_{p \mid 2 \det(A)} ( 1 - p^{-1/2} )^{-1} \numberthis \label{ineq:E\coordsize1 bound4}
\end{align*}
for any $\varepsilon > 0$.

\subsection{Analyzing the error term $E_{F, \Upspsi, \coordsize, 2} (\numrepresented)$}

We now analyze the error term $E_{F, \Upspsi, \coordsize, 2} (\numrepresented)$. We apply Lemma~\ref{lem:Trq\numrepresented x bound} and Theorem~\ref{thm:IFx\coordsize rq PSP bound} to \eqref{eq:E\coordsize2n} to obtain
\begin{align*}
E_{F, \Upspsi, \coordsize, 2} (\numrepresented) &\ll_{\Upspsi} \sum_{1 \le q \le Q} (\gcd(\lvl,q_0))^{\numvars/2} (\gcd(\numrepresented, q_1))^{1/2} q_0^{1/2} q^{(1 - \numvars) / 2} \tau(q) \log(2q) \\
		&\qquad\qquad \times \int_{\frac{1}{q (q + Q)}}^{\frac{1}{q Q}} |x|^{-\numvars/2} |\det(A)|^{-1/2} \sum_{\substack{\bvec{r} \in \ZZ^\numvars \\ \| \bvec{r} \| \le q \coordsize |x| \sigma_1 (\Upsuppvar_\Upspsi + 1) \sqrt{\numvars}}} 1 \ d x . \numberthis \label{ineq:E\coordsize2 bound1}
\end{align*}

By Theorem~\ref{thm:latticeptsinball upperbound}, we know that 
\begin{align*}
\sum_{\substack{\bvec{r} \in \ZZ^\numvars \\ \| \bvec{r} \| \le q \coordsize |x| \sigma_1 (\Upsuppvar_\Upspsi + 1) \sqrt{\numvars}}} 1 \ d x &= \left| \{ \bvec{r} \in \ZZ^\numvars : \| \bvec{r} \| \le q \coordsize |x| \sigma_1 (\Upsuppvar_\Upspsi + 1) \sqrt{\numvars} \} \right| \\
	&\ll_\numvars (q \coordsize |x| \sigma_1 (\Upsuppvar_\Upspsi + 1) \sqrt{\numvars})^\numvars + 1 .
\end{align*}
Therefore,
\begin{align*}
&\int_{\frac{1}{q (q + Q)}}^{\frac{1}{q Q}} |x|^{-\numvars/2} \sum_{\substack{\bvec{r} \in \ZZ^\numvars \\ \| \bvec{r} \| \le q \coordsize |x| \sigma_1 (\Upsuppvar_\Upspsi + 1) \sqrt{\numvars}}} 1 \ d x \\
	&\ll_\numvars \int_{\frac{1}{q (q + Q)}}^{\frac{1}{q Q}} \left( (q \coordsize \sigma_1 (\Upsuppvar_\Upspsi + 1) \sqrt{\numvars})^\numvars |x|^{\numvars/2} + |x|^{-\numvars/2} \right) \ d x \\
	&= (q \coordsize \sigma_1 (\Upsuppvar_\Upspsi + 1) \sqrt{\numvars})^\numvars \frac{1}{\frac{\numvars}{2} + 1} \left( (q Q)^{-\numvars/2 - 1} - (q (q + Q))^{-\numvars/2 - 1} \right) \\
		&\qquad\qquad + \frac{1}{1 - \frac{\numvars}{2}} \left( (q Q)^{\numvars/2 - 1} - (q (q + Q))^{\numvars/2 - 1} \right) \\
	&\ll_\numvars q^{\numvars/2 - 1} \left( (\coordsize \sigma_1 (\Upsuppvar_\Upspsi + 1) )^\numvars Q^{-\numvars/2 - 1} + Q^{\numvars/2 - 1} \right) \numberthis \label{ineq:E\coordsize2 int bound}
\end{align*}
since $1 \le q \le Q$.

Substituting \eqref{ineq:E\coordsize2 int bound} into \eqref{ineq:E\coordsize2 bound1}, we find that
\begin{align*}
E_{F, \Upspsi, \coordsize, 2} (\numrepresented) &\ll_{\Upspsi , \numvars} \sum_{1 \le q \le Q} (\gcd(\lvl,q_0))^{\numvars/2} (\gcd(\numrepresented, q_1))^{1/2} q_0^{1/2} q^{-1/2} \tau(q) \log(2q) \\
		&\qquad\qquad \times |\det(A)|^{-1/2} \left( (\coordsize \sigma_1 (\Upsuppvar_\Upspsi + 1) )^\numvars Q^{-\numvars/2 - 1} + Q^{\numvars/2 - 1} \right) \\
	&\le \lvl^{\numvars/2} |\det(A)|^{-1/2} \left( \coordsize^\numvars \sigma_1^\numvars (\Upsuppvar_\Upspsi + 1)^\numvars Q^{-\numvars/2 - 1} + Q^{\numvars/2 - 1} \right) \\
		&\qquad \times \sum_{1 \le q \le Q} (\gcd(\numrepresented, q_1))^{1/2} q_1^{-1/2} \tau(q) \log(2q) . \numberthis \label{ineq:E\coordsize2 bound2}
\end{align*}

By applying Lemma~\ref{lem:q1gcdsum+eps} to \eqref{ineq:E\coordsize2 bound2}, we conclude that 
\begin{align*}
E_{F, \Upspsi, \coordsize, 2} (\numrepresented) &\ll_{\Upspsi , \numvars , \varepsilon} \lvl^{\numvars/2} |\det(A)|^{-1/2} \left( \coordsize^\numvars \sigma_1^\numvars Q^{-\numvars/2 - 1/2 + \varepsilon} + Q^{(\numvars - 1)/2 + \varepsilon} \right) \\
		&\qquad\qquad \times \tau(\numrepresented) \prod_{p \mid 2 \det(A)} ( 1 - p^{-1/2} )^{-1} . \numberthis \label{ineq:E\coordsize2 bound3}
\end{align*}
for any $\varepsilon > 0$.

\subsection{Analyzing the error term $E_{F, \Upspsi, \coordsize, 3} (\numrepresented)$} \label{sec:E\coordsize3}

We now analyze the error term $E_{F, \Upspsi, \coordsize, 3} (\numrepresented)$. Applying Lemma~\ref{lem:Trq\numrepresented x bound} and Corollary~\ref{cor:IFx\coordsize rq PNSP bound in terms of r} to \eqref{eq:E\coordsize3}, we obtain for any $M \ge 0$ that 
\begin{align*}
E_{F, \Upspsi, \coordsize, 3} (\numrepresented) &\ll_{M, \Upspsi} \sum_{1 \le q \le Q} (\gcd(\lvl,q_0))^{\numvars/2} (\gcd(\numrepresented, q_1))^{1/2} q_0^{1/2} q^{(1 - \numvars) / 2} \tau(q) \log(2q) \\
		&\qquad\qquad \times \int_{0}^{\frac{1}{q Q}} \sum_{\substack{\bvec{r} \in \ZZ^\numvars \\ \| \bvec{r} \| > q \coordsize |x| \sigma_1 (\Upsuppvar_\Upspsi + 1) \sqrt{\numvars}}} \coordsize^{\numvars-M} \left( \frac{1}{q (\Upsuppvar_\Upspsi + 1) \sqrt{\numvars}} \| \bvec{r} \| \right)^{-M} \ d x . \numberthis \label{ineq:E\coordsize3 bound1}
\end{align*}
In order to apply Corollary~\ref{cor:latticesumbound B>=0}, 
we suppose that $M \ge \numvars + 1$. 
Applying Corollary~\ref{cor:latticesumbound B>=0} to \eqref{ineq:E\coordsize3 bound1}, we obtain
\begin{align*}
E_{F, \Upspsi, \coordsize, 3} (\numrepresented) &\ll_{M, \Upspsi , \numvars} \sum_{1 \le q \le Q} (\gcd(\lvl,q_0))^{\numvars/2} (\gcd(\numrepresented, q_1))^{1/2} q_0^{1/2} q^{(1 - \numvars) / 2} \tau(q) \log(2q) \\
		&\qquad\qquad \times \coordsize^{\numvars-M} (q (\Upsuppvar_\Upspsi + 1) \sqrt{\numvars})^M \int_{0}^{\frac{1}{q Q}} 1 \ d x . \\
	&= \coordsize^{\numvars-M} (\Upsuppvar_\Upspsi + 1)^M \numvars^{M/2} Q^{-1} \\
		&\qquad \times \sum_{1 \le q \le Q} (\gcd(\lvl,q_0))^{\numvars/2} (\gcd(\numrepresented, q_1))^{1/2} q_0^{1/2} q^{M - 1/2 - \numvars/2} \tau(q) \log(2q) \\
	&\le \coordsize^{\numvars-M} (\Upsuppvar_\Upspsi + 1)^M \numvars^{M/2} Q^{M - \numvars/2 - 1} \lvl^{\numvars/2} \\
		&\qquad \times \sum_{1 \le q \le Q} (\gcd(\numrepresented, q_1))^{1/2} q_1^{-1/2} \tau(q) \log(2q) . \numberthis \label{ineq:E\coordsize3 bound2}
\end{align*}
We apply Lemma~\ref{lem:q1gcdsum+eps} to obtain
\begin{align}
E_{F, \Upspsi, \coordsize, 3} (\numrepresented) &\ll_{M, \Upspsi , \numvars , \varepsilon} \coordsize^{\numvars-M} Q^{M - \numvars/2 - 1/2 + \varepsilon} \lvl^{\numvars/2} \tau(\numrepresented) \prod_{p \mid 2 \det(A)} ( 1 - p^{-1/2} )^{-1} . \label{ineq:E\coordsize3 bound3}
\end{align}
for any $\varepsilon > 0$.

\subsection{Choosing parameters} \label{sec:parameters}

By putting together estimates from previous sections and by choosing some parameters in this subsection, we prove Theorem~\ref{thm:R\coordsize\numrepresented asymp1} and Corollary~\ref{cor:R\coordsize\numrepresented asymp2}. 
We substitute \eqref{eq:M\coordsize \numrepresented with singint and singseries}, \eqref{ineq:E\coordsize1 bound4}, \eqref{ineq:E\coordsize2 bound3}, and \eqref{ineq:E\coordsize3 bound3} into \eqref{eq:R\coordsize \numrepresented split} to obtain 
\begin{align*}
\repnum_{F, \Upspsi, \coordsize} (\numrepresented) &= \mathfrak{S}_F (\numrepresented) J_{F, \Upspsi} (\numrepresented, \coordsize) \\ 
		&\quad + O_{\Upspsi , \numvars , \varepsilon} \left( \lvl^{\numvars/2} \frac{\coordsize^{\numvars-2}}{\Gamma(\numposevals/2) \left( \prod_{j=1}^\numposevals \lambda_j \right)^{1/2}} \left(\frac{\numrepresented}{\coordsize^2} - \frac{{\Upsuppvar_\Upspsi}^2}{2} \Ind{\numposevals > 1} \sum_{j=\numposevals+1}^\numvars \lambda_j \right)^{\numposevals/2 - 1} \right. \\
				&\qquad\qquad\qquad \left. \times Q^{(3 - \numvars) / 2 + \varepsilon} \tau(\numrepresented) \prod_{p \mid 2 \det(A)} ( 1 - p^{-1/2} )^{-1} \right) \\
		&\quad + O_{\Upspsi , \numvars , \varepsilon} \left( \lvl^{\numvars/2} |\det(A)|^{-1/2} Q^{(\numvars - 1) / 2 + \varepsilon} \tau(\numrepresented) \prod_{p \mid 2 \det(A)} ( 1 - p^{-1/2} )^{-1} \right) \\
		&\quad + O_{\Upspsi , \numvars , \varepsilon} \left( \coordsize^\numvars \sigma_1^\numvars \lvl^{\numvars/2} |\det(A)|^{-1/2} Q^{-\numvars/2 - 1/2 + \varepsilon} \tau(\numrepresented) \prod_{p \mid 2 \det(A)} ( 1 - p^{-1/2} )^{-1} \right) \\
		&\quad + O_{\Upspsi , \numvars , \varepsilon} \left( \lvl^{\numvars/2} |\det(A)|^{-1/2} \left( \coordsize^\numvars \sigma_1^\numvars Q^{-\numvars/2 - 1/2 + \varepsilon} + Q^{(\numvars - 1)/2 + \varepsilon} \right) \right. \\
			&\qquad\qquad\qquad \left. \times \tau(\numrepresented) \prod_{p \mid 2 \det(A)} ( 1 - p^{-1/2} )^{-1} \right) \\
		&\quad + O_{M, \Upspsi , \numvars , \varepsilon} \left( \coordsize^{\numvars-M} Q^{M - \numvars/2 - 1/2 + \varepsilon} \lvl^{\numvars/2} \tau(\numrepresented) \prod_{p \mid 2 \det(A)} ( 1 - p^{-1/2} )^{-1} \right) \numberthis \label{eq:R\coordsize \numrepresented with error1}
\end{align*}
for any $\varepsilon > 0$.

We would like to choose $Q$ and $\coordsize$ that somewhat balance all of the error terms in \eqref{eq:R\coordsize \numrepresented with error1}. To do this, we find a value of $Q$ that practically minimizes the right-hand side of \eqref{ineq:E\coordsize2 bound3}, which is an upper bound for the absolute value of $E_{F, \Upspsi, \coordsize, 2} (\numrepresented)$. By solving 
\begin{align*}
\coordsize^\numvars \sigma_1^\numvars Q^{-\numvars/2 - 1/2 + \varepsilon} = Q^{(\numvars - 1)/2 + \varepsilon}
\end{align*}
for $Q$, we find that 
\begin{align*}
Q = \sigma_1 \coordsize
\end{align*}
will give a bound within a factor of $2$ of the optimal bound for the right-hand side of \eqref{ineq:E\coordsize2 bound3}.
(See the commentary after Theorem~2.3 in \cite{MontgomeryVaughanMultNT} for an explanation as to why this choice of $Q$ will give a bound within a factor of $2$ of the optimal bound for the right-hand side of \eqref{ineq:E\coordsize2 bound3}.
We also note that $\sigma_1 \coordsize$ is not necessarily an integer, which is why we purposefully allowed $Q$ to be not an integer in Subsection~\ref{sec:Fareydissect}.)

In order to set $Q = \sigma_1 \coordsize$, we need to place an additional restriction on $\coordsize$. Recall that $Q \ge 1$. 
Thus, to guarantee that $Q = \sigma_1 \coordsize \ge 1$, we require that $\coordsize \ge 1 / \sigma_1$.

By setting $Q = \sigma_1 \coordsize$ in \eqref{eq:R\coordsize \numrepresented with error1}, we obtain 
\begin{align*}
&\repnum_{F, \Upspsi, \coordsize} (\numrepresented) \\
	&= \mathfrak{S}_F (\numrepresented) J_{F, \Upspsi} (\numrepresented, \coordsize) \\
		&\qquad + O_{\Upspsi , \numvars , \varepsilon} \left( \frac{\lvl^{\numvars/2} \coordsize^{(\numvars - 1)/2 + \varepsilon} \sigma_1^{(3 - \numvars )/2 + \varepsilon}}{\Gamma(\numposevals/2) \left( \prod_{j=1}^\numposevals \lambda_j \right)^{1/2}} \left(\frac{\numrepresented}{\coordsize^2} - \frac{{\Upsuppvar_\Upspsi}^2}{2} \Ind{\numposevals > 1} \sum_{j=\numposevals+1}^\numvars \lambda_j \right)^{\numposevals/2 - 1} \right. \\
				&\qquad\qquad\qquad\qquad \left. \times \tau(\numrepresented) \prod_{p \mid 2 \det(A)} ( 1 - p^{-1/2} )^{-1} \right) \\
		&\qquad + O_{\Upspsi , \numvars , \varepsilon} \left( \coordsize^{(\numvars - 1)/2 + \varepsilon} \sigma_1^{(\numvars - 1)/2 + \varepsilon} \lvl^{\numvars/2} |\det(A)|^{-1/2} \tau(\numrepresented) \prod_{p \mid 2 \det(A)} ( 1 - p^{-1/2} )^{-1} \right) \\
		&\qquad + O_{M, \Upspsi , \numvars , \varepsilon} \left( \coordsize^{(\numvars - 1) /2 + \varepsilon} \sigma_1^{M - \numvars/2 - 1/2 + \varepsilon} \lvl^{\numvars/2} \tau(\numrepresented) \prod_{p \mid 2 \det(A)} ( 1 - p^{-1/2} )^{-1} \right) \numberthis \label{eq:R\coordsize \numrepresented with error2}
\end{align*}
for any $\varepsilon > 0$.

We notice that the only place that $M$ occurs in \eqref{eq:R\coordsize \numrepresented with error2} is as an exponent for $\sigma_1$. 
Now because $F$ is a nonsingular integral quadratic form, the absolute value of the determinant of $A$ is at least one, so at least one of the singular values of $A$ must be at least one. Since $\sigma_1$ is the largest singular value of $A$, we must have $\sigma_1 \ge 1$. 
Therefore, 
we want to set $M$ to be as small as possible. The only condition that we have on $M$ is that $M \ge \numvars + 1$, so we choose to set $M = \numvars + 1$. With this choice of $M$, we find that \eqref{eq:R\coordsize \numrepresented with error2} becomes
\begin{align*}
&\repnum_{F, \Upspsi, \coordsize} (\numrepresented) \\
	&= \mathfrak{S}_F (\numrepresented) J_{F, \Upspsi} (\numrepresented, \coordsize) \\
		&\qquad + O_{\Upspsi , \numvars , \varepsilon} \left( \frac{\lvl^{\numvars/2} \coordsize^{(\numvars - 1)/2 + \varepsilon} \sigma_1^{(3 - \numvars )/2 + \varepsilon}}{\Gamma(\numposevals/2) \left( \prod_{j=1}^\numposevals \lambda_j \right)^{1/2}} \left(\frac{\numrepresented}{\coordsize^2} - \frac{{\Upsuppvar_\Upspsi}^2}{2} \Ind{\numposevals > 1} \sum_{j=\numposevals+1}^\numvars \lambda_j \right)^{\numposevals/2 - 1} \right. \\
				&\qquad\qquad\qquad\qquad \left. \times \tau(\numrepresented) \prod_{p \mid 2 \det(A)} ( 1 - p^{-1/2} )^{-1} \right) \\
		&\qquad + O_{\Upspsi , \numvars , \varepsilon} \left( \coordsize^{(\numvars - 1)/2 + \varepsilon} \sigma_1^{(\numvars - 1)/2 + \varepsilon} \lvl^{\numvars/2} |\det(A)|^{-1/2} \tau(\numrepresented) \prod_{p \mid 2 \det(A)} ( 1 - p^{-1/2} )^{-1} \right) \\
		&\qquad + O_{\Upspsi , \numvars , \varepsilon} \left( \coordsize^{(\numvars - 1) /2 + \varepsilon} \sigma_1^{(\numvars + 1)/2 + \varepsilon} \lvl^{\numvars/2} \tau(\numrepresented) \prod_{p \mid 2 \det(A)} ( 1 - p^{-1/2} )^{-1} \right) \numberthis \label{eq:R\coordsize \numrepresented with error3}
\end{align*}
for any $\varepsilon > 0$. 

Recall that $|\det(A)| \ge 1$ and $\sigma_1 \ge 1$, so $|\det(A)|^{-1/2} \le \sigma_1$. Therefore, 
\begin{align*}
\sigma_1^{(\numvars - 1)/2 + \varepsilon} |\det(A)|^{-1/2} \le \sigma_1^{(\numvars + 1)/2 + \varepsilon} 
\end{align*}
for any $\varepsilon > 0$.
Thus, we conclude that 
\begin{align*}
&\repnum_{F, \Upspsi, \coordsize} (\numrepresented) \\
	&= \mathfrak{S}_F (\numrepresented) J_{F, \Upspsi} (\numrepresented, \coordsize) \\
		&\qquad + O_{\Upspsi , \numvars , \varepsilon} \left( \frac{\lvl^{\numvars/2} \coordsize^{(\numvars - 1)/2 + \varepsilon} \sigma_1^{(3 - \numvars )/2 + \varepsilon}}{\Gamma(\numposevals/2) \left( \prod_{j=1}^\numposevals \lambda_j \right)^{1/2}} \left(\frac{\numrepresented}{\coordsize^2} - \frac{{\Upsuppvar_\Upspsi}^2}{2} \Ind{\numposevals > 1} \sum_{j=\numposevals+1}^\numvars \lambda_j \right)^{\numposevals/2 - 1} \right. \\
				&\qquad\qquad\qquad\qquad \left. \times \tau(\numrepresented) \prod_{p \mid 2 \det(A)} ( 1 - p^{-1/2} )^{-1} \right) \\
		&\qquad + O_{\Upspsi , \numvars , \varepsilon} \left( \coordsize^{(\numvars - 1) /2 + \varepsilon} \sigma_1^{(\numvars + 1)/2 + \varepsilon} \lvl^{\numvars/2} \tau(\numrepresented) \prod_{p \mid 2 \det(A)} ( 1 - p^{-1/2} )^{-1} \right) \numberthis \label{eq:R\coordsize \numrepresented with error4}
\end{align*}
for any $\varepsilon > 0$. 
Because of \eqref{eq:singintnormalized with density} and Theorem~\ref{thm:archmediandensitytovol}, we deduce Theorem~\ref{thm:R\coordsize\numrepresented asymp1} from \eqref{eq:R\coordsize \numrepresented with error4}.

We now want to choose $\coordsize$ to minimize the error terms in \eqref{eq:R\coordsize \numrepresented with error4}. To make both error terms grow at the same rate in terms of $\numrepresented$, we set $\coordsize = \numrepresented^{1/2}$ in Theorem~\ref{thm:R\coordsize\numrepresented asymp1} to obtain 
\begin{align*}
&\repnum_{F, \Upspsi, \coordsize} (\numrepresented) \\
		&= \mathfrak{S}_F \!\left( \numrepresented \right) \sigma_{F, \Upspsi, \infty} \!\left( \numrepresented, \numrepresented^{1/2} \right) \numrepresented^{\numvars/2 - 1} \\
		&\qquad + O_{\Upspsi , \numvars , \varepsilon} \left( \left( \sigma_1^{(\numvars + 1)/2 + \varepsilon} + \frac{\sigma_1^{(3 - \numvars )/2 + \varepsilon}}{\Gamma(\numposevals/2) \left( \prod_{j=1}^\numposevals \lambda_j \right)^{1/2}} \left( 1 - \frac{{\Upsuppvar_\Upspsi}^2}{2} \Ind{\numposevals > 1} \sum_{j=\numposevals+1}^\numvars \lambda_j \right)^{\numposevals/2 - 1} \right) \right. \\
		&\qquad\qquad\qquad\qquad \left. \times \numrepresented^{(\numvars - 1) /4 + \varepsilon/2} \tau(\numrepresented) \lvl^{\numvars/2} \prod_{p \mid 2 \det(A)} ( 1 - p^{-1/2} )^{-1}  \right)  \numberthis \label{eq:R\coordsize \numrepresented with error5}
\end{align*}
for any $\varepsilon > 0$. 
Because $\tau(\numrepresented) \ll_{\varepsilon} \numrepresented^{\varepsilon/2}$ for any $\varepsilon > 0$, we conclude Corollary~\ref{cor:R\coordsize\numrepresented asymp2} from \eqref{eq:R\coordsize \numrepresented with error5}.

\begin{rmk}
The implied constant in Corollary~\ref{cor:R\coordsize\numrepresented asymp2} can be further optimized if $F$ is a positive definite quadratic form. (See Corollary~1.4 in \cite{EdnaJonesThesis}.)
\end{rmk}

\subsection{Proof of Corollary~\ref{cor:repnum F asymp}} \label{sec:proofcorrepnum}

In this subsection, we complete a proof of Corollary~\ref{cor:repnum F asymp}, in which we assume that $F$ is positive definite. 
To prove Corollary~\ref{cor:repnum F asymp}, we choose a bump function $\Upspsi$ that satisfies certain conditions. Notice that 
\begin{align*}
| \{ \bvec{m} \in \ZZ^\numvars : F(\bvec{m}) = \numrepresented \} | &= \sum_{\bvec{m} \in \ZZ^\numvars} \Ind{F(\bvec{m}) = \numrepresented} .
\end{align*}
Because 
\begin{align*}
\repnum_{F, \Upspsi, \coordsize} (\numrepresented) = \sum_{\bvec{m} \in \ZZ^\numvars} \Ind{F(\bvec{m}) = \numrepresented} \Upspsi_\coordsize (\bvec{m}) ,
\end{align*}
we find that 
\begin{align}
| \{ \bvec{m} \in \ZZ^\numvars : F(\bvec{m}) = \numrepresented \} | &=  \repnum_{F, \Upspsi, \coordsize} (\numrepresented) \label{eq:repnum to weightedrepnum}
\end{align}
if $\Upspsi_\coordsize (\bvec{m}) = 1$ for each $\bvec{m} \in \ZZ^\numvars$ satisfying $F(\bvec{m}) = \numrepresented$. Because scaling by $\coordsize > 0$ does not necessarily preserve integrality of solutions to $F(\bvec{m}) = \numrepresented$, we require that $\Upspsi_\coordsize (\bvec{x}) = 1$ whenever $\bvec{x} \in \RR^\numvars$ satisfies $F(\bvec{x}) = \numrepresented$. Using 
the invertible mapping $\bvec{x} \mapsto \coordsize \bvec{m}$ and 
the fact that 
\begin{align*}
\Upspsi_\coordsize (\bvec{m}) = \Upspsi\left( \frac{1}{\coordsize} \bvec{m} \right) ,
\end{align*}
we deduce that $\Upspsi(\bvec{m}) = 1$ whenever $\bvec{m} \in \RR^\numvars$ satisfies $F(\bvec{m}) = \numrepresented / \coordsize^2$ if and only if $\Upspsi_\coordsize (\bvec{x}) = 1$ whenever $\bvec{x} \in \RR^\numvars$ satisfies $F(\bvec{x}) = \numrepresented$.

To remove any dependence of $\Upspsi$ on $\numrepresented$ or $\coordsize$, we set 
\begin{align}
\coordsize = \numrepresented^{1/2} \label{eq:\coordsize=sqrt\numrepresented}
\end{align}
so that the only condition on $\Upspsi \in C_c^\infty (\RR^\numvars)$ is that $\Upspsi(\bvec{m}) = 1$ whenever $\bvec{m} \in \RR^\numvars$ satisfies $F(\bvec{m}) = 1$. We note that it is possible for $\Upspsi$ to satisfy this condition because of Theorem~\ref{thm:preimageFcompact}.

Observe that $\Upspsi$ now depends on $F$, so we now consider any implied constants dependent on $\Upspsi$ to now be dependent on $\Upspsi$ and $F$.
Therefore, by applying \eqref{eq:\coordsize=sqrt\numrepresented} to \eqref{eq:R\coordsize \numrepresented with error4}, we obtain 
\begin{align}
\repnum_{F, \Upspsi, \coordsize} (\numrepresented) &= \mathfrak{S}_F (\numrepresented) J_{F, \Upspsi} (\numrepresented, \coordsize) + O_{F , \Upspsi , \numvars , \varepsilon} \left( \numrepresented^{(\numvars - 1) /4 + \varepsilon/2} \right) \label{eq:R\coordsize \numrepresented with error6} 
\end{align}
for any $\varepsilon > 0$. 
Because $\Upspsi(\bvec{m}) = 1$ whenever $\bvec{m} \in \RR^\numvars$ satisfies $F(\bvec{m}) = \numrepresented / \coordsize^2$, we deduce that Corollary~\ref{cor:singint asymp Upsilon=1 eval}, \eqref{eq:repnum to weightedrepnum}, and \eqref{eq:R\coordsize \numrepresented with error6} imply that 
\begin{align}
| \{ \bvec{m} \in \ZZ^\numvars : F(\bvec{m}) = \numrepresented \} | &= \mathfrak{S}_F (\numrepresented) \frac{(2 \pi)^{\numvars/2}}{\Gamma(\numvars/2) \sqrt{\det(A)}} \numrepresented^{\numvars/2 - 1} + O_{F , \Upspsi , \numvars , \varepsilon} \left( \numrepresented^{(\numvars - 1) /4 + \varepsilon/2} \right) \label{eq:R\coordsize \numrepresented with error7}
\end{align}
for any $\varepsilon > 0$. 
Because $| \{ \bvec{m} \in \ZZ^\numvars : F(\bvec{m}) = \numrepresented \} |$ is a number that does not depend on $\Upspsi$, we can remove the dependency on $\Upspsi$ for the implied constant in \eqref{eq:R\coordsize \numrepresented with error7} and obtain 
\begin{align}
| \{ \bvec{m} \in \ZZ^\numvars : F(\bvec{m}) = \numrepresented \} | &= \mathfrak{S}_F (\numrepresented) \frac{(2 \pi)^{\numvars/2}}{\Gamma(\numvars/2) \sqrt{\det(A)}} \numrepresented^{\numvars/2 - 1} + O_{F , \numvars , \varepsilon} \left( \numrepresented^{(\numvars - 1) /4 + \varepsilon/2} \right) \label{eq:R\coordsize \numrepresented with error8}
\end{align}
for any $\varepsilon > 0$. 
By replacing $\varepsilon/2$ with $\varepsilon$ in \eqref{eq:R\coordsize \numrepresented with error8}, we conclude Corollary~\ref{cor:repnum F asymp}.

\section{A strong asymptotic local-global principle for certain Kleinian sphere packings} \label{chapter:Local-GlobalSpherePackings}

In this paper, we have developed a version of the Kloosterman circle method with a bump function. 
A potential application of this version of the Kloosterman circle method is a proof of a strong asymptotic local-global principle for certain integral Kleinian sphere packings. 

To state a conjectured strong asymptotic local-global principle for certain integral Kleinian sphere packings, we first need to state some definitions. 
The \emph{bend} of a $(\dimnum-1)$-sphere is the reciprocal of the radius of the $(\dimnum-1)$-sphere. 
An $(\dimnum-1)$-sphere packing is called \emph{integral} if the bend of each $(\dimnum-1)$-sphere in the packing is an integer. 
An integral $(\dimnum-1)$-sphere packing is called \emph{primitive} if the greatest common divisor of all of the bends in the packing is~$1$. 
An $(\dimnum-1)$-sphere packing is \emph{Kleinian} if its limit set is that of a geometrically finite group $\Gamma$ of isometries of $(\dimnum+1)$-dimensional hyperbolic space. 
Kontorovich and Nakamura~\cite{KNCrystallographicSpherePackings} and Kontorovich and Kapovich~\cite{KKKleinianSoherePackingsBugsArithGroups} proved that there are infinitely many conformally inequivalent integral Kleinian sphere packings.

There may be local or congruence restrictions on the bends in a fixed Kleinian sphere packing. 
This motivates the following definition of admissibility. 
\begin{defn}[Admissible integers for sphere packings]
Let $\PP$ be an integral Kleinian sphere packing. 
An integer $m$ is \emph{admissible} (or \emph{locally represented}) if, for every $q \ge 1$, we have 
\begin{align*}
m \equiv \text{bend of some sphere in }\PP \pmod{q} .
\end{align*}
\end{defn}

Before we state a conjectured strong asymptotic local-global principle for certain Kleinian sphere packings, we briefly discuss orientation-preserving isometries of $(\dimnum+1)$-dimensional hyperbolic space. 
The group of orientation-preserving isometries of $(\dimnum+1)$-dimensional hyperbolic space can be identified with the group of orientation-preserving M\"obius transformations acting on $\RR^\dimnum \cup \{ \infty \}$. These groups can be identified with a certain group of $2 \times 2$ matrices with entries in a Clifford algebra. (See \cite{VahlenMobiusClifford}, \cite{AhlforsMobiusClifford}, or \cite{WatermanMobiusSeveralDims} for how this can be done.) This matrix group contains $\PSL_2 (\CC)$ if $\dimnum \ge 2$.

The following is a conjectured strong asymptotic local-global principle for certain Kleinian sphere packings.
\begin{conj} \label{conj:LocalGlobalKleinianSpherePackings}
Let $\PP$ be a primitive integral Kleinian $(\dimnum-1)$-sphere packing in $\RR^\dimnum \cup \{ \infty \}$ with an orientation-preserving automorphism group $\Gamma$ of M\"obius transformations.
\begin{enumerate}
\item Suppose that there exists a $(\dimnum-1)$-sphere $S_0 \in\PP$ such that the stabilizer of $S_0$ in $\Gamma$ contains (up to conjugacy) a congruence subgroup of $\PSL_2 (\OO_K)$, where $K$ is an imaginary quadratic field and $\OO_K$ is the ring of integers of $K$. This condition implies that $\dimnum \ge 3$.
\item Suppose that there is a $(\dimnum-1)$-sphere $S_1 \in \PP$ that is tangent to $S_0$. \label{cond:tangencycond}
\end{enumerate}
Then every sufficiently large admissible integer is a bend of a $(\dimnum-1)$-sphere in $\PP$. That is, there exists an $N_0 = N_0 (\PP)$ such that if $m$ is admissible and $m > N_0$, then $m$ is the bend of a $(\dimnum-1)$-sphere in $\PP$.
\end{conj}

There has been work towards Conjecture~\ref{conj:LocalGlobalKleinianSpherePackings}. 
Examples of integral Kleinian sphere packings are integral Soddy sphere packings and integral orthoplicial Apollonian sphere packings. 
Kontorovich~\cite{KontorovichLocal-GlobalSoddy} proved the strong asymptotic local-global principle for integral Soddy sphere packings. 
Dias~\cite{DiasLocal-GlobalOrthoplicial} and Nakamura~\cite{NakamuraLocal-GlobalOrthoplicial} independently did work towards proving the strong asymptotic local-global principle for integral orthoplicial Apollonian sphere packings.

When this was written, the author did not know of a proof of a strong asymptotic local-global principle that applied to multiple conformally inequivalent integral Kleinian sphere packings. The author is making progress towards proving Conjecture~\ref{conj:LocalGlobalKleinianSpherePackings}, which would apply to multiple conformally inequivalent integral Kleinian sphere packings. 
By using M\"obius transformations on $\RR^\dimnum \cup \{ \infty \}$ and inversive coordinates of $(\dimnum-1)$-spheres, one can obtain a family of integral quadratic polynomials in four variables with a coprimality condition on the variables.
Potentially, the version of the Kloosterman circle method discussed in this paper could be then used to prove a result towards Conjecture~\ref{conj:LocalGlobalKleinianSpherePackings} that applies to multiple conformally inequivalent integral Kleinian sphere packings.

\section*{Acknowledgments}

I thank Alex Kontorovich, Gene Kopp, Louis Gaudet, Henryk Iwaniec, and Sam Pushon-Smith for helpful mathematical discussions. 
I thank the anonymous referee for the excellent suggestions that improved this article. 
I thank Alex Kontorovich for giving me the problem of trying to prove a strong asymptotic local-global principle for certain Kleinian sphere packings. 
I thank Trevor Wooley for several helpful conversations with him about the circle method and mathematical writing. I also thank Trevor Wooley for sharing his arithmetic harmonic analysis notes. 

Some of this material is based upon work supported by the National Science Foundation under Grant No.~2402599.

\bibliographystyle{amsalpha}
\bibliography{KloostermanCircleMethodBib}
\end{document}